\documentclass[12pt]{article}

\usepackage{amsmath,amssymb,amsfonts, amsthm}
\usepackage{amscd}
\usepackage{bbm,bm} 
\usepackage{cases}
\usepackage[mathscr]{euscript}

\usepackage{comment}

\usepackage{breakcites}
\usepackage{natbib}
\RequirePackage[colorlinks,citecolor={blue!60!black},urlcolor={blue!70!black},linkcolor={red!60!black},breaklinks,hypertexnames=false]{hyperref}

\usepackage[ruled,vlined]{algorithm2e}
\usepackage{array}
\usepackage[small]{caption}
\usepackage{epic,eepic, epsfig, psfrag}
\usepackage{enumerate}
\usepackage{fullpage}
\usepackage{graphicx}
\usepackage[sf,bf,SF,footnotesize]{subfigure}
\usepackage{setspace}
\usepackage{tikz}
\usepackage{xcolor}

\usepackage{etoolbox}

\newtheorem{theorem}{Theorem}

\newtheorem{lemma}{Lemma}

\usepackage{extraCommands}

\title{An adaptive transfer learning perspective on classification in non-stationary environments}

\author{Henry W J Reeve\\University of Bristol\\henry.reeve@bristol.ac.uk}

\begin{document}

\maketitle

\begin{abstract} 
We consider a semi-supervised classification problem with non-stationary label-shift in which we observe a labelled data set followed by a sequence of unlabelled covariate vectors in which the marginal probabilities of the class labels may change over time. Our objective is to predict the corresponding class-label for each covariate vector, without ever observing the ground-truth labels, beyond the initial labelled data set. Previous work has demonstrated the potential of sophisticated variants of online gradient descent to perform competitively with the optimal dynamic strategy (Bai et al. 2022). In this work we explore an alternative approach grounded in statistical methods for adaptive transfer learning. We demonstrate the merits of this alternative methodology by establishing a high-probability regret bound on the test error at any given individual test-time, which adapt automatically to the unknown dynamics of the marginal label probabilities. Further more, we give bounds on the average dynamic regret which match the average guarantees of the online learning perspective for any given time interval.
\end{abstract}

\newcommand{\sourceIndicator}{{\mathrm{sc}}}
\newcommand{\targetIndicator}{{\mathrm{tg}}}

\newcommand{\sourceDistribution}{{P_\sourceIndicator}}
\newcommand{\targetDistribution}{{P_\targetIndicator}}

\section{Introduction}\label{sec:introduction}

A longstanding paradigm in theoretical statistics has been to focus on settings in which the practitioner has access to a sample consisting of observations drawn from a distribution $P$, and the practitioner's objective is to estimate a parameter of the very same distribution $P$, where ``parameter'' here is to be construed broadly to include any property of the distribution, from a real-valued parameter such as a population mean, through to an infinite-dimensional parameter such as a Bayes optimal classifier. Whilst such settings, often provide a natural model, in many real-world problems we wish to leverage data from source distribution $\sourceDistribution$ which is related, but distinct from the target distribution $\targetDistribution$ corresponding to the parameter of interest. Transfer learning refers to the wide array of methods developed for transferring knowledge from these related source data to a target domain \citep{zhuang2020comprehensive}. Applications range from medical settings, where a practitioner may seek to combine information from new and historical data to estimate the efficacy of a new treatment or vaccine \citep{aloui2023transfer}, through to natural language processing where practitioners frequently utilise vast corpora of unlabelled text data to pre-train language models, before fine-tuning for a specific task of interest \citep{devlin-etal-2019-bert}. 

In light of the impressive empirical successes of transfer learning, the topic has recently attracted a great deal of interest from the statistical and theoretical machine learning communities. There has been a strong focus on the canonical problem of binary classification where both the source $\sourceDistribution$ and the target $\targetDistribution$ are distributions over a space $\metricSpace\times \{0,1\}$ consisting tuples $(x,y)$, where $x \in \metricSpace$ is a covariate-vector and $y \in \{0,1\}$ is a binary label. The  objective is to identify an optimal rule $\classifier:\metricSpace \rightarrow \{0,1\}$ for which 
minimises the probability of an error $\classifier(X^\targetIndicator) \neq Y^\targetIndicator$ for $(X^\targetIndicator,Y^\targetIndicator)\sim \targetDistribution$, whilst the large majority of the data is distributed according to $\targetDistribution$.

One important line of work on binary classification considers the covariate-shift setting \citep{sugiyama2005generalization,gretton2009covariate,candela2009dataset, sugiyama2012machine,sugiyama2012density,kpotufe2018marginal,kpotufe2021marginal} in which conditional probabilities of the labels given the feature vector coincide between the source and target distributions, i.e. $\Prob(Y^{\sourceIndicator}|X^{\sourceIndicator}=x)=\Prob(Y^{\targetIndicator}|X^{\targetIndicator}=x)$ for $(X^{\sourceIndicator},Y^{\sourceIndicator})\sim \sourceDistribution$ and $(X^{\targetIndicator},Y^{\targetIndicator})\sim \targetDistribution$. Crucially, the Bayes optimal classification rule $\classifier$ is a function of this conditional probability $x\mapsto \Prob(Y^{\targetIndicator}|X^{\targetIndicator}=x)$. As such, the covariate shift setting allows for consistent classification with respect to the target $\targetDistribution$ given only data from the source $\sourceDistribution$, provided that the marginal distribution over the target covariates $X^{\targetIndicator}$ is dominated by the corresponding source marginal over $X^{\sourceIndicator}$. The minimax optimal rate for the covariate shift regime depends upon the local scaling behaviour of the ratio between the source and target marginals, characterised by the transfer exponent \citep{kpotufe2021marginal}, in conjunction with other distributional properties and the sample sizes from both $\sourceDistribution$ and $\targetDistribution$. More broadly, consistent classification with respect to a target distribution $\targetDistribution$ using only data from a distinct source distribution $\sourceDistribution$ is also possible provided that the conditional probability of a label for the two distributions are sufficiently similar that the Bayes optimal decision boundaries coincide \citep{menon2018learning,cai2019transfer,cannings2020}. Consistent classification with source data alone is also possible in certain label-noise settings provided that conditional probability of a label for the two distributions are related by an identifiable affine transformation
\citep{blanchard2017domain,
reeve2019classification,reeve2019fastLabelNoise,scott2018generalized,scott2019learning}.

Learning a classification rule who's performance approaches the Bayes error inevitably necessitates distributional assumptions. An alternative approach is the Vapnik-Chervonenkis paradigm which instead places constraints on the family of potential classifiers
\citep{vapnik1974theory}.  Interestingly, if we insist that $\classifier$ be chosen from a low capacity class, then even with the covariate-shift assumption, source data alone is no longer sufficient for effective learning \citep{ben2010impossibility}. Indeed, for certain families of classifier, learning an optimal model is essentially equivalent to learning the entire distribution \citep{hopkins2023pac}. As such, transfer learning within the Vapnik-Chervonenkis paradigm has focused on the role of the discrepancy measures between $\sourceDistribution$ and $\targetDistribution$ which typically depend upon the collection of candidate models \citep{ben2010theory,ben2010impossibility,mansour2009domain,mohri2012new,yang2013theory,
cortes2019adaptation}. A generalised notion of transfer exponent extends to Vapnik-Chervonenkis paradigm by bounding the excess error on the target distribution in terms of the excess error on the source distribution \citep{hanneke2019value,hanneke2023limits}.

A very natural distributional relationship between the source $\sourceDistribution$ and target $\targetDistribution$ distributions is the label-shift setting, where we assume that whilst the marginal probabilities of the class labels may differ $\Prob(Y^{\sourceIndicator}=y)\neq \Prob(Y^{\targetIndicator}=y)$, the class-conditional distributions for the two distributions
$\Prob(X^{\sourceIndicator}=x|Y^{\sourceIndicator}=y)=\Prob(X^{\targetIndicator}=x|Y^{\targetIndicator}=y)$ coincide, for $(X^\sourceIndicator,Y^\sourceIndicator)\sim \sourceDistribution$, $(X^\targetIndicator,Y^\targetIndicator)\sim \targetDistribution$ and  $(x,y) \in \metricSpace\times \{0,1\}$ (see Assumption \ref{assumption:labelShift} below for more details) \citep{saerens2002adjusting,zhang2013domain,du2014semi,nguyen2016continuous,lipton2018detecting,azizzadenesheli2019regularized,garg2020unified,maity2020minimax,wu2021online,bai2022adapting}. The label-shift assumption is motivated by the multitude of settings in which a covariate-vector corresponding to an observation may be viewed as a causal consequence of the class membership represented by the label  \citep{garg2020unified}. A central finding is that the label-shift regime allows for consistent classification with respect to the target distribution $\targetDistribution$ in a semi-supervised setting where the practitioner is given a combination of labelled tuples from the source distribution $(X_i^\sourceIndicator,Y_i^\sourceIndicator)\sim \sourceDistribution$, in conjunction with unlabelled covariates from the target distribution $X_i^\targetIndicator$, where $(X_i^\targetIndicator,Y_i^\targetIndicator)\sim \targetDistribution$, but $Y_i^\targetIndicator$ is unobserved. Recent work \citep{maity2020minimax} has identified the minimax optimal rates for non-parametric classification with label-shift, revealing the dependency upon distributional properties such as smoothness in the aforementioned semi-supervised setting, as well as demonstrating the impact of additional labelled data from the target distribution.

Transfer learning approaches to non-parametric classification have shown the utility of a (typically large) labelled sample from the source distribution $(X_i^\sourceIndicator,Y_i^\sourceIndicator)\sim \sourceDistribution$, in conjunction with smaller labelled sample from the target distribution $(X_i^\targetIndicator,Y_i^\targetIndicator)\sim \targetDistribution$ \citep{cai2019transfer,maity2024linear}. In particular, a relatively small amount of additional labelled data from the target distribution allows one to efficiently learn a classifier in settings where there is more complex relationship between the conditional probabilities of the label for the two distributions which may vary over the covariate domain \citep{reeve2021adaptive}. Statistical approaches to transfer learning have also consider problems beyond the classification setting including high-dimensional linear and generalised regression \citep{li2023estimation,tian2023transfer,sun2023robust}, graphical models \citep{li2023transfer}, non-parameteric regression \citep{liu2023augmented,cai2024transfer}, functional mean estimation \citep{10.1214/24-AOS2362} and contextual multi-armed bandits \citep{suk2021self,cai2024transferBandit}.

A key theme in statistical approaches to supervised transfer learning is adaptivity \citep{kpotufe2021marginal,cai2019transfer,reeve2021adaptive}, particularly with respect to the strength of the relationship between $\sourceDistribution$ and $\targetDistribution$. Indeed, let's suppose we have both a labelled sample from the source and the target distribution. If the source distribution and the target distribution are sufficiently close, in an appropriate sense, then the optimal approach may be to simply run a standard supervised learning algorithm on the combined data set which pools observations from the two samples. On the other hand, if the source and target distributions are completely unrelated then the additional data from the source distribution is without value for the task at hand, and optimal approach may well be to apply a standard supervised learning algorithm on the target data alone. In practice though, we typically won't have prior knowledge of the strength of the relationship between the source and target distributions. This necessitates careful adaptive approaches which typically build upon the groundbreaking approaches to adaptive estimation of Efroimovich, Pinsker, Goldenshluger and Lepski \citep{efroimovich1984learning,klemela2001sharp,lepski1997optimal,lepski2022theory,lepskii1992asymptotically,goldenshluger2013general}.

Recently, there has been an effort to extend the statistical insights of transfer learning to non-stationary sequential settings \citep{maity2022predictor,wu2021online,bai2022adapting}. In such settings, the observations arrive sequentially, with each successive observation sampled from time-dependent distribution. The challenge is to effectively utilise historical data to perform contemporary predictions, despite the distributional changes that occur over time. Of particular importance, is a non-stationary semi-supervised classification setting in which one has access to a fixed labelled data set from a source distribution, and then for subsequent rounds the learner observes a covariate-vector $\unlabelledCovariate_\ell$ to which they are required predict the corresponding class label $\response_\ell$, where at each stage the tuple $(\unlabelledCovariate_\ell,\response_\ell)\sim P_\ell$ is drawn from a distribution on $\metricSpace \times \{0,1\}$. Hence, the learner must make sequential predictions without ever observing the labels $\response_\ell$, beyond the fixed labelled data set, and the distribution $P_\ell$ may change over time. Clearly, distributional assumptions are necessary for any substantive performance guarantees, and a natural assumption is that of label-shift \citep{wu2021online,bai2022adapting}. Wu et al. explored an approach to this non-stationary label-shift setting \citep{wu2021online} grounded in the theory of online learning \citep{cesa2006prediction}. Notably, without further assumptions on the dynamics of the sequence of distributions, beyond the label-shift assumption, it is demonstrated that since it is possible to leverage the initial sample to obtain a noisy estimate of the contemporary risk, one can effectively apply online gradient descent to provide predictions which perform favourably by comparison with the optimal fixed classification rule. However, the optimal fixed classification rule could still be radically sub-optimal compared to a more dynamic strategy, in a non-stationary setting. This observation motivated Bai et al. \citep{bai2022adapting} to introduce a methodology for classification with non-stationary label-shift which provably performs well in comparison with the optimal dynamic policy. The approach of Bai et al. leverages a sophisticated variant of online gradient descent where the learning rate is tuned adaptively via an innovative meta algorithm. Dynamic regret bounds are established which guarantee adaptation to the overall amount of non-stationarity, quantified by a total-variation metric which measures the cumulative changes in the class-label probabilities.

In this work we shall also consider a classification setting with non-stationary label-shift. However, rather than employ techniques from the online learning literature \citep{wu2021online,bai2022adapting} we shall instead take the view that transfer learning in settings where the strength of the distributional relationship is unknown is, in essence, a problem of adaptive statistical estimation. As such, we shall employ a methodology which draws inspiration from the Lepski-technique for adaptive estimation \citep{lepski1997optimal} and demonstrate the merits of such an approach. Our contributions are as follows:

\begin{itemize}
    \item We introduce a methodology for classification with non-stationary label-shift (Section \ref{sec:methodologySec}). Our method combines an approach to estimating a transformed density-ratio which doesn't require that the class-conditional distributions are absolutely continuous (Section \ref{sec:estRegressionFunction}). We then introduce a Legendre-polynomial based procedure for estimating the current marginal probabilities of the class labels (Section \ref{sec:appendixLegendrePolys}). Notably our procedure is fully adaptive and does not require prior knowledge the distributional properties which characterise the degree of non-stationarity.  
    \item We provide a high-probability bound on the classification error for a an individual test-time (Theorem \ref{thm:mainClassificationSingleTimeStep}). This bound exhibits a favourable dependency of the algorithm upon a range of distributional parameters. This includes both parameters which characterise the recent dynamics of the non-stationary label probabilities in terms of the smoothness level over the recent past (Assumption \ref{assumption:smoothlyVaryingLabelProbabilities}), and more familiar distributional parameters which characterise the behaviour of the transformed density in terms of smoothness (Assumption \ref{assumption:smoothRegressionFunctions}) and amount of mass in the vicinity of the optimal decision boundary (Assumption \ref{assumption:tsybakovMarginAssumption}).
    \item We then leverage these insights to provide high-probability bounds on the average dynamic regret. For our first average regret bound (Theorem \ref{thm:highProbBoundForTemporalSmoothnessWithJumps}) we leverage a flexible assumption which allows for sporadic jumps, between which the label probabilities vary smoothly with time (Assumption \ref{assumption:smoothlyVaryingWithJumpsLabelProbabilities}). We also give a high-probability bound on the average dynamic regret which avoids making explicit assumptions on how the label-probabilities vary over time (Corollary \ref{corr:totalVariationLabelProbBound}). Instead, we bound the average dynamic regret in terms of the of cumulative changes quantified through a total-variation metric \eqref{eq:defTotalVariationLabelProbabilities}. The algorithms for these bounds coincide, and bounds are attained without prior knowledge of quantities such as the total-variation of the label-probabilities.
\end{itemize}

Overall, we demonstrate the merits of the adaptive-estimation perspective for sequential non-stationary problems. On the one hand, this approach achieves average regret bounds on the dynamic regret in terms of the total variation of the label-probability changes (Corollary \ref{corr:totalVariationLabelProbBound}), in the spirit of \citep{bai2022adapting}. Moreover, the adaptive methodology also achieves bounds which hold with high-probability for any single test-time (Theorem \ref{thm:mainClassificationSingleTimeStep}), which are not implied by bounds on the overall average.

In the next section we shall formally introduce the statistical setting for the remainder of the paper.









\section{Statistical setting}\label{sec:statisticalSetting}

We shall begin by considering a transfer learning setting for binary classification in which the learner has access to a fixed data set. Later we shall extend our approach to the sequential setting (Section \ref{sec:averageRegretBounds}). Our data consists of three samples $\nullSample$, $\positiveSample$ and $\unlabelledSample$. The sample $\nullSample$ consists $2\numNull$ observations $\nullCovariate_0,\ldots,\nullCovariate_{2\numNull-1}$ drawn independently from a (Borel) distribution $\nullDistribution$ on a metric space $\metricSpace$ with metric $\metric$. The sample $\positiveSample$ consists of $2\numPositive$ observations $\positiveCovariate_0,\ldots,\positiveCovariate_{2\numPositive-1}$ drawn independently from a distribution $\positiveDistribution$ on  $\metricSpace$. It is convenient to assume the sample sizes $2\numNull$, $2\numPositive$ are even to facilitate the discussion of our methodology, which involves sample splitting. We also have a sample $\unlabelledSample$ consisting of $\numUnlabelled$ observations $\unlabelledCovariate_0,\ldots,\unlabelledCovariate_{\numUnlabelled-1}$ which are independent but not necessarily identically distributed. Indeed, we shall assume the existence of a partially observed sequence of independent, but not necessarily identically distributed, random pairs $((\unlabelledCovariate_\ell,\response_\ell))_{\ell =0}^\infty$, where for each $\ell \in \N_0:=\N\cup\{0\}$ we have a covariate $\unlabelledCovariate_{\ell}$, which takes values in $\metricSpace$, and a binary label $\response_\ell$, which takes values in $\{0,1\}$. The learner's objective will be to accurately predict the binary label $\response_{\testTime}$ corresponding a test point $\unlabelledCovariate_{\testTime}$. 
Later we shall consider a dynamic analogue of this setting in which we wish to classify test points over a series of rounds (Section \ref{sec:averageRegretBounds}).

To state our original goal precisely, we let $\setOfDataDependentClassifiers$ denote the set of all data--dependent classifiers of the form $\dataDependentClassifier\equiv \dataDependentClassifier_{\testTime}:\metricSpace^{2\numNull+2\numPositive+\numUnlabelled+1}\rightarrow \{0,1\}$. Our objective is to choose $\dataDependentClassifier \in \setOfDataDependentClassifiers$ to minimise the test error
\begin{align*}
\testError_{\testTime}(\dataDependentClassifier) &:= \Prob\bigl\{\dataDependentClassifier(\unlabelledCovariate_{\testTime}) \neq \response_{\testTime}~|~\nullSample,\positiveSample,\unlabelledSample \bigr\}, 
\end{align*}
where we will often suppress the dependence of both $\dataDependentClassifier$ and $\testError$ on the data $\nullSample$, $\positiveSample$, $\unlabelledSample$ for the sake of notational convenience. We shall also apply the notation $\testError_{\testTime}(\classifier)$ to (data independent) classifiers $\classifier:\metricSpace \rightarrow \{0,1\}$ by viewing such maps as elements of $\setOfDataDependentClassifiers$ which only depend upon their final argument.

Whilst we consider a non-stationary setting, we will require some persistent structure on the distribution of $(\unlabelledCovariate_\ell,\response_\ell)$ for $\ell \in \N_0$. We shall adopt a generalisation of the popular \emph{label-shift} assumption \cite{saerens2002adjusting,zhang2013domain,du2014semi,nguyen2016continuous,lipton2018detecting,azizzadenesheli2019regularized,garg2020unified,maity2020minimax,wu2021online,bai2022adapting}.

\begin{assumption}[Label-shift]\label{assumption:labelShift} There exists a pair of Borel probability distributions 
$\classConditionalDistribution{0}$ and $\classConditionalDistribution{1}$ on $\metricSpace$ such that for all $\ell \in \N_0$ and $y \in\{0,1\}$, the conditional distribution of $\unlabelledCovariate_\ell$ given $\response_\ell=y$ is given by $\classConditionalDistribution{y}$.
\end{assumption}

\Cref{assumption:labelShift} ensures that that the class-conditional distributions $\classConditionalDistribution{0}$ and $\classConditionalDistribution{1}$ are stationary. Hence, the joint distribution of each tuple $(\unlabelledCovariate_\ell,\response_\ell)$ is governed by $\classConditionalDistribution{1}$, $\classConditionalDistribution{0}$ in conjunction with the non-stationary label-probability
\begin{align}\label{eq:defLabelProbability}
\labelProbability[\ell]:= \Prob( \response_\ell=1).
\end{align}
In order to predict $\response_{\testTime}$ we shall require that $\labelProbability[\testTime]$ is in some sense ``predictable'' from $(\labelProbability[\ell])_{\ell =0}^{\testTime-1}$. To make this notion precise we fix $\maximalHolderExponentWeightFunction \in \N$ and for each  $\holderExponentWeightFunction \in [0,\maximalHolderExponentWeightFunction]$, $\holderConstantWeightFunction \in [0,\infty)$ we let $\holderFunctionClass(\holderExponentWeightFunction,\holderConstantWeightFunction)$ denote the H\"{o}lder class consisting of all  ${\pOfAlpha}:=\max\{\lceil \holderExponentWeightFunction \rceil-1,0\}$-times differentiable functions $g: [0,1] \rightarrow \R$ such that the $(\pOfAlpha)$-th derivative $g^{({\pOfAlpha})}$ satisfies 
\begin{align*}
|g^{({\pOfAlpha})}(u_0)-g^{({\pOfAlpha})}(u_1) | \leq \holderConstantWeightFunction\, \left| u_0-u_1\right|^{\holderExponentWeightFunction-{\pOfAlpha}},
\end{align*}
for all $u_0$, $u_1 \in [0,1]$. Note that $g^{(0)}$ refers to the function $g$ itself. In particular, $\holderFunctionClass(0,\holderConstantWeightFunction)$ denotes the class of all functions $g:[0,1]\rightarrow \R$ such that $\sup g-\inf g \leq \holderConstantWeightFunction$.

\begin{assumption}[Temporal smoothness]\label{assumption:smoothlyVaryingLabelProbabilities} There exists $\windowSize \in [\testTime]$, $\holderExponentWeightFunction \in [0,\maximalHolderExponentWeightFunction]$, $\holderConstantWeightFunction \in [0,\infty)$ and $\weightFunction \in \holderFunctionClass(\holderExponentWeightFunction,\holderConstantWeightFunction)$ such that for all $\ell \in \{\testTime-\windowSize,\ldots,\testTime\}$, 
\begin{align*}
\labelProbability=\weightFunction\bigg( \frac{\numUnlabelled-\ell}{\windowSize}\bigg).
\end{align*}
\end{assumption}

We shall also assume that the relationship between the class-conditional distributions varies smoothly over the metric space $\metricSpace$. To state our assumption we let $\classConditionalDistribution{w}:=(1-w) \classConditionalDistribution{0}+w \classConditionalDistribution{1}$ for $w \in (0,1)$ and let $\regressionFunction:\metricSpace \rightarrow [0,1]$ denote the Radon–Nikodym derivative \cite[Section 5.5]{dudley2018real}
\begin{align*}
\regressionFunction(x):= \frac{1}{2}\frac{d\positiveDistribution}{d\classConditionalDistribution{1/2}}(x),
\end{align*}
for all $x \in \metricSpace$, the existence of which is guaranteed by the Radon-Nikodym theorem. Similarly, we define $\averageRegressionFunction(A):= \positiveDistribution(A)/\{2\classConditionalDistribution{1/2}(A)\}$ for Borel sets $A \subseteq \metricSpace$. Given $x \in \X$ and $r \in [0,\infty)$ we let $\openMetricBall{x}{r}:=\{\tilde{x} \in \metricSpace ~:~ \metric(x,\tilde{x}) < r\}$ and $\closedMetricBall{x}{r}:=\{\tilde{x} \in \metricSpace ~:~ \metric(x,\tilde{x}) \leq r\}$ denote, respectively, the open closed metric balls centered at $x$ and of radius $r$. We shall make the following assumption on $\regressionFunction$.

\begin{assumption}[Spatial smoothness]\label{assumption:smoothRegressionFunctions} There exists $\smoothnessExponentRegressionFunction \in (0,1]$, $\smoothnessConstantRegressionFunction \in [0,\infty)$ such that for all $x\in \metricSpace$ and all $r>0$ we have, 
\begin{align*}
\bigl|\regressionFunction(x)-\averageRegressionFunction\bigl\{\closedMetricBall{x}{r}\bigr\} \bigr| \leq \smoothnessConstantRegressionFunction\,\classConditionalDistribution{1/2} \bigl\{\openMetricBall{x}{r}\bigr\}^{\smoothnessExponentRegressionFunction}.
\end{align*}
\end{assumption}

Assumption \ref{assumption:smoothRegressionFunctions} is an analogue of the measure-theoretic notion smoothness of Chaudhuri and Dasgupta \cite{chaudhuri2014rates} and is relatively mild. For example, suppose that $\X \subseteq \R^d$ and the Lebesgue absolutely continuous component of $\nullDistribution+\positiveDistribution$ has a density bounded from below on its support $\mathrm{supp}(\nullDistribution+\positiveDistribution)$, which is sufficiently well behaved (regular in the sense of \cite[(2.1)]{audibert2004classification}). Suppose further that $\regressionFunction$ is $\smoothnessExponentRegressionFunction_H$-H\"{o}lder continuous with respect to the Euclidean norm. Then Assumption \ref{assumption:smoothRegressionFunctions} holds with $\smoothnessExponentRegressionFunction= \smoothnessExponentRegressionFunction_H/d$. More broadly though, Assumption \ref{assumption:smoothRegressionFunctions} does not require that $\mathrm{supp}(\nullDistribution+\positiveDistribution)$ is bounded. Furthermore, Assumption \ref{assumption:smoothRegressionFunctions} does not require that the individual class-conditional densities have a density bounded below, in contrast to \cite{maity2020minimax}.

Finally, we shall utilise a margin assumption in the spirit of \cite{mammen1999}. Let's write $\marginalDistribution[\testTime]$ for the distribution of $X_\testTime$, so that 
\begin{align}\label{eq:marginalDistributionAsMixture}
\marginalDistribution[\testTime]=\classConditionalDistribution{{\labelProbability[\testTime]}}=(1-\labelProbability[\testTime])\nullDistribution+\labelProbability[\testTime] \positiveDistribution.
\end{align}

\begin{assumption}[Tysbakov margin]\label{assumption:tsybakovMarginAssumption} There exists a function $ \tsybakovMarginFunction:(0,\infty) \rightarrow (0,1]$ which is non-decreasing with $\tsybakovMarginFunction(1)=1$, and for all $\zeta \in (0,1]$ we have 
\begin{align*}
\classConditionalDistribution{1/2}\bigl( \bigl\{x \in \metricSpace : 0<|\regressionFunction(x)- 1-\labelProbability[\testTime]| < \zeta \bigr\}\bigr) \leq \tsybakovMarginFunction(\zeta)/\zeta.
\end{align*}
\end{assumption}

Assumptions \ref{assumption:tsybakovMarginAssumption} is best understood in the context of the following lemma. 

\begin{lemma}\label{lemma:classificationErrorSimpleFunction} Suppose $\dataDependentClassifier:\metricSpace^{2\numNull+2\numPositive+\testTime}\rightarrow \{0,1\}$ is a classifier. Then,
\begin{align*}
\testError_{\testTime}(\dataDependentClassifier) = \labelProbability[\testTime]+2\int \dataDependentClassifier\, (1-\labelProbability[\testTime]-\regressionFunction)\,d\classConditionalDistribution{1/2}. 
\end{align*}
Hence, letting $ \bayesClassifier[\testTime]:\X \rightarrow \{0,1\}$ be the classifier defined by $\bayesClassifier[\testTime](x):=\one{\left\lbrace \regressionFunction(x)> 1 -\labelProbability[\testTime] \right\rbrace}$ for $x \in \metricSpace$ we have
\begin{align*}
\testError_\testTime(\dataDependentClassifier)=\testError_\testTime(\bayesClassifier[\testTime])+2\int_{\{ \dataDependentClassifier\neq \bayesClassifier[\testTime]\}}|1-\labelProbability[\testTime]-\regressionFunction|\,d\classConditionalDistribution{1/2}.
\end{align*}
\end{lemma}

Hence, $\bayesClassifier[\testTime]$ is the Bayes optimal classifier at time $\testTime$ and Assumption \ref{assumption:tsybakovMarginAssumption} bounds the amount of mass in the vicinity of the Bayes optimal decision boundary. Lemma \ref{lemma:classificationErrorSimpleFunction} will also serve as the starting point for our methodology. 

An important special case of Assumption \ref{assumption:tsybakovMarginAssumption} is the polynomial setting where we may take a function of the form $\tsybakovPolynomialMarginFunction: (0,\infty) \rightarrow [0,1]$ by $\tsybakovPolynomialMarginFunction(z):=  1 \wedge \left(\tsybakovMarginConstant z^{\tsybakovMarginExponent}\right)$, in place of $\tsybakovMarginFunction$. Notice that Assumption \ref{assumption:tsybakovMarginAssumption} always holds with  $\tsybakovMarginFunction=\tsybakovPolynomialMarginFunction[1,1]$. However, we can often choose $\tsybakovMarginFunction$ so that $\lim_{\zeta \searrow 0} \tsybakovMarginFunction(\zeta)/\zeta = 0$. For example, if $\metricSpace=\R^d$, $\nullDistribution$ and $\positiveDistribution$ have bounded Lebesgue density, and $\regressionFunction$ is smooth with a uniform lower bound on $\|\nabla \regressionFunction\|$ within a bounded neighbourhood of the $(d-1)$-dimensional manifold $\{x \in \R^d,\:\,\regressionFunction(x) = 1-\labelProbability[\testTime]\}$, then for a sufficiently large constant $\tsybakovMarginConstant\geq 1$, Assumption \ref{assumption:tsybakovMarginAssumption} will hold with $\tsybakovMarginFunction =\tsybakovPolynomialMarginFunction[2,\tsybakovMarginConstant]$. As we shall see, in Theorem \ref{thm:mainClassificationSingleTimeStep}, such situations will allow for sharper bounds on the excess classification error.

We close this section with some additional notation. Given $z \in \R$ we let 
\begin{align*}
\logBar(z):= \begin{cases} \log(z) & \text{ if } z>e\\
1 & \text{ if } z \leq 1,
\end{cases}
\end{align*}
where logarithms are base $e$. Given $\sampleSizeVariable \in \N$ and $\tilde{\delta} \in (0,1)$ we define 
\begin{align*}
\epsilonByNDeltaBase{\sampleSizeVariable}{\tilde{\delta}}&:=\logBar( 1/\tilde{\delta})/\sampleSizeVariable,
\end{align*}
and let $\epsilonByNDeltaLogarithm{\sampleSizeVariable}{\tilde{\delta}}:=\epsilonByNDeltaBase{\sampleSizeVariable}{\delta/\sampleSizeVariable}$ and $\epsilonByNDeltaIteratedLogarithm{\sampleSizeVariable}{\tilde{\delta}}:=\epsilonByNDeltaBase{\sampleSizeVariable}{\tilde{\delta}/\logBar(\sampleSizeVariable)} $. Given $k \in \N$ we let $[k]:=\{1,\ldots,k\}$.

\section{Methodology}\label{sec:methodologySec}

Our starting point is Lemma \ref{lemma:classificationErrorSimpleFunction} which verifies that the Bayes optimal classifier $\bayesClassifier[\testTime]:\metricSpace \rightarrow  \{0,1\}$ is given by \[\bayesClassifier[\testTime](x):=\one{\left\lbrace \regressionFunction(x)+\labelProbability[\testTime]> 1 \right\rbrace},\]for $x \in \metricSpace$. This motivates a plug-in strategy whereby we:\medskip

\noindent (1) Obtain an estimate $\estimatedRegressionFunction$ for  $\regressionFunction$;\medskip

\noindent (2) Obtain an estimate $\estimatedLabelProbability$ of $\labelProbability[\testTime]$;\medskip

\noindent (3) Define $\dataDependentClassifier\equiv \dataDependentClassifier_{\testTime}$ by  $\dataDependentClassifier(x) := \one\left\lbrace \estimatedRegressionFunction(x) > 1- \estimatedLabelProbability\right\rbrace$.\medskip

We shall address (1) and (2) in subsections \ref{sec:estRegressionFunction} and \ref{sec:estLabelProbability}, respectively.

\sloppy

\subsection{Estimating the transformed density ratio}\label{sec:estRegressionFunction}

In this section we describe our methodology for estimating the transformed density ratio $\regressionFunction$. Estimating the density-ratio is an important topic in itself \citep{sugiyama2012density}. However, by estimating the transformed density-ration $\regressionFunction$, rather than the density-ratio itself we require less stringent assumptions on the class-conditional distributions. In particular, the transformed density ratio $\regressionFunction$ may be well-behaved even when the density ratio itself does not exist.

Our starting point is to utilise local confidence intervals for the class-conditional distributions with employ local Dvoretzky--Kiefer--Wolfowitz--Massart--type concentration inequalities \cite{reeve2024short}.

Given $q \in [0,1]$ and ${\tilde{\varepsilon}} >0$ we define
\begin{align*}
\empiricalUncertaintyFlat(q,{\tilde{\varepsilon}})&:=\frac{8 \bigl(  \sqrt{ {\tilde{\varepsilon}} \sigma^2(q) + {\tilde{\varepsilon}}^2}+ \{1-2q\}{\tilde{\varepsilon}}\bigr)}{3(1+2{\tilde{\varepsilon}}) },
\end{align*}
Given $\sampleSizeVariable \in \N$ and $\delta \in (0,1)$ we define $\empiricalUncertainty(q,\sampleSizeVariable,\delta):=\empiricalUncertaintyFlat\left(q,\epsilonByNDeltaIteratedLogarithm{\sampleSizeVariable}{\delta}\right) \vee \left({1}/{\sampleSizeVariable}\right)$. Next, for each $y \in \{0,1\}$ we let $\empiricalClassConditionalDistribution{y}$ be the estimate of $\classConditionalDistribution{y}$ defined by
\begin{align*}
\empiricalClassConditionalDistribution{y}(A):= \frac{1}{\numClassConditional{y}}\sum_{i=0}^{\numClassConditional{y}-1}\one_{\{X^y_i \in A\}}
\end{align*}
for Borel sets $A \subseteq \metricSpace$. Further more, for $\confidenceSign \in \{-1,1\}$ and $\delta \in (0,1)$ we let 
\begin{align*}
\empiricalClassConditionalDistributionConfidenceBound{y}{\confidenceSign}{\delta}(A):=\begin{cases}\empiricalClassConditionalDistribution{y}(A)-\empiricalUncertainty(1-\empiricalClassConditionalDistribution{y}(A),\numClassConditional{y},\delta) & \text{ if }\confidenceSign=-1\\
\empiricalClassConditionalDistribution{y}(A)+\empiricalUncertainty(\empiricalClassConditionalDistribution{y}(A),\numClassConditional{y},\delta) & \text{ if }\confidenceSign=1.
\end{cases}
\end{align*}
Furthermore, for $\confidenceSign \in \{-1,1\}$ we let
\begin{align*}
\empiricalRegressionFunctionConfidenceBound{\confidenceSign}{\delta}(A):=\biggl(\biggl(\frac{\empiricalClassConditionalDistributionConfidenceBound{1}{\confidenceSign}{\delta}(A)}{\empiricalClassConditionalDistributionConfidenceBound{0}{-\confidenceSign}{\delta}(A)+\empiricalClassConditionalDistributionConfidenceBound{1}{\confidenceSign}{\delta}(A)}\biggr)\vee 0 \biggr)\wedge 1,
\end{align*} whenever $\empiricalClassConditionalDistributionConfidenceBound{0}{-\confidenceSign}{\delta}(A)+\empiricalClassConditionalDistributionConfidenceBound{1}{\confidenceSign}{\delta}(A)>0$ and let $\empiricalRegressionFunctionConfidenceBound{\confidenceSign}{\delta}(A):=(\confidenceSign+1)/2$ otherwise. We also define,
\begin{align*}
\empiricalRegressionFunctionConfidenceBoundMid{\delta}(A)&:= \frac{1}{2}\left\lbrace \empiricalRegressionFunctionConfidenceBound{-1}{\delta}(A)+\empiricalRegressionFunctionConfidenceBound{1}{\delta}(A)\right\rbrace\\
\regressionFunctionIntervalWidth{\delta}(A)&:= \empiricalRegressionFunctionConfidenceBound{1}{\delta}(A)-\empiricalRegressionFunctionConfidenceBound{-1}{\delta}(A)\\
\regressionFunctionInterval{\delta}(A)&:=\left[ \empiricalRegressionFunctionConfidenceBoundMid{\delta}(A)-\regressionFunctionIntervalWidth{\delta}(A),\empiricalRegressionFunctionConfidenceBoundMid{\delta}(A)+\regressionFunctionIntervalWidth{\delta}(A)\right].
\end{align*}
For each $x \in \metricSpace$ we let
\begin{align*}
\setOfDistances{x}:=\bigcup_{y \in \{0,1\}}\bigl\{\metric(x,X_{\ell}^y) : \ell \in \N_0 \cap [0, \numClassConditional{y}-1]\bigr\}.
\end{align*}
Next, we define $\lepskiRadius{x}{\delta}$ to be the maximal value of $r \in \setOfDistances{x}$ such that 
\begin{align*}
\bigcap_{\tilde{r} \in \setOfDistances{x}\,:\,\tilde{r}\leq r}\regressionFunctionInterval{\delta}\bigl\{\closedMetricBall{x}{\tilde{r}}\bigr\}\neq \emptyset.
\end{align*}
Finally, we define $\estimatedRegressionFunction(x):=\empiricalRegressionFunctionConfidenceBoundMid{\delta}\bigl\{\closedMetricBall{x}{\lepskiRadius{x}{\delta}}\bigr\}$ as our estimate of $\regressionFunction(x)$.

In Lemma \ref{lemma:pointwiseRegressionFunctionHighProbBound} we shall show that, provided Assumption \ref{assumption:smoothRegressionFunctions} holds, for each $x \in \metricSpace$ the following bound holds with high probability,
\begin{align*}
\left| \estimatedRegressionFunction(x)-\regressionFunction(x) \right|\leq 14\left(\left\lbrace 
\smoothnessConstantRegressionFunction^{\frac{1}{\smoothnessExponentRegressionFunction}}\epsilonByNDeltaIteratedLogarithm{\nMin}{\delta} \right\rbrace^{\frac{\smoothnessExponentRegressionFunction}{2\smoothnessExponentRegressionFunction+1}}\vee\sqrt{\epsilonByNDeltaIteratedLogarithm{\nMin}{\delta}}\right),
\end{align*}
where $\nMin:=\numNull \wedge \numPositive$. Moreover, this point-wise bound will be shown to entail a corresponding high-probability simultaneously bound for a set of large measure (Proposition \ref{prop:regressionFunctionHighProbBound}).

\subsection{Estimating the label probability}\label{sec:estLabelProbability}

In this section we introduce our estimator for the label probabilities $\labelProbability$. Our approach will leverage the label-shift assumption (Assumption \ref{assumption:labelShift}). We begin by fixing $\delta \in (0,1)$ and $\maximalHolderExponentWeightFunction \in (0,\infty)$ along with a function $f:\X^{\numClassConditional{0}+\numClassConditional{1}+1} \rightarrow [0,1]$. In what follows we shall view $f$ as a real-valued function on $\X$ by conditioning on the the values of $\nullCovariate_0,\ldots,\nullCovariate_{\numNull-1}$, $\positiveCovariate_0,\ldots,\positiveCovariate_{\numPositive-1}$, which we take to be the first $\numNull+\numPositive$ arguments of $f$, thus leaving only the final argument unspecified so that 
\begin{align*}
f(x)\equiv f(\nullCovariate_0,\ldots,\nullCovariate_{\numNull-1},\positiveCovariate_0,\ldots,\positiveCovariate_{\numPositive-1},x).
\end{align*}
Moreover, given a finite Borel measure $\tilde{\mu}$ on $\metricSpace$ we let
\begin{align*}
\tilde{\mu}(f)&:=\int f d\tilde{\mu}\equiv \int f(\nullCovariate_0,\ldots,\nullCovariate_{\numNull-1},\positiveCovariate_0,\ldots,\positiveCovariate_{\numPositive-1},x) d\tilde{\mu}(x).
\end{align*}
In order to estimate $\labelProbability[\testTime]$ we apply \eqref{eq:marginalDistributionAsMixture} to deduce that, provided $ \positiveDistribution(f)\neq \nullDistribution(f)$, we have
\begin{align}\label{eq:justificationOfPluginRuleLabelProbability}
\labelProbability[\testTime]&=\frac{\marginalDistribution[\testTime](f)-  \nullDistribution(f)}{\positiveDistribution(f)-\nullDistribution(f)}.
\end{align}
For $y \in \{0,1\}$, we can estimate $\classConditionalDistribution{y}(f)$ by $\estimatorClassConditionalSecondSample{y}(f):={\numClassConditional{y}}^{-1}\sum_{i=\numClassConditional{y}}^{2\numClassConditional{y}-1}f(X^y_i)$. Hence, the primary challenge is to obtain a suitable estimate of $\marginalDistribution[\testTime](f)$. The primary challenge is to adapt to the level of temporal smoothness encoded via Assumption \ref{assumption:smoothlyVaryingLabelProbabilities}. We shall introduce a family of non-parametric estimators $\estimatorMarginalDistribution[q]{\testTime}(f)$ of the integral $\marginalDistribution[\testTime](f):=\int fd\marginalDistribution[\testTime]$ where $q  \in [\testTime]$. We introduce a family of shifted Legendre polynomials $(\legendrePolynomialShiftedOrthonormal)_{k \in \N_0}$ defined by $\legendrePolynomialShiftedOrthonormal[0]\equiv 1$ and
\begin{align*}
\legendrePolynomialShiftedOrthonormal(z):=\frac{\sqrt{2k+1}}{k!}\frac{d^{k}}{dz^{k}}\{z(z-1)\}^k,
\end{align*}
for $k \in \N$. Given $q \in [\testTime]$ and $p \in \N_0$ we let $ \weightMatrixByPAndQ$ denote the $q\times (p+1)$ matrix with entries $ \weightMatrixByPAndQNoArg_{i,j}(q,p):= \legendrePolynomialShiftedOrthonormal[j-1](i/q)$ for $(i,j) \in [q]\times [p+1]$. We then choose $\pOfQ$ to be the maximal value of $p \in \{0,\ldots, \maximalHolderExponentWeightFunction -1\}$ such that $\weightMatrixByPAndQ$ has rank $p+1$. We then let $\weightMatrixByQNoArg \equiv \weightMatrixByQ := \weightMatrixByPAndQNoArg(q,\pOfQ)$ denote the corresponding $q \times (\pOfQ+1)$ matrix. Given $i \in \{0,\ldots,q\}$, we let $\weightMatrixByQNoArg _{i,:}=(\weightMatrixByQNoArg _{i,j})_{j=1}^{\pOfQ+1}$ is the length $\pOfQ+1$ row vector with entries $\weightMatrixByQNoArg _{i,j}:=\legendrePolynomialShiftedOrthonormal[j-1](i/q)$ for $j=1,\ldots,\pOfQ+1$. Next, for each $i \in [q]$ we let
\begin{align*}
\estimatorMarginalDistributionWeights:=\weightMatrixByQNoArg _{0,:}\bigl( \weightMatrixByQNoArg^\top \weightMatrixByQNoArg \bigr)^+{{\weightMatrixByQNoArg_{i,:}}}^\top.
\end{align*}\sloppy
We then define $\estimatorMarginalDistribution[q]{\testTime}(f):=\sum_{i=1}^q \estimatorMarginalDistributionWeights f(X_{t-i})$ and let 
\begin{align*}
\estimatorMarginalDistributionRootVarTerm[q]{\testTime}{\delta}:= \|\estimatorMarginalDistributionWeightsNoArg\|_2\sqrt{2\log\bigg(\frac{\pi^2q^2}{\delta}\bigg)}.
\end{align*}
where $\|\estimatorMarginalDistributionWeightsNoArg\|^2_2:=\sum_{i=1}^q\estimatorMarginalDistributionWeights^2$. Next, for $\testTime\geq 8\maximalHolderExponentWeightFunction^2(\maximalHolderExponentWeightFunction+1)^2$, let $\lepskiQ$ be the maximal value of $q \in  \{8\maximalHolderExponentWeightFunction^2(\maximalHolderExponentWeightFunction+1)^2,\ldots,\testTime\}$ such that 
\begin{align*}
\left|\estimatorMarginalDistribution[q]{\testTime}(f)-\estimatorMarginalDistribution[q_\flat]{\testTime}(f)\right|\leq 2 \left\lbrace \estimatorMarginalDistributionRootVarTerm[q]{\testTime}{\delta}+\estimatorMarginalDistributionRootVarTerm[q_\flat]{\testTime}{\delta}\right\rbrace,
\end{align*}
for all $q_\flat \in \{8\maximalHolderExponentWeightFunction^2(\maximalHolderExponentWeightFunction+1)^2,\ldots,q-1\}$. For  $\testTime <8\maximalHolderExponentWeightFunction^2(\maximalHolderExponentWeightFunction+1)^2$ we let $\lepskiQ :=\testTime$. Our final estimator $\estimatorMarginalDistributionLepski(f)$ for $\marginalDistribution[\testTime](f)$ is then given by $\estimatorMarginalDistributionLepski(f):=\estimatorMarginalDistribution[\lepskiQ]{\testTime}(f)$. Proposition \ref{prop:marginalDistributionEstimatorBound} gives a high probability bound on the distance between $\estimatorMarginalDistribution[\lepskiQ]{\testTime}(f)$  and $\marginalDistribution[\testTime](f)$.

Our final estimator $\estimatorLabelProbabilityLepski$ for $\labelProbability[\testTime]$ is defined by
\begin{align*}
\estimatorLabelProbabilityLepski&:=\left(\left(\frac{\estimatorMarginalDistributionLepski(f)-  \estimatorClassConditionalSecondSample{0}(f)}{\estimatorClassConditionalSecondSample{1}(f)-\estimatorClassConditionalSecondSample{0}(f)}\right) \vee 0 \right) \wedge 1,
\end{align*}
provided $\estimatorClassConditionalSecondSample{0}(f) \neq \estimatorClassConditionalSecondSample{1}(f) $, and $\estimatorLabelProbabilityLepski:=1/2$ otherwise. This corresponds to a plug-in estimator motivated by \eqref{eq:justificationOfPluginRuleLabelProbability}. Corollary \ref{corr:labelProbabilityEstimatorBound} gives a high-probability bound on the distance between $\estimatorLabelProbabilityLepski$ and $\labelProbability[\testTime]$.


It remains to select a suitable function $f:\X^{\numNull+\numPositive+1} \rightarrow [0,1]$ and finalise our classifier, to apply the estimator $\estimatorLabelProbabilityLepski$. The denominator in the construction of $\estimatedLabelProbability$ suggests that we should choose $f:\metricSpace^{\numNull+\numPositive+1}\rightarrow [0,1]$ to maximise the absolute value of
\begin{align*}
(\positiveDistribution-\nullDistribution)(f) =\int fd\positiveDistribution-\int f\nullDistribution&=2\int f (2\regressionFunction-1)d\classConditionalDistribution{1/2},
\end{align*}
with a view to ensuring that denominator $\estimatorClassConditionalSecondSample{1}(f)-\estimatorClassConditionalSecondSample{0}(f)$ in the construction of $\estimatorLabelProbabilityLepski$ is not too close to zero. The bound given in Corollary \ref{corr:labelProbabilityEstimatorBound} will make this intuition precise. The maximum absolute value of $(\positiveDistribution-\nullDistribution)(f)$ would be attained by taking the function $f_{\regressionFunction}:\metricSpace \rightarrow [0,1]$ defined by $f_{\regressionFunction}(x):=\one\{2\regressionFunction(x)\geq 1\}$ for $x \in \metricSpace$, viewed as a function on $\metricSpace^{\numNull+\numPositive+1}$ which is constant in all but the final argument. Indeed, we have
\begin{align*}
(\positiveDistribution-\nullDistribution)(f_\regressionFunction) &= \totalVariation(\nullDistribution,\positiveDistribution):= \sup_{A\subseteq \metricSpace} \left| \positiveDistribution(A) - \nullDistribution(A)\right|,
\end{align*}
where the supremum is taken over all Borel measurable subsets $A \subseteq \metricSpace$. Since $\regressionFunction$ is not known a priori we shall leverage $\hat{f}_\delta$ defined by $\hat{f}_\delta(x):=\one\{2\estimatedRegressionFunction(x)\geq 1\}$, where $\estimatedRegressionFunction$ is the estimate for $\regressionFunction$ presented in Section \ref{sec:estRegressionFunction}.  Note that since $\estimatedRegressionFunction(x)$ is measurable with respect to $\nullCovariate_0,\ldots,\nullCovariate_{\numNull-1}$, $\positiveCovariate_0,\ldots,\positiveCovariate_{\numPositive-1}$, for each $x \in \metricSpace$, the same is true of $\hat{f}_\delta(x)$. 

Finally, we let $\estimatedLabelProbability:=\estimatorLabelProbabilityLepski[\hat{f}_\delta]$ and let $ \hat{\dataDependentClassifier}_{\testTime,\delta}$ be the classifier defined by  $\hat{\dataDependentClassifier}_{\testTime,\delta}(x) := \one\left\lbrace \estimatedRegressionFunction(x)+ \estimatedLabelProbability>1\right\rbrace$.

\section{A bound on the simple regret}\label{sec:regretBoundSingleTestTime}

In this section we shall present high-probability bound for the excess classification error bound for the classifier $\hat{\dataDependentClassifier}_{\testTime,\delta}$, described in Section \ref{sec:methodologySec}. To state our results, we require the following notation. Define the quantities $\highProbBoundDeltaLabelledData  \equiv \highProbBoundDeltaLabelledDataNoArg(\nMin,\smoothnessExponentRegressionFunction,\smoothnessConstantRegressionFunction, \totalVariation(\nullDistribution,\positiveDistribution))$ and $\highProbBoundSingleTimeStepDeltaUnlabelledData\equiv \highProbBoundSingleTimeStepDeltaUnlabelledDataNoArg(\windowSize,\holderExponentWeightFunction,\holderConstantWeightFunction, \totalVariation(\nullDistribution,\positiveDistribution))$ by
\begin{align*}
\highProbBoundDeltaLabelledData &:= \frac{\epsilonByNDeltaLogarithm{\nMin}{\delta}^{\frac{1}{2}}}{\totalVariation(\nullDistribution,\positiveDistribution)} \vee \left\lbrace 
\smoothnessConstantRegressionFunction^{\frac{1}{\smoothnessExponentRegressionFunction}}\epsilonByNDeltaIteratedLogarithm{\nMin}{\delta} \right\rbrace^{\frac{\smoothnessExponentRegressionFunction}{2\smoothnessExponentRegressionFunction+1}} \hspace{7mm}\text{and}\hspace{7mm}
\highProbBoundSingleTimeStepDeltaUnlabelledData:= \frac{\epsilonByNDeltaLogarithm{\windowSize }{\delta}^{\frac{1}{2}} }{\totalVariation(\nullDistribution,\positiveDistribution)}\vee \left\lbrace \frac{\holderConstantWeightFunction^{\frac{1}{\holderExponentWeightFunction}}\epsilonByNDeltaLogarithm{\windowSize}{\delta}}{\totalVariation(\nullDistribution,\positiveDistribution)^2}\right\rbrace^{\frac{\holderExponentWeightFunction}{2\holderExponentWeightFunction+1}},
\end{align*}
for $\holderExponentWeightFunction>0$ and extending continuously to $\holderExponentWeightFunction=0$ by $\highProbBoundSingleTimeStepDeltaUnlabelledData:=\bigl\{\epsilonByNDeltaLogarithm{\windowSize }{\delta}^{\frac{1}{2}} /\totalVariation(\nullDistribution,\positiveDistribution) \bigr\} \vee \holderConstantWeightFunction$.

\begin{theorem}\label{thm:mainClassificationSingleTimeStep}
Suppose Assumption \ref{assumption:labelShift} holds, Assumption \ref{assumption:smoothlyVaryingLabelProbabilities} hold with  $\windowSize \in [\numUnlabelled]$, $\holderExponentWeightFunction \in [0,\maximalHolderExponentWeightFunction]$, $\holderConstantWeightFunction \in [0,\infty)$, Assumption \ref{assumption:smoothRegressionFunctions} holds with $\smoothnessExponentRegressionFunction \in (0,1]$, $\smoothnessConstantRegressionFunction \in [0,\infty)$ and Assumption \ref{assumption:tsybakovMarginAssumption} with $\tsybakovMarginFunction:(0,\infty) \rightarrow (0,1]$. There exists a constant $\constantForMainClassificationSingleTimeStep\geq 1$, depending only on $\maximalHolderExponentWeightFunction$ such that for all $\delta \in (0,\totalVariation(\nullDistribution,\positiveDistribution)/\constantForMainClassificationSingleTimeStep)$ and $\nMin=\min\{\numNull,\numPositive\}$ if $ \smoothnessConstantRegressionFunction\epsilonByNDeltaIteratedLogarithm{\nMin}{\delta}^\smoothnessExponentRegressionFunction \leq \left\lbrace  \totalVariation(\nullDistribution,\positiveDistribution)/\constantForMainClassificationSingleTimeStep\right\rbrace^{2\smoothnessExponentRegressionFunction+1}$ then with probability at least $1-\delta$,
\begin{align*}
{\testError_{\testTime}(\hat{\dataDependentClassifier}_{\testTime,\delta})-\testError_\testTime(\bayesClassifier[\testTime])}\leq {2\tsybakovMarginFunction\bigl(\constantForMainClassificationSingleTimeStep \bigl\{\highProbBoundDeltaLabelledData\vee \highProbBoundSingleTimeStepDeltaUnlabelledData\bigr\}\bigr) +\constantForMainClassificationSingleTimeStep\delta}.
\end{align*}
\end{theorem}

To interpret Theorem \ref{thm:mainClassificationSingleTimeStep} let's consider the setting where $\tsybakovMarginFunction=\tsybakovPolynomialMarginFunction: (0,\infty) \rightarrow [0,1]$, by $\tsybakovPolynomialMarginFunction(z):=  1 \wedge \left(\tsybakovMarginConstant z^{\tsybakovMarginExponent}\right)$ for $z \in [0,\infty)$, where $\tsybakovMarginExponent \in [1,2]$ and $\smoothnessConstantRegressionFunction$, $\holderConstantWeightFunction$, $\tsybakovMarginConstant$ and $\totalVariation(\nullDistribution,\positiveDistribution)$ are viewed as a constants and take $\delta = 1/\nMin$. Then Theorem \ref{thm:mainClassificationSingleTimeStep} gives the bound
\begin{align*}
\E\left\lbrace \testError_{\testTime}(\hat{\dataDependentClassifier}_{\testTime,\delta})\right\rbrace \leq \testError_\testTime(\bayesClassifier[\testTime])+\tilde{O}\left( {\nMin}^{-\frac{\smoothnessExponentRegressionFunction\tsybakovMarginExponent}{2\smoothnessExponentRegressionFunction+1}}+ {\windowSize}^{-\frac{\holderExponentWeightFunction\tsybakovMarginExponent}{2\holderExponentWeightFunction+1}}\right),
\end{align*}
where $\tilde{O}$ suppresses dependence on poly-logarithmic factors. One interesting setting is when Assumption \ref{assumption:smoothlyVaryingLabelProbabilities} holds with $\holderConstantWeightFunction=0$, so that the label probabilities are constant over $\windowSize$ time steps. Then, if we again view $\smoothnessConstantRegressionFunction$, $\tsybakovMarginConstant$ and $\totalVariation(\nullDistribution,\positiveDistribution)$ are viewed as a constants and take $\delta = 1/\nMin$, Theorem \ref{thm:mainClassificationSingleTimeStep} yields a bound of the form  
\begin{align*}
\E\left\lbrace \testError_{\testTime}(\hat{\dataDependentClassifier}_{\testTime,\delta})\right\rbrace \leq \testError_\testTime(\bayesClassifier[\testTime])+\tilde{O}\left( {\nMin}^{-\frac{\smoothnessExponentRegressionFunction\tsybakovMarginExponent}{2\smoothnessExponentRegressionFunction+1}}+ {\windowSize}^{-\frac{\tsybakovMarginExponent}{2}}\right).
\end{align*}
This matches the minimax optimal rate of \cite[Theorem 12]{maity2020minimax} without prior knowledge of the number $\windowSize$ of informative covariate observations.

Our approach to proving Theorem \ref{thm:mainClassificationSingleTimeStep} is described in Section \ref{sec:mainIdeasForProofSection}.

\section{Bounds on the average dynamic regret}\label{sec:averageRegretBounds}

\newcommand{\timeInterval}{\mathbb{T}}

\newcommand{\timeHorizon}{\mathrm{T}}
\newcommand{\numJumps}{\mathrm{J}}

\newcommand{\labelProbabilityFunctionWithJumps}[1][j]{h_{#1}}

\newcommand{\testTimeWithinSequence}{t}

\newcommand{\unlabelledDataWithinSequence}{\unlabelledSample(\testTimeWithinSequence)}

\newcommand{\policy}{\bm{\varphi}}
\newcommand{\empiricalPolicy}[1][\delta]{\bm{\hat{\varphi}}_{#1}^{\diamond}}

\newcommand{\policySingleTimeStep}{\varphi}
\newcommand{\empiricalPolicySingleTimeStep}[1][\testTimeWithinSequence,\delta_{\testTimeWithinSequence}]{\hat{\varphi}_{#1}^{\diamond}}

\newcommand{\comparisonClassifierForRegretDef}{\mathring{\policySingleTimeStep}_{\testTimeWithinSequence}}

\newcommand{\regretOverTimeInterval}{\mathcal{R}}

We shall now consider a sequential framework in which our task is to devise a policy for classifying examples over a series rounds. As before, we shall assume that we have access to a sample $\nullSample$ consisting of $2\numNull$ independent examples with distribution $\nullDistribution$ and a sample $\positiveSample$ consisting of $2\numPositive$ independent examples with distribution $\positiveDistribution$. Also as before, we assume the existence of the sequence of independent random pairs $((\unlabelledCovariate_\ell,\response_\ell))_{\ell =0}^\infty$ in $\metricSpace \times \{0,1\}$. At each time step $\testTimeWithinSequence \in \N$ we have access to an unlabelled sample $\unlabelledDataWithinSequence:=(\unlabelledCovariate_\ell)_{\ell=0}^{\testTimeWithinSequence}$ consisting of the previous $\testTimeWithinSequence$ covariate vectors. Our goal is to leverage the currently visible data $\nullSample$, $\positiveSample$, $\unlabelledDataWithinSequence$, to provide a rule for predicting the $\response_{\testTime}$ from $\unlabelledCovariate_{\testTime}$. As such we seek a policy $\policy=(\policySingleTimeStep_\testTimeWithinSequence)_{\testTimeWithinSequence \in \N}$ where for each $\testTimeWithinSequence \in \N$ we would like to obtain $\policySingleTimeStep \in \setOfDataDependentClassifiers[\testTimeWithinSequence]$ to minimise
{\small{\begin{align*}
\testError_{\testTimeWithinSequence}(\policySingleTimeStep_{\testTimeWithinSequence}) &:= \Prob\bigl\{\policySingleTimeStep_{\testTimeWithinSequence}(\nullSample,\positiveSample,\unlabelledDataWithinSequence,\unlabelledCovariate_{\testTimeWithinSequence}) \neq \response_{\testTimeWithinSequence}~|~\nullSample,\positiveSample,\unlabelledDataWithinSequence \bigr\}, 
\end{align*}}}

\noindent In order to quantify the performance of our policy $\policy$, we consider time intervals $\timeInterval=\{\min\timeInterval,\ldots,\max\timeInterval\}$ and let
\begin{align*}
\regretOverTimeInterval_{\timeInterval}(\policy)&:=\frac{1}{|\timeInterval|}\sum_{\testTimeWithinSequence \in \timeInterval} \biggl\{  \testError_{\testTimeWithinSequence}(\policySingleTimeStep_{\testTimeWithinSequence})-  \inf_{\comparisonClassifierForRegretDef}\testError_{\testTimeWithinSequence}(\comparisonClassifierForRegretDef)\biggr\},  
\end{align*}
where the infimum is taken over all maps $\comparisonClassifierForRegretDef \in \setOfDataDependentClassifiers[\testTimeWithinSequence]$. 

We shall consider the policy $\empiricalPolicy=(\empiricalPolicySingleTimeStep)_{\testTimeWithinSequence \in \N}$ as follows. First, we leverage the first half of the labelled samples $\nullSample$, $\positiveSample$ to construct the estimate $\estimatedRegressionFunction$ of the regression function as in Section \ref{sec:estRegressionFunction}. We then take $\hat{f}_\delta$ to be the function defined by $\hat{f}_\delta(x):=\one\{2\estimatedRegressionFunction(x)\geq 1\}$ for $x \in \metricSpace$. We then use the second half of our labelled samples $\nullSample$, $\positiveSample$ to construct estimates $\estimatorClassConditionalSecondSample{y}(\hat{f}_\delta)$ for $\classConditionalDistribution{y}(\hat{f}_\delta)$ for $y \in \{0,1\}$, as described in Section \ref{sec:estLabelProbability}. Next, for each $\testTimeWithinSequence \in \N$, we take $\delta_{\testTimeWithinSequence}:=(\pi^2\delta)/(6\testTimeWithinSequence^2)$ and leverage our unlabelled data $\unlabelledDataWithinSequence$ to construct an estimate $\hat{\marginalDistribution}_{\testTimeWithinSequence,\delta_t}(\hat{f}_\delta)$ for $\marginalDistribution[\testTimeWithinSequence](\hat{f}_\delta)$ constructed as in Section \ref{sec:estLabelProbability} but with $\unlabelledDataWithinSequence$ in place of $\unlabelledSample$ and $\delta_t$ in place of $\delta$. We then combine our estimates  $\estimatorClassConditionalSecondSample{y}(\hat{f}_\delta)$ for $y \in \{0,1\}$, and $\hat{\marginalDistribution}_{\testTimeWithinSequence,\delta_t}(\hat{f}_\delta)$ as in Section \ref{sec:estLabelProbability} to obtain the corresponding estimate $\estimatorLabelProbabilityLepskiUnspecified{\testTimeWithinSequence}{\delta_t}{\hat{f}_\delta}$ for $\labelProbability[\testTimeWithinSequence]$. Finally, we let $\empiricalPolicySingleTimeStep \in \setOfDataDependentClassifiers$ denote the data-dependent classifier defined for $x \in \metricSpace$ by  
\begin{align*}\empiricalPolicySingleTimeStep(x) := \one\left\lbrace \estimatedRegressionFunction(x)+ \estimatorLabelProbabilityLepskiUnspecified{\testTimeWithinSequence}{\delta_t}{\hat{f}_\delta}>1\right\rbrace.
\end{align*}

In order to demonstrate the efficacy of the policy $\empiricalPolicy$ we shall consider two distinct regimes: First, a regime in which the sequence label probabilities $(\labelProbability[\testTimeWithinSequence])_{\testTimeWithinSequence \in \N}$ varies smoothly, except for a small number of jumps. Second a regime in which we bound the overall variation in the sequence of label probabilities. As before, we shall also adopt the label-shift assumption (Assumption \ref{assumption:labelShift}) as well as the smoothness assumption on $\regressionFunction$ (Assumption \ref{assumption:smoothRegressionFunctions}).

\begin{assumption}[Temporal smoothness with jumps]\label{assumption:smoothlyVaryingWithJumpsLabelProbabilities} Let $\timeInterval \subseteq \N$ be a time interval. Suppose $\timeHorizon =|\timeInterval|$, $\numJumps \in [\timeHorizon]$, $\holderExponentWeightFunction \in [0,\maximalHolderExponentWeightFunction]$ and $\holderConstantWeightFunction \in [0,\infty)$. Suppose also that there exists a strictly increasing sequence $(s_j)_{j=0}^{\numJumps} \in \N_0^{\numJumps+1}$ with $s_0=\min \timeInterval-1$, $s_{\numJumps}=\max \timeInterval$ and for each $j \in [\numJumps]$ there exists $\labelProbabilityFunctionWithJumps \in \holderFunctionClass(\holderExponentWeightFunction,\holderConstantWeightFunction)$ and 
\begin{align*}
\labelProbability=\labelProbabilityFunctionWithJumps\bigg(\frac{\ell-s_{j-1}-1}{s_{j}-s_{j-1}-1}\bigg),
\end{align*}
such that for all $\ell \in \{s_{j-1}+1,\ldots,s_{j}\}$.
\end{assumption}

We shall also use the following sequential analogue of Assumption \ref{assumption:tsybakovMarginAssumption}.

\begin{assumption}[Sequential Tysbakov margin]\label{assumption:tsybakovMarginAssumptionSequential} Let $\timeInterval \subseteq \N$ be a time interval. Suppose $\tsybakovMarginExponent \in [1/2,\infty)$ and $\tsybakovMarginConstant \in [1,\infty)$, are such that for all $\testTimeWithinSequence \in \timeInterval$ and $\zeta \in (0,1]$, 
\begin{align*}
\classConditionalDistribution{1/2}\bigl( \bigl\{x \in \metricSpace : 0<|\regressionFunction(x)- 1-\labelProbability[\testTimeWithinSequence]| < \zeta \bigr\}\bigr) \leq \tsybakovMarginConstant\, \zeta^{\tsybakovMarginExponent-1}.
\end{align*}
\end{assumption}

Note that Assumption \ref{assumption:tsybakovMarginAssumptionSequential} always holds with $\tsybakovMarginExponent=1$ and $\tsybakovMarginConstant=2$. However, we would often expect Assumption \ref{assumption:tsybakovMarginAssumptionSequential} to hold with $\tsybakovMarginExponent=2$ (see Section \ref{sec:statisticalSetting}). Given $q \in (0,\infty)$ and $r \in [1,\infty)$ we let
\begin{alignat*}{2}
\averagePowerTransform{r}{q}:&=\left\lbrace r^{q-1}\left(1+\int_1^r  z^{-q} dz\right)\right\rbrace^{\frac{1}{q}}= \begin{cases} 
\log(r) &\text{ if }q=1\\ \left(
 \frac{1-qr^{q-1}}{1-q}\right) ^{\frac{1}{q}} &\text{ if } q\neq 1.\end{cases}
\end{alignat*}
Note that for $q<1$ we have $\sup_{r \in [1,\infty)}\averagePowerTransform{r}{q}= (1-q)^{-\frac{1}{q}}$. To state our bound we also define
\begin{align*}
\highProbabilityUpperBoundUnlabelledForAverageRegretBound(r,\holderConstantWeightFunction):=\sqrt{\frac{\averagePowerTransform{r}{\frac{\tsybakovMarginExponent}{2}}\epsilonByNDeltaBase{r}{\frac{\delta}{\maxTimeInterval}}}{ \totalVariation(\nullDistribution,\positiveDistribution)^2}}\vee\left\lbrace\frac{\holderConstantWeightFunction^{\frac{1}{\holderExponentWeightFunction}}\epsilonByNDeltaBase{r}{\frac{\delta}{\maxTimeInterval}}}{ \totalVariation(\nullDistribution,\positiveDistribution)^2}\right\rbrace^{\frac{\holderExponentWeightFunction}{2\holderExponentWeightFunction+1}}\hspace{-5mm},
\end{align*}
for $\holderExponentWeightFunction>0$ and extend continuously to $\holderExponentWeightFunction=0$ by $\highProbabilityUpperBoundUnlabelledForAverageRegretBoundZeroHolder(r,\holderConstantWeightFunction):=\holderConstantWeightFunction \vee \sqrt{{\averagePowerTransform{r}{{\tsybakovMarginExponent}/{2}}\epsilonByNDeltaBase{r}{{\delta}/{\maxTimeInterval}}}}/{ \totalVariation(\nullDistribution,\positiveDistribution)}$. Note that in the piece-wise stationary case we have $\holderConstantWeightFunction=0$ so that \[\highProbabilityUpperBoundUnlabelledForAverageRegretBound(r,0)=\sqrt{{\averagePowerTransform{r}{{\tsybakovMarginExponent}/{2}}\epsilonByNDeltaBase{r}{{\delta}/{\maxTimeInterval}}}}/{ \totalVariation(\nullDistribution,\positiveDistribution)}.\]

Our bound for the setting with smoothly varying labels with jumps is as follows.

\begin{theorem}\label{thm:highProbBoundForTemporalSmoothnessWithJumps} Let $\timeInterval \subseteq \N$ be a time interval with $\max \timeInterval=\maxTimeInterval \in \N$. Suppose Assumption \ref{assumption:labelShift} holds, Assumption \ref{assumption:smoothRegressionFunctions} holds with $\smoothnessExponentRegressionFunction \in (0,1]$, $\smoothnessConstantRegressionFunction \in [0,\infty)$, Assumption \ref{assumption:smoothlyVaryingWithJumpsLabelProbabilities} holds with $\timeHorizon =|\timeInterval|$, $\numJumps \in [\timeHorizon]$, $\holderExponentWeightFunction \in [0,\maximalHolderExponentWeightFunction]$ and $\holderConstantWeightFunction \in [0,\infty)$, and Assumption \ref{assumption:tsybakovMarginAssumptionSequential} holds with $\tsybakovMarginExponent \in [1,2+1/\holderExponentWeightFunction]$, $\tsybakovMarginConstant \in [1,\infty)$. There exists a constant $\constantForSmoothnessWithJumpsAverageRegretBound \geq 1$, depending only on $\maximalHolderExponentWeightFunction$ and $\tsybakovMarginExponent$ such that for all $\delta \in (0,\totalVariation(\nullDistribution,\positiveDistribution)/\constantForSmoothnessWithJumpsAverageRegretBound]$ and $\nMin=\min\{\numNull,\numPositive\}$ with $ \smoothnessConstantRegressionFunction\epsilonByNDeltaIteratedLogarithm{\nMin}{\delta}^\smoothnessExponentRegressionFunction \leq \left\lbrace  \totalVariation(\nullDistribution,\positiveDistribution)/\constantForSmoothnessWithJumpsAverageRegretBound \right\rbrace^{2\smoothnessExponentRegressionFunction+1}$  then with probability at least $1-\delta$,
\begin{align*}
\regretOverTimeInterval_{\timeInterval}(\policy)\leq \tsybakovMarginConstant\,\constantForSmoothnessWithJumpsAverageRegretBound \bigl\{\highProbBoundDeltaLabelledData\vee \highProbabilityUpperBoundUnlabelledForAverageRegretBound(\timeHorizon/\numJumps,\holderConstantWeightFunction)\bigr\}^{\tsybakovMarginExponent} +\constantForSmoothnessWithJumpsAverageRegretBound\delta.
\end{align*}
\end{theorem}

To interpret Theorem \ref{thm:highProbBoundForTemporalSmoothnessWithJumps} let's view $\tsybakovMarginConstant$ and $\totalVariation(\nullDistribution,\positiveDistribution)$ as a constant and take $\holderExponentWeightFunction \in [1,2]$ and $\delta = 1/\nMin$. Then Theorem \ref{thm:highProbBoundForTemporalSmoothnessWithJumps} yields a bound of the form 
\begin{align*}
\E\left\lbrace \regretOverTimeInterval_{\timeInterval}(\policy) \right\rbrace \leq \tilde{O}\left(\highProbBoundDeltaLabelledData^{\tsybakovMarginExponent}+\left\lbrace \holderConstantWeightFunction+\sqrt{\frac{\numJumps}{\timeHorizon}}\right\rbrace^{\frac{\tsybakovMarginExponent}{2\holderExponentWeightFunction+1}}\left\lbrace \frac{\numJumps}{\timeHorizon}\right\rbrace^{\frac{\tsybakovMarginExponent\holderExponentWeightFunction}{2\holderExponentWeightFunction+1}}\right).
\end{align*}

As an alternative to making assumptions on the behaviour of the sequence of label-probabilities $(\labelProbability)_{\ell \in \N_0}$ we can instead introduce a metric on the overall level of change over a given time interval. Given $\holderExponentWeightFunction \in [1,\infty)$ we let
\begin{align}\label{eq:defTotalVariationLabelProbabilities}
\totalVariationLabelProbabilities:=\left(\sum_{\ell=\min \timeInterval}^{\max \timeInterval-1}|\labelProbability[\ell]-\labelProbability[\ell+1]|^{\frac{1}{\holderExponentWeightFunction}}\right)^{\holderExponentWeightFunction},
\end{align}
where we suppress the dependence upon $\timeHorizon$ for notational convenience. When $\holderExponentWeightFunction=1$, the metric $\totalVariationLabelProbabilities$ corresponds precisely to the total variation measure explored by Bai et al. \cite{bai2022adapting}. For larger values of $\holderExponentWeightFunction$, the metric $\totalVariationLabelProbabilities$ places an increasing penalty on sequences with numerous small changes, rather than a small number of jumps. Note that for $\holderExponentWeightFunction>1$, the quantity $\totalVariationLabelProbabilities$ does not correspond to a norm on the sequence $(\labelProbability[\ell]-\labelProbability[\ell+1])_{\ell}$ since the function $(z_\ell)_{\ell} \mapsto \left(\sum_{\ell}|z_\ell|^{\frac{1}{\holderExponentWeightFunction}}\right)^{\holderExponentWeightFunction}$ is not sub-additive.

\begin{corollary}\label{corr:totalVariationLabelProbBound} Let $\timeInterval \subseteq \N$ be a time interval with $\max \timeInterval=\maxTimeInterval \in \N$. Suppose Assumption \ref{assumption:labelShift} holds, Assumption \ref{assumption:smoothRegressionFunctions} holds with $\smoothnessExponentRegressionFunction \in (0,1]$, $\smoothnessConstantRegressionFunction \in [0,\infty)$ and Assumption \ref{assumption:tsybakovMarginAssumptionSequential} holds with $\tsybakovMarginExponent \in [1,\infty)$ and $\tsybakovMarginConstant \in [1,\infty)$. There exists a constant $\constantForSmoothnessWithJumpsAverageRegretBound \geq 1$, depending only on  $\tsybakovMarginExponent$, such that for all $\delta \in (0,\totalVariation(\nullDistribution,\positiveDistribution)/\constantForTotalVariationLabelProbabilitiesBound]$ and $\nMin=\min\{\numNull,\numPositive\}$ with $ \smoothnessConstantRegressionFunction\epsilonByNDeltaIteratedLogarithm{\nMin}{\delta}^\smoothnessExponentRegressionFunction \leq \left\lbrace  \totalVariation(\nullDistribution,\positiveDistribution)/\constantForTotalVariationLabelProbabilitiesBound \right\rbrace^{2\smoothnessExponentRegressionFunction+1}$  then with probability at least $1-\delta$,
\begin{align*}
\regretOverTimeInterval_{\timeInterval}(\policy)\leq \tsybakovMarginConstant\,\constantForTotalVariationLabelProbabilitiesBound \bigl\{\highProbBoundDeltaLabelledData\vee \highProbabilityUpperBoundUnlabelledForAverageRegretBound(\timeHorizon,\totalVariationLabelProbabilities)\bigr\}^{\tsybakovMarginExponent} +\constantForTotalVariationLabelProbabilitiesBound \delta.
\end{align*}
\end{corollary}

To interpret Corollary \ref{corr:totalVariationLabelProbBound} let's again view $\tsybakovMarginConstant$ and $\totalVariation(\nullDistribution,\positiveDistribution)$ as a constant and take $\holderExponentWeightFunction \in [1,2]$ and $\delta = 1/\nMin$. Then Corollary \ref{corr:totalVariationLabelProbBound} yields a bound of the form 
\begin{align*}
\E\left\lbrace \regretOverTimeInterval_{\timeInterval}(\policy) \right\rbrace \leq \tilde{O}\left(\highProbBoundDeltaLabelledData^{\tsybakovMarginExponent}+\left\lbrace \totalVariationLabelProbabilities+\timeHorizon^{-\frac{1}{2}}\right\rbrace^{\frac{\tsybakovMarginExponent}{2\holderExponentWeightFunction+1}}\,\timeHorizon^{-\frac{\tsybakovMarginExponent\holderExponentWeightFunction}{2\holderExponentWeightFunction+1}}\right).
\end{align*}

This recovers the bound of \cite{bai2022adapting} for the important special case of $\holderExponentWeightFunction=1$ and $\tsybakovMarginExponent=1$. 

\section{The proof structure for the regret bounds}\label{sec:mainIdeasForProofSection}

The goal of this section is to give a high-level overview of the proof structure for the high-probability regret bounds in Theorems \ref{thm:mainClassificationSingleTimeStep} and \ref{thm:highProbBoundForTemporalSmoothnessWithJumps} and Corollary \ref{corr:totalVariationLabelProbBound}.

\subsection{The transformed density ratio estimator}

We shall leverage our regulartiy assumption on the transformed regression function $\regressionFunction$ (Assumption \ref{assumption:smoothRegressionFunctions}) to provide a performance bound on $\estimatedRegressionFunction$ as an estimator of $\regressionFunction$.
Recall that $\nMin=\numNull\wedge\numPositive$, and let $\goodSetRegressionFunctionEstimatedWell$ denote the random subset of $\metricSpace$ consisting of all $x \in \metricSpace$ such that 
\begin{align*}
\left| \estimatedRegressionFunction(x)-\regressionFunction(x) \right|\leq 14\left(\left\lbrace 
\smoothnessConstantRegressionFunction^{\frac{1}{\smoothnessExponentRegressionFunction}}\epsilonByNDeltaIteratedLogarithm{\nMin}{\delta} \right\rbrace^{\frac{\smoothnessExponentRegressionFunction}{2\smoothnessExponentRegressionFunction+1}}\vee\sqrt{\epsilonByNDeltaIteratedLogarithm{\nMin}{\delta}}\right).
\end{align*}

\begin{prop}\label{prop:regressionFunctionHighProbBound} Suppose Assumption \ref{assumption:smoothRegressionFunctions} holds with $\smoothnessExponentRegressionFunction \in (0,1]$, we have $\smoothnessConstantRegressionFunction \in [0,\infty)$ and $\nMin \geq 2^4$. Given  $\delta \in (0,1)$, $\Prob\left( \classConditionalDistribution{1/2}\left\lbrace \metricSpace \setminus \goodSetRegressionFunctionEstimatedWell\right\rbrace > 2^5\delta \right) \leq \delta/3$.
\end{prop}

In order to prove Proposition \ref{prop:regressionFunctionHighProbBound} we first establish the validity of local confidence intervals for the class-conditional distributions in Section \ref{sec:confidenceClassConditional}. These bounds build upon local Dvoretzky--Kiefer--Wolfowitz--Massart--type concentration inequalities from \cite{reeve2024short}. Next we deduce Proposition \ref{prop:regressionFunctionHighProbBound} via a corresponding pointwise result (Lemma \ref{lemma:pointwiseRegressionFunctionHighProbBound}) in Section \ref{sec:estTransformedDensityRatio}.

\subsection{The label probability estimator}

Next, we turn to our performance bounds for estimating the label probabilities. First we consider our estimator $\estimatorMarginalDistributionLepski(f)$ of the integral $\marginalDistribution[\testTime](f)$ for an arbitrary function of the form $f:\metricSpace^{\numNull+\numPositive+1}\rightarrow [0,1]$. To state our results we introduce the following notation,
\begin{align*}
\holderConstantSpecifiedFunction&:=\holderConstantWeightFunction\big| \classConditionalDistribution{1}(f)-\classConditionalDistribution{0}(f)\big|/\{\pOfAlpha!\,\sqrt{2\holderExponentWeightFunction+1}\}\\ 
\estimatorMarginalDistributionErrorBound&:= 2^{7/2}\maximalHolderExponentWeightFunction^2(\maximalHolderExponentWeightFunction+1)^2 \, \biggl(\sqrt{\epsilonByNDeltaBase{\windowSize}{{\delta}/{\{\pi\windowSize\}^2}}}  \vee \left\lbrace \holderConstantSpecifiedFunction^{\frac{1}{2\holderExponentWeightFunction+1}}\epsilonByNDeltaBase{\windowSize}{{\delta}/{\{\pi\windowSize\}^2}}^{\frac{\holderExponentWeightFunction}{2\holderExponentWeightFunction+1}} \right\rbrace \biggr).
\end{align*}
Our bound on the estimation error of $\estimatorMarginalDistributionLepski(f)$ given the label-shift assumption (Assumption \ref{assumption:labelShift}) and our assumption on the label-probability dynamics (Assumption \ref{assumption:smoothlyVaryingLabelProbabilities}) is as follows.

\begin{prop}\label{prop:marginalDistributionEstimatorBound}  Suppose that Assumptions \ref{assumption:labelShift} and \ref{assumption:smoothlyVaryingLabelProbabilities} hold with  $\windowSize \in [\numUnlabelled]$, $\holderExponentWeightFunction \in [0,\maximalHolderExponentWeightFunction]$, $\holderConstantWeightFunction \in [0,\infty)$. Then, given a function $f:\X^{\numNull+\numPositive+1} \rightarrow [0,1]$ and $\delta \in (0,1)$ we have
\begin{align*}
\Prob\left\lbrace \left| \estimatorMarginalDistributionLepski(f)-\marginalDistribution[\testTime](f)\right|> \estimatorMarginalDistributionErrorBound \big| \nullSample, \positiveSample \right\rbrace \leq \frac{\delta}{3}.
\end{align*}
\end{prop}

Proposition \ref{prop:marginalDistributionEstimatorBound} yields the following corollary on estimating the label probabilities.

\begin{corollary}\label{corr:labelProbabilityEstimatorBound}  Suppose that we are in the setting of Proposition \ref{prop:marginalDistributionEstimatorBound} for some $f:\X^{\numNull+\numPositive+1} \rightarrow [0,1]$. For all $\delta \in (0,1)$ we have
\begin{align*}
\Prob\biggl\{ \frac{|\estimatorLabelProbabilityLepski-\labelProbability[\testTime]| |(\positiveDistribution-\nullDistribution)(f)|}{\estimatorMarginalDistributionErrorBound+3\underset{y \in \{0,1\}}{\max} | (\estimatorClassConditionalSecondSample{y}-\classConditionalDistribution{y})(f) | }>1\bigg| \nullSample,\positiveSample \biggr\} \leq \frac{\delta}{3}.
\end{align*}
\end{corollary}

The proof of both Proposition \ref{prop:marginalDistributionEstimatorBound} and Corollary \ref{corr:labelProbabilityEstimatorBound} is given in Section \ref{sec:appendixWithLemmasOnEstimatingLabelProbs}.

Corollary \ref{corr:labelProbabilityEstimatorBound} reveals hat the label probabilities can be estimated given Assumptions \ref{assumption:labelShift} and \ref{assumption:smoothlyVaryingLabelProbabilities}, along with a suitable function $f:\metricSpace^{\numNull+\numPositive+1} \rightarrow [0,1]$. In particular, the bound in Corollary \ref{corr:labelProbabilityEstimatorBound} demonstrates that accurate estimation is possible provided $|(\positiveDistribution-\nullDistribution)(f)|$ is not too small. This justifies the choice of  $\hat{f}_\delta$ defined by $\hat{f}_\delta(x):=\one\{2\estimatedRegressionFunction(x)\geq 1\}$ which corresponds to a plug-in estimate of the function attaining the total variation distance between $\nullDistribution$ and $\positiveDistribution$.

\subsection{A modular regret bound}

Next, we provide a general error bound given with respect to an arbitrary estimate of the transformed density ratio. This modular approach allows us to prove Theorems \ref{thm:mainClassificationSingleTimeStep} and \ref{thm:highProbBoundForTemporalSmoothnessWithJumps} in unified way.

Our general setting is as follows. Suppose we have a function $\hat{\regressionFunction}:\metricSpace^{\numNull+\numPositive+1} \rightarrow [0,1]$, where for each $x \in \metricSpace$, we view $\hat{\regressionFunction}(x)\equiv \hat{\regressionFunction}(\positiveCovariate_0,\ldots,\positiveCovariate_{\numPositive-1},x)$ as an estimate of $\regressionFunction(x)$. Let $\hat{f} \in \setOfDataDependentClassifiers$ be the data-dependent classifier given by $\hat{f}(x):=\one\left\lbrace \hat{\regressionFunction}(x)\geq 1\right\rbrace$. Given $\tilde{\delta} \in (0,1)$ we let $\estimatorLabelProbabilityLepskiUnspecified{\testTime}{\tilde{\delta}}{\hat{f}}$ denote the procedure for estimating $\labelProbability[\testTime]$ described in Section \ref{sec:estLabelProbability}, with $\tilde{\delta}$ in place of $\delta$, and $\hat{f}$ in place of $f$. Finally, let $\tilde{\dataDependentClassifier}_{\testTime,\tilde{\delta}}^{\hat{\regressionFunction}} \in \setOfDataDependentClassifiers$ denote the data-dependent classifier defined for $x \in \metricSpace$ by  
\begin{align*}\tilde{\dataDependentClassifier}_{\testTime,\tilde{\delta}}^{\hat{\regressionFunction}}(x) := \one\left\lbrace \hat{\regressionFunction}(x)+ \estimatorLabelProbabilityLepskiUnspecified{\testTime}{\tilde{\delta}}{\hat{f}}>1\right\rbrace.
\end{align*}

The following result (Proposition \ref{prop:keyTechnicalLemmaForProofOfMainClassificationSingleTimeStep}) gives a bound on the excess classification error provided $\hat{\regressionFunction}$ and $\estimatorClassConditionalSecondSample{y}$ are sufficiently close to $\regressionFunction$ and  $\classConditionalDistribution{y}(\hat{f})$, respectively.

\begin{prop}\label{prop:keyTechnicalLemmaForProofOfMainClassificationSingleTimeStep}
Suppose that Assumption \ref{assumption:labelShift} holds, Assumption \ref{assumption:smoothlyVaryingLabelProbabilities} hold with  $\windowSize \in [\numUnlabelled]$, $\holderExponentWeightFunction \in [0,\maximalHolderExponentWeightFunction]$, $\holderConstantWeightFunction \in [0,\infty)$ and Assumption \ref{assumption:tsybakovMarginAssumption} with $\tsybakovMarginFunction:(0,\infty) \rightarrow (0,1]$. Suppose $\Delta_a,\, \Delta_b,\, \Delta_c >0$ such that $2\Delta_a+\Delta_b\leq \totalVariation(\nullDistribution,\positiveDistribution)/4$ and let $\eventForCentralTechnicalLemma \equiv \eventForCentralTechnicalLemma(\Delta_a,\Delta_b,\Delta_c)$ denote the event
\begin{align*}
\eventForCentralTechnicalLemma:=&\left\lbrace\classConditionalDistribution{1/2}\left(\left\lbrace x \in \X : |\hat{\regressionFunction}(x)-\regressionFunction(x)|>\Delta_a \right\rbrace\right) \leq \Delta_b \right\rbrace \cap \left\lbrace \max_{y \in \{0,1\}} | (\estimatorClassConditionalSecondSample{y}-\classConditionalDistribution{y})(\hat{f}) | \leq \Delta_c \right\rbrace.
\end{align*}
Given $\tilde{\delta} \in (0,1)$ let 
\[\Delta_{\testTime}(\tilde{\delta}):= \highProbBoundSingleTimeStepDeltaUnlabelledData[\tilde{\delta}] \vee \left(\frac{\Delta_a}{16}\right) \vee \left(\frac{\Delta_c}{2\totalVariation(\nullDistribution,\positiveDistribution)}\right),\] and $
\constantForMainClassificationSingleTimeStep :=  4 \left(57\maximalHolderExponentWeightFunction^2(\maximalHolderExponentWeightFunction+1)^2+7\right)$. Then, we have
\begin{align*}
\Prob\left\lbrace\frac{ \testError_{\testTime}(\tilde{\dataDependentClassifier}_{\testTime,\tilde{\delta}}^{\hat{\regressionFunction}})-\testError_\testTime(\bayesClassifier[\testTime])}{\tsybakovMarginFunction\left(\constantForMainClassificationSingleTimeStep\Delta_{\testTime}(\tilde{\delta})\right) +\Delta_b }>2\,\Bigg|\, \eventForCentralTechnicalLemma \right\rbrace \leq \frac{\tilde{\delta}}{3}.  
\end{align*}
\end{prop}

The proof of this modular bound hinges upon Corollary \ref{corr:labelProbabilityEstimatorBound} and is given in Section \ref{sec:proofOfErrorBoundsOneStep}.

\subsection{Completing the error bounds}

Theorem \ref{thm:mainClassificationSingleTimeStep} follows as a relatively straightforward consequence of Proposition \ref{prop:regressionFunctionHighProbBound} and Proposition \ref{prop:keyTechnicalLemmaForProofOfMainClassificationSingleTimeStep}, and is provided in Section \ref{sec:proofOfErrorBoundsOneStep}. The proof of Theorem \ref{thm:highProbBoundForTemporalSmoothnessWithJumps} which bounds the average regret in a setting with smoothly varying label probabilities between sporadic jumps is given in Section \ref{sec:proofOfAverageRegretBounds} and also leverages  Propositions \ref{prop:regressionFunctionHighProbBound} and \ref{prop:keyTechnicalLemmaForProofOfMainClassificationSingleTimeStep}, along with a convexity argument. Corollary \ref{corr:totalVariationLabelProbBound} is then deduced from Theorem \ref{thm:highProbBoundForTemporalSmoothnessWithJumps} by demonstrating that bounds on the quantity $\totalVariationLabelProbabilities$ imply that the the time interval $\timeInterval$ may be decomposed into a relatively small number of sub-intervals over which the label-probabilities are contained in a small region.

\section{Confidence intervals for the class-conditional distributions}\label{sec:confidenceClassConditional}

Recall the definitions of $\empiricalUncertaintyFlat$, $\epsilonByNDeltaIteratedLogarithm{\sampleSizeVariable}{\delta}$ and $\empiricalUncertainty$ from Section \ref{sec:estRegressionFunction}. Furthermore, given $p \in [0,1]$ and $\tilde{\varepsilon} >0$ we define
\begin{align*}
\populationUncertaintyFlat(p, \tilde{\varepsilon})&:=\frac{8 \bigl( \sqrt{9 \varepsilon \sigma^2(p) +  \tilde{\varepsilon}^2}+\{1-2p\} \tilde{\varepsilon}\bigr)}{9+2  \tilde{\varepsilon}}.
\end{align*} For $p \in [0,1]$, $\sampleSizeVariable \in \N$, $\delta \in (0,1]$ with $\log(\sampleSizeVariable)\geq e\delta$,
\begin{align*}
\populationUncertainty(p,\sampleSizeVariable,\delta)&:=\populationUncertaintyFlat(p,\epsilonByNDeltaIteratedLogarithm{\sampleSizeVariable}{\delta}) \vee (1/\sampleSizeVariable),\\
\populationCI(p,\sampleSizeVariable,\delta)&:=\bigl[ p-\populationUncertainty(1-p,\sampleSizeVariable,\delta), p+\populationUncertainty(p,\sampleSizeVariable,\delta)\bigr],\\
\empiricalCI(q,\sampleSizeVariable,\delta)&:=\bigl[ q-\empiricalUncertainty(1-q,\sampleSizeVariable,\delta), q+\empiricalUncertainty(q,\sampleSizeVariable,\delta)\bigr].
\end{align*}

Given $x \in \metricSpace$ we let $\classOfMetricBalls(x):=\bigl\{ \closedMetricBall{x}{r} \bigr\}_{r\geq 0}$ and for each $y \in \{0,1\}$ we define events
\begin{align*}
\eventCI{y}(x,\delta)&:= \bigcap_{B \in \classOfMetricBalls(x)}\biggl\{ \empiricalClassConditionalDistribution{y}(B) \in \populationCI\bigl( \classConditionalDistribution{y}(B),\numClassConditional{y},\delta\bigr) \biggr\},\\
\empiricalEventCI{y}(x,\delta)&:= \bigcap_{B \in \classOfMetricBalls(x)}\biggl\{ \classConditionalDistribution{y}(B) \in \empiricalCI\bigl( \empiricalClassConditionalDistribution{y}(B),\numClassConditional{y},\delta\bigr) \biggr\}.
\end{align*}

\begin{lemma}\label{lemma:technicalLemmaForPopulationCIEventImpliesEmpiricalCIEventLemma} Suppose that $\delta \in (0,1)$, $\sampleSizeVariable \in \N$ and $q, p \in [0,1]$ satisfy $\log(\sampleSizeVariable)\geq e\delta$ and $q -p \leq \populationUncertainty(p,\sampleSizeVariable,\delta)$. Then $q-p \leq \empiricalUncertainty(1-q,\sampleSizeVariable,\delta)$.
\end{lemma}
\begin{proof} 
Since $q -p \leq \populationUncertainty(p,\sampleSizeVariable,\modDelta)$ we either have  $q-p \leq 1/\sampleSizeVariable$ or $q-p> 1/\sampleSizeVariable$. In the former case the consequence $q-p \leq 1/\sampleSizeVariable \leq \empiricalUncertainty(1-q,\sampleSizeVariable,\delta)$ is immediate. Henceforth, we shall focus on the latter case. Since  $1/\sampleSizeVariable <q -p \leq \populationUncertainty(p,\sampleSizeVariable,\modDelta)=\max\{1/\sampleSizeVariable,\populationUncertaintyFlat(p,\epsilonByNDeltaIteratedLogarithm{\sampleSizeVariable}{\delta})\}$ we have $q-p \leq \populationUncertaintyFlat(p,\epsilonByNDeltaIteratedLogarithm{\sampleSizeVariable}{\delta})$. By rearranging it follows that $q -p\leq\empiricalUncertaintyFlat(1-q,\epsilonByNDeltaIteratedLogarithm{\sampleSizeVariable}{\delta})= \empiricalUncertainty(1-q,\sampleSizeVariable,\delta)$.    
\end{proof}

\begin{lemma}\label{lemma:populationCIEventImpliesEmpiricalCIEvent} Given $\delta \in (0,1)$, $y \in \{0,1\}$ and $\numClassConditional{y} \in \N$ satisfy $\log(\numClassConditional{y})\geq e\delta$. Then $\eventCI{y}(x,\delta) \subseteq \empiricalEventCI{y}(x,\delta)$ for all $x \in \metricSpace$.
\end{lemma}
\begin{proof} Suppose that $\eventCI{y}(x,\delta)$ holds for some $x \in \metricSpace$ and $y \in \{0,1\}$, and take $B \in \classOfMetricBalls(x)$ so that $\empiricalClassConditionalDistribution{y}(B) \in \populationCI( \classConditionalDistribution{y}(B),\numClassConditional{y},\delta)$. Hence, we have
\begin{align*}
 \empiricalClassConditionalDistribution{y}(B) -\classConditionalDistribution{y}(B)\leq \populationUncertainty(\classConditionalDistribution{y}(B),\numClassConditional{y},\delta),
\end{align*}
so that applying Lemma \ref{lemma:technicalLemmaForPopulationCIEventImpliesEmpiricalCIEventLemma} with $p=\classConditionalDistribution{y}(B)$ and $q=\empiricalClassConditionalDistribution{y}(B)$ yields
\begin{align*}
 \empiricalClassConditionalDistribution{y}(B) -\classConditionalDistribution{y}(B)\leq \empiricalUncertainty(1-\empiricalClassConditionalDistribution{y}(B),\numClassConditional{y},\delta).
\end{align*}
Similarly, by $\empiricalClassConditionalDistribution{y}(B) \in \populationCI( \classConditionalDistribution{y}(B),\numClassConditional{y},\delta)$ we also have 
\begin{align*}
\bigl\{1- \empiricalClassConditionalDistribution{y}(B)\bigr\}-\bigl\{1-\classConditionalDistribution{y}(B)\bigr\} \leq \populationUncertainty(1-\classConditionalDistribution{y}(B),\numClassConditional{y},\delta),
\end{align*}
so that applying Lemma \ref{lemma:technicalLemmaForPopulationCIEventImpliesEmpiricalCIEventLemma} with $p=1-\classConditionalDistribution{y}(B)$ and $q=1-\empiricalClassConditionalDistribution{y}(B)$ yields $\classConditionalDistribution{y}(B)- \empiricalClassConditionalDistribution{y}(B)\leq \empiricalUncertainty(\empiricalClassConditionalDistribution{y}(B),\numClassConditional{y},\delta)$. Thus, we have  $\classConditionalDistribution{y}(B) \in \empiricalCI( \empiricalClassConditionalDistribution{y}(B),\numClassConditional{y},\delta)$. Since this holds for all $B \in \classOfMetricBalls(x)$ we have  $\eventCI{y}(x,\delta) \subseteq \empiricalEventCI{y}(x,\delta)$.
\end{proof}

\begin{lemma}\label{lemma:uniformConcentrationMetricBallsSinglePoint} Suppose $y \in \{0,1\}$, $\numClassConditional{y} \in \N$, $\delta \in (0,1]$ satisfy $\log(\numClassConditional{y})\geq e\delta$ and let  $\modDeltaClassConditional :=  4 \lceil \log_{4/3}(\numClassConditional{y}) \rceil \{\delta/\log(\numClassConditional{y})\}^2$. Then, for each $x \in \metricSpace$, we have $\Prob\bigl\{\eventCI{y}(x,\delta)\bigr\} \geq 1-\modDeltaClassConditional$.
\end{lemma}
\begin{proof} The definition of $\modDeltaClassConditional$ entails
\begin{align*}
\epsilonByNDeltaIteratedLogarithm{\numClassConditional{y}}{\delta}&= \log(\log(\numClassConditional{y})/\delta)/\numClassConditional{y} = \log\bigl( 4 \lceil \log_{4/3}(\numClassConditional{y}) \rceil/\modDeltaClassConditional  \bigr)/(2\numClassConditional{y}).
\end{align*}
The result now follows from two applications of \cite[Corollary 3]{reeve2024short} combined with a union bound.
\end{proof}

\begin{lemma}\label{lemma:empiricalCIWidthBound}  Suppose $y \in \{0,1\}$, $\delta \in (0,1]$, $x \in \metricSpace$ and $\numClassConditional{y} \in \N$ satisfy $\log(\numClassConditional{y})\geq e\delta$. Then, for all $B \in \classOfMetricBalls(x)$ with $p:=\min\{ \classConditionalDistribution{y}(B),1-\classConditionalDistribution{y}(B)\}$ we have
\begin{align*}
\Lebesgue\bigl\{ \empiricalCI\bigl( \empiricalClassConditionalDistribution{y}(B),\numClassConditional{y},\delta\bigr)\bigr\} &\leq \frac{16}{3}\sqrt{p\epsilonByNDeltaIteratedLogarithm{\numClassConditional{y}}{\delta}}+20\epsilonByNDeltaIteratedLogarithm{\numClassConditional{y}}{\delta}.
\end{align*}
\end{lemma}
\begin{proof} Let's suppose that the event $\eventCI{y}(x,\delta)$ holds, take $B \in \classOfMetricBalls(x)$ and let $p:=\min\{ \classConditionalDistribution{y}(B),1-\classConditionalDistribution{y}(B)\}$. On the event $\eventCI{y}(x,\delta)$ we then have
\begin{align*}
q&:=\min\{ \empiricalClassConditionalDistribution{y}(B),1-\empiricalClassConditionalDistribution{y}(B)\} \leq p+\populationUncertaintyFlat(p,\epsilonByNDeltaIteratedLogarithm{\numClassConditional{y}}{\delta})\vee \biggl(\frac{1}{\numClassConditional{y}}\biggr) \\ & \leq p + 8\sqrt{p  \epsilonByNDeltaIteratedLogarithm{\numClassConditional{y}}{\delta}}/3+16\epsilonByNDeltaIteratedLogarithm{\numClassConditional{y}}{\delta}/9 = \biggl\{\sqrt{p}+4\sqrt{  \epsilonByNDeltaIteratedLogarithm{\numClassConditional{y}}{\delta}}/3\biggr\}^2.
\end{align*}
Hence, we have 
\begin{align*}
\Lebesgue\bigl\{ \empiricalCI\bigl( \empiricalClassConditionalDistribution{y}(B),\numClassConditional{y},\delta\bigr)\bigr\} &\leq \frac{16 \bigl(  \sqrt{ \epsilonByNDeltaIteratedLogarithm{\numClassConditional{y}}{\delta}\sigma^2(q) + \epsilonByNDeltaIteratedLogarithm{\numClassConditional{y}}{\delta}^2}}{3\{1+2\epsilonByNDeltaIteratedLogarithm{\numClassConditional{y}}{\delta}\} }+\frac{2}{\numClassConditional{y}}\\
& \leq \frac{16}{3}\left\lbrace \sqrt{q  \epsilonByNDeltaIteratedLogarithm{\numClassConditional{y}}{\delta} }+ \epsilonByNDeltaIteratedLogarithm{\numClassConditional{y}}{\delta} \right\rbrace+\frac{2}{\numClassConditional{y}}
 \leq \frac{16}{3} \sqrt{ p \epsilonByNDeltaIteratedLogarithm{\numClassConditional{y}}{\delta} }+20 \epsilonByNDeltaIteratedLogarithm{\numClassConditional{y}}{\delta}.
\end{align*}
\end{proof}

\section{Estimating the transformed density ratio}\label{sec:estTransformedDensityRatio}

Our primary goal in this section is to establish the following pointwise result (Lemma \ref{lemma:pointwiseRegressionFunctionHighProbBound}). Proposition \ref{prop:regressionFunctionHighProbBound} then follows straightforwardly.

\begin{lemma}\label{lemma:pointwiseRegressionFunctionHighProbBound} Suppose that Assumption \ref{assumption:smoothRegressionFunctions} holds with $\smoothnessExponentRegressionFunction \in (0,1]$, $\smoothnessConstantRegressionFunction \in [0,\infty)$. Given $x \in \metricSpace$, $\nMin \in \N$ with $\nMin \geq 2^4$ and $\delta \in (0,1]$ we have $\Prob\left\lbrace x \notin \goodSetRegressionFunctionEstimatedWell\right\rbrace \leq 21 \delta^2/2$.
\end{lemma}

Before proving Lemma \ref{lemma:pointwiseRegressionFunctionHighProbBound} we shall first prove Lemmas \ref{lemma:regressionFunctionConfidenceIntervalEvent} and \ref{lemma:boundOnWidthOfRegressionFunctionConfidenceIntervals}.

\begin{lemma}\label{lemma:regressionFunctionConfidenceIntervalEvent} Suppose that $x \in \metricSpace$, $\delta \in (0,1)$ and $\log(\nMin)\geq e\delta$. Then on the event $ \empiricalEventCI{0}(x,\delta)\cap \empiricalEventCI{1}(x,\delta)$, for all $B \in \classOfMetricBalls(x)$, we have $\empiricalRegressionFunctionConfidenceBound{-1}{\delta}(B) \leq \averageRegressionFunction(B) \leq \empiricalRegressionFunctionConfidenceBound{1}{\delta}(B)$.
\end{lemma}
\begin{proof} Take $B \in \classOfMetricBalls(x)$. On the event $\empiricalEventCI{0}(x,\delta)\cap \empiricalEventCI{1}(x,\delta)$ we have 
\begin{align*}
\classConditionalDistribution{y}(B) \in \empiricalCI\bigl( \empiricalClassConditionalDistribution{y}(B),\numClassConditional{y}\bigr)=\left[ \empiricalClassConditionalDistributionConfidenceBound{y}{-1}{\delta}(B),\empiricalClassConditionalDistributionConfidenceBound{y}{1}{\delta}(B)\right],
\end{align*}
for $y \in \{0,1\}$. Rearranging, we deduce that 
\begin{align*}
\frac{\empiricalClassConditionalDistributionConfidenceBound{1}{-1}{\delta}(B)}{\empiricalClassConditionalDistributionConfidenceBound{0}{1}{\delta}(B)+\empiricalClassConditionalDistributionConfidenceBound{1}{-1}{\delta}(B)}& \leq \frac{\classConditionalDistribution{1}(B)}{\classConditionalDistribution{0}(B)+\classConditionalDistribution{1}(B)}= \averageRegressionFunction(B) \leq \frac{\empiricalClassConditionalDistributionConfidenceBound{1}{1}{\delta}(B)}{\empiricalClassConditionalDistributionConfidenceBound{0}{-1}{\delta}(B)+\empiricalClassConditionalDistributionConfidenceBound{1}{1}{\delta}(B)}.
\end{align*}
Noting also that $0 \leq \averageRegressionFunction(B) \leq 1$, the lemma follows from the definition of $\empiricalRegressionFunctionConfidenceBound{\confidenceSign}{\delta}(B)$.
\end{proof}

\begin{lemma}\label{lemma:boundOnWidthOfRegressionFunctionConfidenceIntervals}
Take $\delta \in (0,1]$ with $\log(\nMin)\geq e\delta$ and $x \in \metricSpace$. On the event  $\eventCI{0}(x,\delta)\cap \eventCI{1}(x,\delta)$, the following bound holds for all $B \in \classOfMetricBalls(x)$, 
\[ \empiricalRegressionFunctionConfidenceBound{1}{\delta}(B)-\empiricalRegressionFunctionConfidenceBound{-1}{\delta}(B) \leq 10\sqrt{\frac{\epsilonByNDeltaIteratedLogarithm{\nMin}{\delta}}{\classConditionalDistribution{1/2}(B)}}.
\]
\end{lemma}
\begin{proof} Suppose that the event $\eventCI{0}(x,\delta)\cap \eventCI{1}(x,\delta)$ holds and fix some $B \in \classOfMetricBalls(x)$. As such, the event $\empiricalEventCI{0}(x,\delta)\cap \empiricalEventCI{1}(x,\delta)$ also holds by Lemma \ref{lemma:populationCIEventImpliesEmpiricalCIEvent}. In particular, it follows that $\classConditionalDistribution{1/2}(B) \leq \empiricalClassConditionalDistributionConfidenceBound{1/2}{1}{\delta}(B)$. Now suppose that $\classConditionalDistribution{1/2}(B)\geq 100 \epsilonByNDeltaIteratedLogarithm{\nMin}{\delta}$ so  that
\begin{align*}
\bernsteinTypeErrorOnSingleBall&:=\frac{16}{3}\sqrt{\{\nullDistribution(B)\vee\positiveDistribution(B)\}\epsilonByNDeltaIteratedLogarithm{\nMin}{\delta}}+20\epsilonByNDeltaIteratedLogarithm{\nMin}{\delta} \\
&\leq 16\sqrt{2\classConditionalDistribution{1/2}(B)\epsilonByNDeltaIteratedLogarithm{\nMin}{\delta}}/3+20\epsilonByNDeltaIteratedLogarithm{\nMin}{\delta}\leq \min\left\lbrace \classConditionalDistribution{1/2}(B), 10 \sqrt{\classConditionalDistribution{1/2}(B) \epsilonByNDeltaIteratedLogarithm{\nMin}{\delta}}\right\rbrace.
\end{align*}
For the purpose of the proof we shall let $\empiricalClassConditionalDistributionConfidenceBound{y}{\confidenceSign}{\delta}$ denote $\empiricalClassConditionalDistributionConfidenceBound{y}{\confidenceSign}{\delta}(B)$ for $y \in \{0,1/2,1\}$ and $\confidenceSign \in \{-1,1\}$. By Lemma \ref{lemma:empiricalCIWidthBound} we have $0 \leq \empiricalClassConditionalDistributionConfidenceBound{y}{1}{\delta}-\empiricalClassConditionalDistributionConfidenceBound{y}{-1}{\delta} \leq \bernsteinTypeErrorOnSingleBall$, where we have used the fact that $\nMin=\min\{\numClassConditional{0},\numClassConditional{1}\} \in \N$, so that $\epsilonByNDeltaIteratedLogarithm{\nMin}{\delta} \geq \epsilonByNDeltaIteratedLogarithm{\numClassConditional{0}}{\delta}\vee\epsilonByNDeltaIteratedLogarithm{\numClassConditional{1}}{\delta}$. Hence, provided $\classConditionalDistribution{1/2}(B)\geq 100 \epsilonByNDeltaIteratedLogarithm{\nMin}{\delta}$ we have
\begin{align*}
\empiricalRegressionFunctionConfidenceBound{1}{\delta}(B)-\empiricalRegressionFunctionConfidenceBound{-1}{\delta}(B)
&\leq  \frac{\empiricalClassConditionalDistributionConfidenceBound{1}{1}{\delta}}{\empiricalClassConditionalDistributionConfidenceBound{0}{-1}{\delta}+\empiricalClassConditionalDistributionConfidenceBound{1}{1}{\delta}}-\frac{\empiricalClassConditionalDistributionConfidenceBound{1}{-1}{\delta}}{\empiricalClassConditionalDistributionConfidenceBound{0}{1}{\delta}+\empiricalClassConditionalDistributionConfidenceBound{1}{-1}{\delta}}\\
& = \frac{\empiricalClassConditionalDistributionConfidenceBound{1}{1}{\delta}\empiricalClassConditionalDistributionConfidenceBound{0}{1}{\delta}-\empiricalClassConditionalDistributionConfidenceBound{1}{-1}{\delta}\empiricalClassConditionalDistributionConfidenceBound{0}{-1}{\delta}}{\{\empiricalClassConditionalDistributionConfidenceBound{0}{-1}{\delta}+\empiricalClassConditionalDistributionConfidenceBound{1}{1}{\delta}\}\{\empiricalClassConditionalDistributionConfidenceBound{0}{1}{\delta}+\empiricalClassConditionalDistributionConfidenceBound{1}{-1}{\delta}\}}\\
& \leq \frac{\empiricalClassConditionalDistributionConfidenceBound{1}{1}{\delta}\empiricalClassConditionalDistributionConfidenceBound{0}{1}{\delta}-\left(\empiricalClassConditionalDistributionConfidenceBound{1}{1}{\delta}-\bernsteinTypeErrorOnSingleBall \right)\left(\empiricalClassConditionalDistributionConfidenceBound{0}{1}{\delta}-\bernsteinTypeErrorOnSingleBall \right)}{\bigl(\empiricalClassConditionalDistributionConfidenceBound{0}{1}{\delta}+\empiricalClassConditionalDistributionConfidenceBound{1}{1}{\delta}-\bernsteinTypeErrorOnSingleBall\bigr)^2}\\
& = \frac{\empiricalClassConditionalDistributionConfidenceBound{1/2}{1}{\delta} \bernsteinTypeErrorOnSingleBall }{\bigl(2\empiricalClassConditionalDistributionConfidenceBound{1/2}{1}{\delta}-\bernsteinTypeErrorOnSingleBall\bigr)^2} \leq  \frac{ \bernsteinTypeErrorOnSingleBall }{\empiricalClassConditionalDistributionConfidenceBound{1/2}{1}{\delta}} \leq 10\sqrt{\frac{\epsilonByNDeltaIteratedLogarithm{\nMin}{\delta}}{\classConditionalDistribution{1/2}(B)}}.
\end{align*}
On the other hand, if $\classConditionalDistribution{1/2}(B)< 100 \epsilonByNDeltaIteratedLogarithm{\nMin}{\delta}$ then
\begin{align*}
\empiricalRegressionFunctionConfidenceBound{1}{\delta}(B)-\empiricalRegressionFunctionConfidenceBound{-1}{\delta}(B)\leq 1 < 10\sqrt{\frac{\epsilonByNDeltaIteratedLogarithm{\nMin}{\delta}}{\classConditionalDistribution{1/2}(B)}}.
\end{align*}
\end{proof}

We now complete the proof of Lemma \ref{lemma:pointwiseRegressionFunctionHighProbBound}.

\begin{proof}[Proof of Lemma \ref{lemma:pointwiseRegressionFunctionHighProbBound}] We shall show that $x \in \goodSetRegressionFunctionEstimatedWell$ for each $x \in \metricSpace$ such that the event $\eventCI{0}(x,\delta)\cap \eventCI{1}(x,\delta)$ holds. As such, we begin by choosing $x \in \metricSpace$ such that $\eventCI{0}(x,\delta)\cap \eventCI{1}(x,\delta)$ holds. Then, by Lemma \ref{lemma:populationCIEventImpliesEmpiricalCIEvent}, the event $\empiricalEventCI{0}(x,\delta)\cap \empiricalEventCI{1}(x,\delta)$ also holds. Now choose $\oracleProbability<1$ and let
\begin{align*}
\oracleRadius:=\inf_{r>0}\left\lbrace  \classConditionalDistribution{1/2}\bigl(\openMetricBall{x}{\oracleRadius}\bigr) \geq \oracleProbability \right\rbrace. 
\end{align*}
First suppose there exists $\tilde{r}_0,\,\tilde{r}_1\in [0,\oracleRadius]$ with 
\begin{align}\label{eq:lemmaProofIntervalDisjointCondition}
\regressionFunctionInterval{\delta}\bigl\{\closedMetricBall{x}{\tilde{r}_0}\bigr\}\cap \regressionFunctionInterval{\delta}\bigl\{\closedMetricBall{x}{\tilde{r}_1}\bigr\}=\emptyset.
\end{align}
Note that the function $r\mapsto \regressionFunctionInterval{\delta}\bigl\{\closedMetricBall{x}{r}\bigr\}$ is piece wise constant with discontinuities restricted to the discrete set $\setOfDistances{x}$. Hence, there exist $r_0,\, r_1\in [0,\oracleRadius] \cap \setOfDistances{x}$ with $r_0+r_1$ minimal and
\begin{align}\label{eq:lemmaProofIntervalDisjointConditionInequality}
\sup \regressionFunctionInterval{\delta}\bigl\{\closedMetricBall{x}{r_0}\bigr\}< \inf \regressionFunctionInterval{\delta}\bigl\{\closedMetricBall{x}{r_1}\bigr\}.
\end{align}
Note also that for all $r \leq \oracleRadius$ we have $\classConditionalDistribution{1/2}\bigl(\openMetricBall{x}{r}\bigr) \leq \oracleProbability$. Hence, by Assumption \ref{assumption:smoothRegressionFunctions} for all $r \in  [0,\oracleRadius]$ we have
\begin{align}
\bigl|\regressionFunction(x)-\averageRegressionFunction\bigl\{\closedMetricBall{x}{r}\bigr\} \bigr|  \leq \smoothnessConstantRegressionFunction \oracleProbability^\smoothnessExponentRegressionFunction. \label{eq:applySmoothnessOfRegFunc}
\end{align}
Consequently, since $\{r_0,r_1\} \subseteq [0,\oracleRadius]$ from Lemma \ref{lemma:regressionFunctionConfidenceIntervalEvent} we deduce
\begin{align*}
\empiricalRegressionFunctionConfidenceBound{-1}{\delta}\bigl\{\closedMetricBall{x}{r_1}\bigr\}&\leq \averageRegressionFunction\bigl\{\closedMetricBall{x}{r_1}\bigr\} \leq \regressionFunction(x)+ \smoothnessConstantRegressionFunction \oracleProbability^\smoothnessExponentRegressionFunction\\
& \leq \averageRegressionFunction\bigl\{\closedMetricBall{x}{r_0}\bigr\}+2\smoothnessConstantRegressionFunction \oracleProbability^\smoothnessExponentRegressionFunction \leq \empiricalRegressionFunctionConfidenceBound{1}{\delta}\bigl\{\closedMetricBall{x}{r_0}\bigr\} +2\smoothnessConstantRegressionFunction \oracleProbability^\smoothnessExponentRegressionFunction.
\end{align*}
Noting that $\lepskiRadius{x_0}{\delta} =\min\{r_0,r_1\}$ and combining the above bound with 
\eqref{eq:lemmaProofIntervalDisjointConditionInequality} yields
\begin{align*}
\regressionFunctionIntervalWidth{\delta}\left(\closedMetricBall{x}{\lepskiRadius{x_0}{\delta}}\right) & \leq \max\left\lbrace  \regressionFunctionIntervalWidth{\delta}\left(\closedMetricBall{x}{r_0}\right),\regressionFunctionIntervalWidth{\delta}\left(\closedMetricBall{x}{r_1}\right)\right\rbrace \\
& < 2 \left(\empiricalRegressionFunctionConfidenceBound{-1}{\delta}\bigl\{\closedMetricBall{x}{r_1}\bigr\}-\empiricalRegressionFunctionConfidenceBound{1}{\delta}\bigl\{\closedMetricBall{x}{r_0}\bigr\}\right) \leq 4\smoothnessConstantRegressionFunction \oracleProbability^\smoothnessExponentRegressionFunction.
\end{align*}
Since $\lepskiRadius{x}{\delta} \leq \oracleRadius$ we deduce from Lemma \ref{lemma:regressionFunctionConfidenceIntervalEvent} that with $\hat{B}:=\closedMetricBall{x}{\lepskiRadius{x}{\delta}}$ and $\estimatedRegressionFunction(x)=\empiricalRegressionFunctionConfidenceBoundMid{\delta}(\hat{B})$ we have
\begin{align*}
\left| \estimatedRegressionFunction(x)-\regressionFunction(x)\right|&\leq \left|\empiricalRegressionFunctionConfidenceBoundMid{\delta}(\hat{B})-\averageRegressionFunction(\hat{B})\right| +\left| \averageRegressionFunction(\hat{B})-\regressionFunction(x)\right| \leq \regressionFunctionIntervalWidth{\delta}(\hat{B}) + \smoothnessConstantRegressionFunction \oracleProbability^\smoothnessExponentRegressionFunction \leq 5\smoothnessConstantRegressionFunction \oracleProbability^\smoothnessExponentRegressionFunction.
\end{align*}
On the other hand, suppose \eqref{eq:lemmaProofIntervalDisjointCondition} does not hold. Then, with $\tilde{r}:=\max\bigl(\setOfDistances{x}\cap [0,\oracleRadius]\bigr)$ and $\tilde{B}:=\closedMetricBall{x}{\tilde{r}}$ we have $\tilde{r} \leq \lepskiRadius{x}{\delta}$. Note also that $\regressionFunctionInterval{\delta}\bigl\{\closedMetricBall{x}{\tilde{r}}\bigr\}=\regressionFunctionInterval{\delta}\bigl\{\closedMetricBall{x}{\oracleRadius}\bigr\}$ since $r\mapsto \regressionFunctionInterval{\delta}\bigl\{\closedMetricBall{x}{r}\bigr\}$ is piece wise constant with discontinuities restricted to $\setOfDistances{x}$. By applying Lemma \ref{lemma:regressionFunctionConfidenceIntervalEvent} once again, combined with Assumption \ref{assumption:smoothRegressionFunctions} we have 
\begin{align*}
\bigl|\regressionFunction(x)-\averageRegressionFunction\bigl\{\closedMetricBall{x}{\tilde{r}}\bigr\} \bigr| & \leq \smoothnessConstantRegressionFunction \classConditionalDistribution{1/2}\bigl(\openMetricBall{x}{r}\bigr)^\smoothnessExponentRegressionFunction\leq \smoothnessConstantRegressionFunction \oracleProbability^\smoothnessExponentRegressionFunction.
\end{align*}
From the definition of $\oracleRadius$ we have the lower bound $\classConditionalDistribution{1/2}\bigl(\closedMetricBall{x}{r}\bigr) \geq \oracleProbability$, for all $r \geq \oracleRadius$. Hence, by Lemma  \ref{lemma:boundOnWidthOfRegressionFunctionConfidenceIntervals} with $\hat{r}:=\lepskiRadius{x}{\delta}$ we have, 
\begin{align*}
\max_{r \in \{\tilde{r},\hat{r}\} } \regressionFunctionIntervalWidth{\delta}\bigl\{\closedMetricBall{x}{r}\bigr\} & =
\max_{r \in \{\oracleRadius, \hat{r}\} } \regressionFunctionIntervalWidth{\delta}\bigl\{\closedMetricBall{x}{r}\bigr\}   \leq  \max_{r \in \{\oracleRadius, \hat{r}\} } 
 10\sqrt{\frac{\epsilonByNDeltaIteratedLogarithm{\nMin}{\delta}}{\classConditionalDistribution{1/2}\bigl\{\closedMetricBall{x}{r}\bigr\}}}
  \leq 2 \sqrt{\frac{5\epsilonByNDeltaIteratedLogarithm{\nMin}{\delta}}{\oracleProbability}}.
\end{align*}
Thus, applying Lemma \ref{lemma:regressionFunctionConfidenceIntervalEvent} once again, combined with \eqref{eq:applySmoothnessOfRegFunc} and $\lepskiRadius{x}{\delta}\geq \tilde{r}$ so that $\regressionFunctionInterval{\delta}(\tilde{B})\cap \regressionFunctionInterval{\delta}(\hat{B}) \neq \emptyset$ we have
\begin{align*}
 \left| \estimatedRegressionFunction(x)-\regressionFunction(x) \right|& = \left| \empiricalRegressionFunctionConfidenceBoundMid{\delta}(\hat{B})-\regressionFunction(x)\right|\\ &
 \leq \inf_{q \in  \regressionFunctionInterval{\delta}(\hat{B}) }\left| q-\regressionFunction(x)\right|+\regressionFunctionIntervalWidth{\delta}(\hat{B})\\
& \leq \sup_{q \in  \regressionFunctionInterval{\delta}(\tilde{B}) }\left| q-\regressionFunction(x)\right|+\regressionFunctionIntervalWidth{\delta}(\hat{B})\\
& \leq \max_{ \confidenceSign \in \{-1,1\}}\left| \empiricalRegressionFunctionConfidenceBound{y}{\delta}(\tilde{B})-\regressionFunction(x)\right|+\frac{1}{2} \regressionFunctionIntervalWidth{\delta}(\tilde{B})+\regressionFunctionIntervalWidth{\delta}(\hat{B})\\
& \leq \left| \averageRegressionFunction(\tilde{B})-\regressionFunction(x)\right|+\frac{3}{2} \regressionFunctionIntervalWidth{\delta}(\tilde{B})+\regressionFunctionIntervalWidth{\delta}(\hat{B}) \leq  \smoothnessConstantRegressionFunction \oracleProbability^\smoothnessExponentRegressionFunction+5\sqrt{\frac{5\epsilonByNDeltaIteratedLogarithm{\nMin}{\delta}}{\oracleProbability}}.
\end{align*}
Combining these two cases and noting that $\oracleProbability<1$ was arbitrary we see that the event $\eventCI{0}(x,\delta)\cap \eventCI{1}(x,\delta)$,
\begin{align*}
\left| \estimatedRegressionFunction(x)-\regressionFunction(x) \right|
&\leq 6 \inf_{\oracleProbability<1}\max\left\lbrace \smoothnessConstantRegressionFunction \oracleProbability^\smoothnessExponentRegressionFunction, \sqrt{{5\epsilonByNDeltaIteratedLogarithm{\nMin}{\delta}}/{\oracleProbability}}\right\rbrace = 6\max\left\lbrace \smoothnessConstantRegressionFunction^{\frac{1}{2\smoothnessExponentRegressionFunction+1}}\left\lbrace 25\epsilonByNDeltaIteratedLogarithm{\nMin}{\delta}\right\rbrace^{\frac{\smoothnessExponentRegressionFunction}{2\smoothnessExponentRegressionFunction+1}} ,\sqrt{5\epsilonByNDeltaIteratedLogarithm{\nMin}{\delta}}\right\rbrace,
\end{align*}
so $x \in \goodSetRegressionFunctionEstimatedWell$. Hence, taking  $\modDeltaClassConditional :=  4 \lceil \log_{4/3}(\numClassConditional{y}) \rceil \{\delta/\log(\numClassConditional{y})\}^2$ for $y\in \{0,1\}$ and applying Lemma \ref{lemma:uniformConcentrationMetricBallsSinglePoint} we have
\begin{align*}
\Prob\left\lbrace x \notin \goodSetRegressionFunctionEstimatedWell \right\rbrace &\leq \Prob\left\lbrace \eventCI{0}(x,\delta)\right\rbrace + \Prob\left\lbrace \eventCI{1}(x,\delta)\right\rbrace \leq \modDeltaClassConditional[0]+\modDeltaClassConditional[1] \leq \frac{21 \delta^2}{2},
\end{align*}
provided $\nMin\geq 2^4$. 
\end{proof}

\begin{proof}[Proof of Proposition \ref{prop:regressionFunctionHighProbBound}] Proposition \ref{prop:regressionFunctionHighProbBound} now follows from Lemma \ref{lemma:pointwiseRegressionFunctionHighProbBound} by Markov's lemma and Fubini's theorem as follows,
\begin{align*}
150\delta\,\Prob\left( \classConditionalDistribution{\labelProbability[\testTime]}\left\lbrace \metricSpace \setminus \goodSetRegressionFunctionEstimatedWell\right\rbrace > 150\delta \right)&\leq \E\left( \classConditionalDistribution{\labelProbability[\testTime]}\left\lbrace \metricSpace \setminus \goodSetRegressionFunctionEstimatedWell\right\rbrace\right) \\ &= \int_{\metricSpace} \Prob\left( x \notin \goodSetRegressionFunctionEstimatedWell\right) d\classConditionalDistribution{\labelProbability[\testTime]}(x)\leq 50 \delta^2.
\end{align*}
Dividing both sides by $150\delta$ yields the bound.
\end{proof}


\section{Estimating the label probabilities} \label{sec:appendixWithLemmasOnEstimatingLabelProbs}

In order to prove Proposition \ref{prop:marginalDistributionEstimatorBound} we first leverage Assumption \ref{assumption:smoothlyVaryingLabelProbabilities} to bound the bias of the estimators $\estimatorMarginalDistribution[q]{\testTime}$ for $q \in [\testTime]$. Given $\estimatorMarginalDistributionExtraParams=(q,\windowSize,\holderExponentWeightFunction,\holderConstantWeightFunction)$ we define
\begin{align*}
\estimatorMarginalDistributionBiasTerm[f,\estimatorMarginalDistributionExtraParams]{\testTime}:=\frac{\holderConstantWeightFunction}{\pOfAlpha!}\big| \classConditionalDistribution{1}(f)-\classConditionalDistribution{0}(f)\big| \hspace{1mm}\sum_{i=1}^q |\estimatorMarginalDistributionWeights|\bigg(\frac{i}{m}\bigg)^{\holderExponentWeightFunction}.
\end{align*}

\begin{lemma}\label{lemma:weightFunctionLowBias} Suppose that Assumptions \ref{assumption:labelShift} and \ref{assumption:smoothlyVaryingLabelProbabilities} hold with  $\windowSize \in [\numUnlabelled]$, $\holderExponentWeightFunction \in [0,\maximalHolderExponentWeightFunction]$, $\holderConstantWeightFunction \in [0,\infty)$ and $\weightFunction \in \holderFunctionClass(\holderExponentWeightFunction,\holderConstantWeightFunction)$. Then, given any function $f:\X^{\numNull+\numPositive}\times \X \rightarrow [0,1]$, and any $q \in [\windowSize]$ with $\pOfQ > \lfloor \holderExponentWeightFunction \rfloor$ we have 
\begin{align*}
\left|\sum_{i=1}^q \estimatorMarginalDistributionWeights \marginalDistribution[t-i](f)-\marginalDistribution[t](f)\right| \leq \estimatorMarginalDistributionBiasTerm[f,\estimatorMarginalDistributionExtraParams]{\testTime}.
\end{align*}

\end{lemma}

\begin{proof}[Proof of Lemma \ref{lemma:weightFunctionLowBias}] By Assumption \ref{assumption:smoothlyVaryingLabelProbabilities} there exists $\weightFunction \in \holderFunctionClass(\holderExponentWeightFunction,\holderConstantWeightFunction)$ such that $
\labelProbability=\weightFunction( {\numUnlabelled-\ell+1}/{\windowSize})$ for all $\ell \in \{\testTime-\windowSize,\ldots,\testTime\}$. Let's write $h:[0,1] \rightarrow [0,1]$ for the function defined by
\begin{align*}
h(u):= \classConditionalDistribution{0}(f)+\big\{   \classConditionalDistribution{1}(f)-\classConditionalDistribution{0}(f)\big\}\weightFunction(u), 
\end{align*}
for $u \in [0,1]$. Recall that
$\holderConstantSpecifiedFunction:=\holderConstantWeightFunction\big| \classConditionalDistribution{1}(f)-\classConditionalDistribution{0}(f)\big|/\{\pOfAlpha!\,\sqrt{2\holderExponentWeightFunction+1}\}$. It follows that $h \in \holderFunctionClass(\holderExponentWeightFunction,\holderConstantSpecifiedFunction 
\{\pOfAlpha!\,\sqrt{2\holderExponentWeightFunction+1}\} )$ and so $h$ is $\pOfAlpha$-times differentiable. Now suppose $\pOfAlpha>0$. Then, by Taylor's theorem for each $u \in [0,1]$,
\begin{align*}
h\left(u\right)&=\sum_{k=0}^{\pOfAlpha-1}\frac{h^{(k)}(0)}{k!}\cdot u^k + \int_0^u \frac{h^{(\pOfAlpha)}(\xi)}{(\pOfAlpha-1)!}\cdot \xi^{\pOfAlpha-1} d\xi.
\end{align*}
Thus, for each $u \in [0,1]$,
\begin{align*}
\bigg| h\left(u\right)-\sum_{k=0}^{\pOfAlpha}\frac{h^{(k)}(0)}{k!}\cdot u^k\bigg|
&=\bigg| \int_0^u \frac{h^{(\pOfAlpha)}(\xi)}{(\pOfAlpha-1)!}\cdot \xi^{\pOfAlpha-1} d\xi-\frac{h^{(\pOfAlpha)}(0)}{\pOfAlpha!}\cdot u^{\pOfAlpha}\bigg|\\
&=\frac{1}{({\pOfAlpha}-1)!}\int_0^u\big| h^{({\pOfAlpha})}(\xi) - h^{({\pOfAlpha})}(0)\big|\xi^{{\pOfAlpha}-1} d\xi\\
&\leq \frac{\overline{\holderConstantWeightFunction}\, u^{\holderExponentWeightFunction-{\pOfAlpha}}}{({\pOfAlpha}-1)!}\int_0^u \xi^{{\pOfAlpha}-1} d\xi=\holderConstantSpecifiedFunction \sqrt{2\holderExponentWeightFunction+1}\ \,u^{\holderExponentWeightFunction}.
\end{align*}
On the other hand, if $\pOfAlpha=0$, then for $u \in [0,1]$, the bound 
\begin{align*}
\bigg| h\left(u\right)-\sum_{k=0}^{\pOfAlpha}\frac{h^{(k)}(0)}{k!}\cdot u^k\bigg|\leq \holderConstantSpecifiedFunction \sqrt{2\holderExponentWeightFunction+1}\ \,u^{\holderExponentWeightFunction},
\end{align*} 
follows immediately from $h \in \holderFunctionClass(\holderExponentWeightFunction,\holderConstantSpecifiedFunction 
\{\pOfAlpha!\,\sqrt{2\holderExponentWeightFunction+1}\} )$.

The polynomials $(\legendrePolynomialShiftedOrthonormal)_{k=0}^{\pOfAlpha}$ are orthonormal with respect to the inner product $\left\langle f,g \right\rangle_1 := \int_0^1 f(z)g(z)dz$ (Lemma \ref{lemma:shiftedLegendrePolysOrthonormal}). Hence, the collection of polynomials of the form $z\mapsto \legendrePolynomialShiftedOrthonormal(\windowSize z/q)$ for $k \in [{\pOfAlpha}]$ are orthonormal with respect to the inner product $\left\langle a,b \right\rangle_{q/\windowSize} := (\windowSize/q)\int_0^{q/\windowSize} a(z)b(z)dz$, and as such form a basis of the $({\pOfAlpha}+1)$-dimensional vector space consisting of polynomials of degree at most $p$. Thus, we may choose $a=(a_j)_{j=1}^{\pOfQ+1} \in \R^{p+1}$ such that $a_j =0$ for all $j>\pOfAlpha+1$ and
\begin{align*}
\sum_{k=0}^{\pOfAlpha}\frac{h^{(k)}(0)}{k!}\cdot u^k&=\sum_{j=1}^{\pOfAlpha+1}a_j \legendrePolynomialShiftedOrthonormal[j-1]\left(\frac{\windowSize u}{q}\right)=\sum_{j=1}^{\pOfQ+1}a_j \legendrePolynomialShiftedOrthonormal[j-1]\left(\frac{\windowSize u}{q}\right),
\end{align*}
for all $u \in [0,1]$. Note also that for $i \in \{0\}\cup [m]$ we have
\begin{align*}
h\bigg(\frac{i}{m}\bigg)&= \classConditionalDistribution{0}(f)+\big\{   \classConditionalDistribution{1}(f)-\classConditionalDistribution{0}(f)\big\}\weightFunction\bigg(\frac{i}{m}\bigg)= \classConditionalDistribution{0}(f)+\big\{   \classConditionalDistribution{1}(f)-\classConditionalDistribution{0}(f)\big\}\labelProbability[t-i]=\marginalDistribution[t-i](f).
\end{align*}
Thus, with  $\weightMatrixByQNoArg \equiv \weightMatrixByQ$ we have 
\begin{align*}
\big| \marginalDistribution[t-i](f)- \left(\weightMatrixByQNoArg_{i,:}\right) a\big|&=\bigg| h\bigg(\frac{i}{m}\bigg)-\sum_{k=0}^{\pOfQ}a_j \legendrePolynomialShiftedOrthonormal[j-1]\bigg(\frac{i}{q}\bigg)\bigg|\\
& = \bigg| h\bigg(\frac{i}{m}\bigg)-\sum_{k=0}^{\pOfQ}\frac{h^{(k)}(0)}{k!}\cdot\bigg(\frac{i}{m}\bigg)^k\bigg|\leq \holderConstantSpecifiedFunction \sqrt{2\holderExponentWeightFunction+1}\ \,u^{\holderExponentWeightFunction} \,\bigg(\frac{i}{m}\bigg)^{\holderExponentWeightFunction},
\end{align*}
for all $i \in  [m]$ and  $\marginalDistribution[t](f)= \left(\weightMatrixByQNoArg_{0,:}\right) a$. Note also that the choice of $\pOfQ$ ensures that $\weightMatrixByQNoArg^\top \weightMatrixByQNoArg$ is invertible, so 
\begin{align*}
\sum_{i=1}^q  \estimatorMarginalDistributionWeights \weightMatrixByQNoArg_{i,:} &= \sum_{i=1}^q \weightMatrixByQNoArg_{0,:}\bigl( {\weightMatrixByQNoArg}^\top {\weightMatrixByQNoArg}\bigr)^+{{\weightMatrixByQNoArg_{i,:}}}^\top \weightMatrixByQNoArg_{i,:}\\
&= \weightMatrixByQNoArg_{0,:}\bigl( {\weightMatrixByQNoArg}^\top {\weightMatrixByQNoArg}\bigr)^{-1} \left(\sum_{i=1}^q {{\weightMatrixByQNoArg_{i,:}}}^\top \weightMatrixByQNoArg_{i,:}\right)\\
& = \weightMatrixByQNoArg_{0,:}\bigl( {\weightMatrixByQNoArg}^\top {\weightMatrixByQNoArg}\bigr)^{-1}\bigl( {\weightMatrixByQNoArg}^\top {\weightMatrixByQNoArg}\bigr)= \weightMatrixByQNoArg_{0,:}.
\end{align*}
Consequently, we have
\begin{align*}
\bigg| &\sum_{i=1}^q \estimatorMarginalDistributionWeights \marginalDistribution[t-i](f)-\marginalDistribution[t](f)\bigg| \\ & \leq \sum_{i=1}^q |\estimatorMarginalDistributionWeights| \left| \marginalDistribution[t-i](f)- \left(\weightMatrixByQNoArg_{i,:}\right) a\right|  +\left|\sum_{i=1}^q  \estimatorMarginalDistributionWeights \weightMatrixByQNoArg_{i,:} a- \weightMatrixByQNoArg_{0,:}a\right|+\left| \left(\weightMatrixByQNoArg_{0,:}\right) a-\marginalDistribution[t](f)\right| \\ &\leq \holderConstantSpecifiedFunction \sqrt{2\holderExponentWeightFunction+1} \, \hspace{1mm}\sum_{i=1}^q |\estimatorMarginalDistributionWeights|\bigg(\frac{i}{m}\bigg)^{\holderExponentWeightFunction}\\
& =\frac{\holderConstantWeightFunction}{\pOfAlpha!}\big| \classConditionalDistribution{1}(f)-\classConditionalDistribution{0}(f)\big| \hspace{1mm}\sum_{i=1}^q |\estimatorMarginalDistributionWeights|\bigg(\frac{i}{m}\bigg)^{\holderExponentWeightFunction} = \estimatorMarginalDistributionBiasTerm[f,\estimatorMarginalDistributionExtraParams]{\testTime}.
\end{align*}
\end{proof}

\begin{lemma}\label{lemma:biasBoundViaWeightBound}  Given $\estimatorMarginalDistributionExtraParams=(q,\windowSize,\holderExponentWeightFunction,\holderConstantWeightFunction)$ we have
\begin{align*}
\estimatorMarginalDistributionBiasTerm[f,\estimatorMarginalDistributionExtraParams]{\testTime}\leq \frac{\holderConstantWeightFunction\, \big| \classConditionalDistribution{1}(f)-\classConditionalDistribution{0}(f)\big| \hspace{1mm}\|\estimatorMarginalDistributionWeightsNoArg\|_2 (q+1)^{\holderExponentWeightFunction+1/2}}{\pOfAlpha!\sqrt{{2\holderExponentWeightFunction+1}}\,\windowSize^\holderExponentWeightFunction}.
\end{align*}
\end{lemma}
\begin{proof} Note that for any $\holderExponentWeightFunction \geq 0$ we have
\begin{align*}
\sum_{i=1}^qi^{2\holderExponentWeightFunction}&\leq \int_1^{q+1}z^{2\holderExponentWeightFunction}dz=\frac{(q+1)^{2\holderExponentWeightFunction+1}-1}{2\holderExponentWeightFunction+1}.
\end{align*}
Hence, the bound follows from the definition of $\estimatorMarginalDistributionBiasTerm[f,\estimatorMarginalDistributionExtraParams]{\testTime}$ combined with the Cauchy-Schwartz inequality.
\end{proof}

\begin{lemma}\label{lemma:weightFunctionLowVariance} Suppose $f:\X^{\numNull+\numPositive+1} \rightarrow [0,1]$ and take $q \in [\numUnlabelled]$. Then,
\begin{align*}
\Prob\left(\frac{\left|\estimatorMarginalDistribution[q]{\testTime}(f)-\sum_{i=1}^q \estimatorMarginalDistributionWeights \marginalDistribution[t-i](f)\right|}{\estimatorMarginalDistributionRootVarTerm[q]{\testTime}{\delta}} > 1 \bigg| \nullSample,\positiveSample \right) \leq \frac{2\delta}{\pi^2q^2}.
\end{align*}
\end{lemma}
\begin{proof} Note that $\estimatorMarginalDistribution[q]{\testTime}(f):=\sum_{i=1}^q \estimatorMarginalDistributionWeights f(X_{t-i})$. Hence, the result follows from Azuma's inequality \cite{azuma1967weighted} applied conditionally on the values of $\nullSample$, $\positiveSample$.
\end{proof}

To complete the proof of Proposition \ref{prop:marginalDistributionEstimatorBound} we require Lemma \ref{lemma:boundingEstimatorMarginalDistributionWeights}.

\begin{lemma}\label{lemma:boundingEstimatorMarginalDistributionWeights}\hspace{-4mm} Suppose $q \in \{ 8\maximalHolderExponentWeightFunction^2(\maximalHolderExponentWeightFunction+1)^2,\ldots,\testTime\}$. Then $\pOfQ =\maximalHolderExponentWeightFunction - 1$ 
 and $\|\estimatorMarginalDistributionWeightsNoArg\|_2 \leq  \maximalHolderExponentWeightFunction \sqrt{2/q}$.
\end{lemma}

In order to prove Lemma \ref{lemma:boundingEstimatorMarginalDistributionWeights} we leverage the following result on Legendre polynomials, a proof of which is given in Section \ref{sec:appendixLegendrePolys} for the convenience of the reader.

\begin{lemma}\label{lemma:UCovarianceMatrixMinEigenValue} Suppose $p\in \N_0$ and $q \in \N$ and let
$ \weightMatrixByPAndQNoArg$ denote the $q\times (p+1)$ matrix with entries $ \weightMatrixByPAndQNoArg_{i,j}:= \legendrePolynomialShiftedOrthonormal[j-1](i/q)$ for $(i,j) \in [q]\times [p+1]$. Then, the
eigenvalues of $\weightMatrixByPAndQNoArg^\top \weightMatrixByPAndQNoArg$ all differ from $q$ by at most $2p^2(p+1)(2p+1)$.
\end{lemma}

\begin{proof}[Proof of Lemma \ref{lemma:boundingEstimatorMarginalDistributionWeights}] First note that since $q \geq 4 \left(\maximalHolderExponentWeightFunction - 1\right)^{2} \left(2 \maximalHolderExponentWeightFunction - 1\right) \maximalHolderExponentWeightFunction$,
it follows from Lemma \ref{lemma:UCovarianceMatrixMinEigenValue} that the matrix $\weightMatrixByPAndQNoArg^\top \weightMatrixByPAndQNoArg$ has minimal eigenvalue $q/2$, where $\weightMatrixByPAndQNoArg \equiv \weightMatrixByPAndQNoArg(\left\lceil\maximalHolderExponentWeightFunction\right\rceil - 1,q)$. Thus, $\weightMatrixByPAndQNoArg$ is of rank $\left\lceil\maximalHolderExponentWeightFunction\right\rceil $ and $\pOfQ =\maximalHolderExponentWeightFunction - 1$. As such, $\weightMatrixByQNoArg \equiv \weightMatrixByQ = \weightMatrixByPAndQNoArg(\left\lceil\maximalHolderExponentWeightFunction\right\rceil - 1,q)$ and $\weightMatrixByQNoArg^\top \weightMatrixByQNoArg$ is invertible with minimal eigenvalue $q/2$. Thus, we have
\begin{align*}
\|\estimatorMarginalDistributionWeightsNoArg\|_2^2 &=\sum_{i=1}^q\estimatorMarginalDistributionWeights^2=\sum_{i=1}^q \weightMatrixByQNoArg_{0,:}\bigl( \weightMatrixByQNoArg^\top \weightMatrixByQNoArg\bigr)^{-1}\weightMatrixByQNoArg_{i,:}^\top \weightMatrixByQNoArg_{i,:} \bigl( \weightMatrixByQNoArg^\top \weightMatrixByQNoArg\bigr)^{-1} \weightMatrixByQNoArg_{0,:}^\top \\ & =  \weightMatrixByQNoArg_{0,:}\bigl( \weightMatrixByQNoArg^\top \weightMatrixByQNoArg\bigr)^{-1}\weightMatrixByQNoArg_{0,:}^\top \leq \frac{2}{q} \weightMatrixByQNoArg_{0,:} \weightMatrixByQNoArg_{0,:}^\top  = \frac{2}{q} \sum_{j=1}^{\pOfQ+1} \legendrePolynomialShiftedOrthonormal[j-1](0)^2  \leq \frac{2}{q} \sum_{j=1}^{\maximalHolderExponentWeightFunction } (2j-1) \leq \frac{2\maximalHolderExponentWeightFunction^2}{q},
\end{align*}
where we applied Lemma \ref{lemma:shiftedLegendrePolysOrthonormal}  for the penultimate inequality.
\end{proof}

\newcommand{\logTermForProofPropMarginalDistributionEstimatorBound}{\ell_{f,\estimatorMarginalDistributionExtraParams}(\windowSize)}

\newcommand{\qForProofPropMarginalDistributionEstimatorBound}{q_f}

We are now ready to complete the proof of Proof of Proposition \ref{prop:marginalDistributionEstimatorBound}.

\begin{proof}[Proof of Proposition \ref{prop:marginalDistributionEstimatorBound}] We shall assume throughout the proof that $ \estimatorMarginalDistributionErrorBound<1$ since otherwise the bound in Proposition \ref{prop:marginalDistributionEstimatorBound} holds almost surely. Recall that we let $\holderConstantSpecifiedFunction:=\holderConstantWeightFunction\big| \classConditionalDistribution{1}(f)-\classConditionalDistribution{0}(f)\big|/\{\pOfAlpha!\,\sqrt{2\holderExponentWeightFunction+1}\}$. Let's also define 
\begin{align*}
\estimatorMarginalDistributionErrorBoundForProof&:=\left\lbrace \holderConstantSpecifiedFunction \left( \frac{2\log\{(\pi\windowSize)^2/\delta\}}{\windowSize}\right)^{\holderExponentWeightFunction}\right\rbrace^{\frac{1}{2\holderExponentWeightFunction+1}}.
\end{align*}
Next, we let $\qSetByWindowSize$ denote the set conisting of integers $q \in \{8\maximalHolderExponentWeightFunction^2(\maximalHolderExponentWeightFunction+1)^2,\ldots,\windowSize\}$ such that 
\begin{align*}
q \leq \frac{2\log\{(\pi q)^2/\delta\}}{\estimatorMarginalDistributionErrorBoundForProof^2} -1.
\end{align*}
Note that since $ \estimatorMarginalDistributionErrorBound<1$ we have $8\maximalHolderExponentWeightFunction^2(\maximalHolderExponentWeightFunction+1)^2 \in \qSetByWindowSize\neq \emptyset$. We let $\qForProofPropMarginalDistributionEstimatorBound:=\max \qSetByWindowSize$. Note that since $q\mapsto (q+1)/\log\{(\pi q)^2/\delta\}$ is increasing on $\N$, we have $\qSetByWindowSize=\{8\maximalHolderExponentWeightFunction^2(\maximalHolderExponentWeightFunction+1)^2,\ldots,\qForProofPropMarginalDistributionEstimatorBound\}$. Now by Lemma \ref{lemma:biasBoundViaWeightBound}
for $q \in \qSetByWindowSize$ we have
\begin{align*}
\estimatorMarginalDistributionBiasTerm[f,\estimatorMarginalDistributionExtraParams]{\testTime}&\leq \frac{\holderConstantSpecifiedFunction  \hspace{1mm}\|\estimatorMarginalDistributionWeightsNoArg\|_2 (q+1)^{\holderExponentWeightFunction+1/2}}{\windowSize^\holderExponentWeightFunction}\\
&\leq \frac{\holderConstantSpecifiedFunction  \hspace{1mm}\|\estimatorMarginalDistributionWeightsNoArg\|_2 }{\windowSize^\holderExponentWeightFunction} \left(\frac{2\log\{(\pi q)^2/\delta\}}{\estimatorMarginalDistributionErrorBoundForProof^2}\right)^{\holderExponentWeightFunction+1/2}\\
&\leq \frac{\holderConstantSpecifiedFunction   \left({2\log\{(\pi\windowSize)^2/\delta\}}\right)^{\holderExponentWeightFunction} \estimatorMarginalDistributionRootVarTerm[q]{\testTime}{\delta}}{\windowSize^\holderExponentWeightFunction \estimatorMarginalDistributionErrorBoundForProof ^{2\holderExponentWeightFunction+1}} = \estimatorMarginalDistributionRootVarTerm[q]{\testTime}{\delta}.
\end{align*}
Let's introduce the event
\begin{align*}
\highProbEventEstimatorMarginalDistribution:=\bigcap_{q \in [\windowSize]}\left\lbrace \left|\estimatorMarginalDistribution[q]{\testTime}(f)-\sum_{i=1}^q \estimatorMarginalDistributionWeights \marginalDistribution[t-i](f)\right| \leq \estimatorMarginalDistributionRootVarTerm[q]{\testTime}{\delta}\right\rbrace.
\end{align*}
On the event $\highProbEventEstimatorMarginalDistribution$, it follows from Lemma \ref{lemma:weightFunctionLowBias} that
\begin{align}\label{eq:boundOnTheLowBiasRegion}
\left|\estimatorMarginalDistribution[q]{\testTime}(f)-\marginalDistribution[\testTime](f)\right| \leq \estimatorMarginalDistributionBiasTerm[f,\estimatorMarginalDistributionExtraParams]{\testTime}+\estimatorMarginalDistributionRootVarTerm[q]{\testTime}{\delta}\leq 2\estimatorMarginalDistributionRootVarTerm[q]{\testTime}{\delta},
\end{align}
for all $q \in \qSetByWindowSize$, and consequently, we have
\begin{align*}
\left|\estimatorMarginalDistribution[\qForProofPropMarginalDistributionEstimatorBound]{\testTime}(f)-\estimatorMarginalDistribution[q]{\testTime}(f)\right|\leq 2 \left\lbrace \estimatorMarginalDistributionRootVarTerm[\qForProofPropMarginalDistributionEstimatorBound]{\testTime}{\delta}+\estimatorMarginalDistributionRootVarTerm[q]{\testTime}{\delta}\right\rbrace,
\end{align*}
for all $q \in \qSetByWindowSize= \{8\maximalHolderExponentWeightFunction^2(\maximalHolderExponentWeightFunction+1)^2,\ldots,\qForProofPropMarginalDistributionEstimatorBound-1\}$. Hence, on $\highProbEventEstimatorMarginalDistribution$ we have $\lepskiQ \geq \qForProofPropMarginalDistributionEstimatorBound$, which in turn implies
\begin{align*}
\left|\estimatorMarginalDistribution[\lepskiQ]{\testTime}(f)-\estimatorMarginalDistribution[\qForProofPropMarginalDistributionEstimatorBound]{\testTime}(f)\right|\leq 2 \left\lbrace \estimatorMarginalDistributionRootVarTerm[\lepskiQ]{\testTime}{\delta}+\estimatorMarginalDistributionRootVarTerm[\qForProofPropMarginalDistributionEstimatorBound]{\testTime}{\delta}\right\rbrace.
\end{align*}
Thus, by combining with \eqref{eq:boundOnTheLowBiasRegion} and Lemma \ref{lemma:boundingEstimatorMarginalDistributionWeights} we have
\begin{align*}
\big|\estimatorMarginalDistributionLepski(f)-\marginalDistribution[\testTime](f)\big| &= \left|\estimatorMarginalDistribution[\lepskiQ]{\testTime}(f)-\marginalDistribution[\testTime](f)\right| \\
 &\leq  \left|\estimatorMarginalDistribution[\lepskiQ]{\testTime}(f)-\estimatorMarginalDistribution[\qForProofPropMarginalDistributionEstimatorBound]{\testTime}(f)\right| + \left|\estimatorMarginalDistribution[\qForProofPropMarginalDistributionEstimatorBound]{\testTime}(f)-\marginalDistribution[\testTime](f)\right|\\
 & \leq 2 \estimatorMarginalDistributionRootVarTerm[\lepskiQ]{\testTime}{\delta}+4\estimatorMarginalDistributionRootVarTerm[\qForProofPropMarginalDistributionEstimatorBound]{\testTime}{\delta}\\
 & \leq 4\maximalHolderExponentWeightFunction \sqrt{\frac{{\log\{(\pi \lepskiQ)^2/\delta\}}}{\lepskiQ}}+8\maximalHolderExponentWeightFunction \sqrt{\frac{\log\{(\pi \qForProofPropMarginalDistributionEstimatorBound)^2/\delta\}}{\qForProofPropMarginalDistributionEstimatorBound}} \\
 &\leq  12\maximalHolderExponentWeightFunction \sqrt{\frac{\log\{(\pi \qForProofPropMarginalDistributionEstimatorBound)^2/\delta\}}{\qForProofPropMarginalDistributionEstimatorBound}}\\
 & \leq  (12\maximalHolderExponentWeightFunction+1)  \max\left\lbrace \sqrt{\frac{\log\{(\pi \windowSize)^2/\delta\}}{\windowSize}} , \estimatorMarginalDistributionErrorBoundForProof  \right\rbrace <\estimatorMarginalDistributionErrorBound,
\end{align*}
provided that the event $\highProbEventEstimatorMarginalDistribution$ holds. Moreover, by Lemma \ref{lemma:weightFunctionLowVariance} we have $\Prob\{\highProbEventEstimatorMarginalDistribution| \nullSample,\positiveSample\}\geq 1-\delta/3$.
\end{proof}

The following elementary lemma will be useful in completing the proof of Corollary \ref{corr:labelProbabilityEstimatorBound}. Given $t \in \R$ we shall let $[t]_0^1:= (t \vee 0) \wedge 1$, so $[t]_0^1$ denotes the projection of $t$ onto $[0,1]$.

\begin{lemma}\label{lemma:ratioDifference} Given $a,\, \hat{a} \in \R$ and $b,\, \hat{b} \in \R \setminus\{0\}$,
\begin{align}\label{eq:firstBoundFromRatioDifference}
\left| \frac{\hat{a}}{\hat{b}}-\frac{a}{b}\right| \leq   \frac{|\hat{a}-a|+| {\hat{a}}/{\hat{b}}||\hat{b}-b|}{|b|}.
\end{align}
Moreover, if $a/b \in [0,1]$ then 
\begin{align}\label{eq:secondBoundFromRatioDifference}
\left| \left[ \frac{\hat{a}}{\hat{b}}\right]_0^1-\frac{a}{b}\right| \leq   \frac{|\hat{a}-a|+|\hat{b}-b|}{|b|}.
\end{align}
\end{lemma}

A proof of Lemma \ref{lemma:ratioDifference} is given in Section \ref{sec:appendixLegendrePolys}.

\begin{proof}[Proof of Corollary \ref{corr:labelProbabilityEstimatorBound}] Let's assume, without loss of generality, that  $\positiveDistribution(f)\neq \nullDistribution(f)$ and
\begin{align*}
\max_{y \in \{0,1\}} \left| \estimatorClassConditionalSecondSample{y}(f)-\classConditionalDistribution{y}(f) \right| \leq \frac{1}{3}|\positiveDistribution(f)-\nullDistribution(f)|,
\end{align*}
since otherwise the bound holds almost surely. As such, we have $\estimatorClassConditionalSecondSample{0}(f) \neq \estimatorClassConditionalSecondSample{1}(f)$. We shall introduce an event
\begin{align*}
\eventLabelProbEstimation:=&\left\lbrace \left| \estimatorMarginalDistributionLepski(f)-\marginalDistribution[\testTime](f)\right|\leq  \estimatorMarginalDistributionErrorBound \right\rbrace.
\end{align*}
By Proposition \ref{prop:marginalDistributionEstimatorBound} we have $\Prob\{\eventLabelProbEstimation|\nullSample,\positiveSample\} \leq \delta/3$.

Since $\labelProbability[\testTime] \in [0,1]$, it follows from Lemma \ref{lemma:ratioDifference} that on the event $\eventLabelProbEstimation$ we have
\begin{align*}
|\estimatorLabelProbabilityLepski-\labelProbability[\testTime]|&\leq \left| \frac{\estimatorMarginalDistributionLepski(f)-  \estimatorClassConditionalSecondSample{0}(f)}{\estimatorClassConditionalSecondSample{1}(f)-\estimatorClassConditionalSecondSample{0}(f)}-\frac{\marginalDistribution[\testTime](f)-  \nullDistribution(f)}{\positiveDistribution(f)-\nullDistribution(f)}\right|\\
& \leq \frac{\left|\estimatorMarginalDistributionLepski(f)-\marginalDistribution[\testTime](f)\right|+3\max_{y \in \{0,1\}} \left| \estimatorClassConditionalSecondSample{y}(f)-\classConditionalDistribution{y}(f) \right| }{|\positiveDistribution(f)-\nullDistribution(f)|}.
\end{align*}
Since $\Prob\{\eventLabelProbEstimation|\nullSample,\positiveSample\} \leq \delta/3$ this completes the proof.
\end{proof}

\begin{lemma}\label{lemma:boundingFunctionExpetationsUsingSecondHalvesOfLabelledData} Given any function $f:\X^{\numNull+\numPositive+1} \rightarrow [0,1]$ we have
\begin{align*}
\Prob\left\lbrace \max_{y \in \{0,1\}} 2\left( \estimatorClassConditionalSecondSample{y}(f)-\classConditionalDistribution{y}(f) \right)^2 \leq \epsilonByNDeltaBase{\nMin}{\delta/12} \right\rbrace \leq \frac{\delta}{3}.
\end{align*}
\end{lemma}
\begin{proof} The bound follows from four applications of Hoeffding's inequality \cite[Theorem 2.8]{boucheron2013concentration}, applied conditionally on the the values of $\nullCovariate_0,\ldots,\nullCovariate_{\numNull-1}$, $\positiveCovariate_0,\ldots,\positiveCovariate_{\numPositive-1}$, combined with a union bound.
\end{proof}


\section{Proof of Theorem \ref{thm:mainClassificationSingleTimeStep}}\label{sec:proofOfErrorBoundsOneStep}

Before completing the proof of Theorem \ref{thm:mainClassificationSingleTimeStep} we first prove Lemma \ref{lemma:classificationErrorSimpleFunction}.

\begin{proof}[Proof of Lemma \ref{lemma:classificationErrorSimpleFunction}] Conditioning on the data $\nullSample$, $\positiveSample$, $\unlabelledSample$ we have
\begin{align*}
\testError_{\testTime}(\dataDependentClassifier) &= \Prob\bigl\{\dataDependentClassifier(\unlabelledCovariate_{\testTime}) \neq \response_{\testTime}~|~\nullSample,\positiveSample,\unlabelledSample \bigr\}\\
&= \labelProbability[\testTime]\int (1-\dataDependentClassifier) d\positiveDistribution+(1-\labelProbability[\testTime])\int \dataDependentClassifier d\nullDistribution\\
& = \labelProbability[\testTime]+\int \dataDependentClassifier\,d\left\lbrace (1-\labelProbability[\testTime]) \nullDistribution-\labelProbability[\testTime] \positiveDistribution\right\rbrace\\
& = \labelProbability[\testTime]+2\int \dataDependentClassifier\,\left\lbrace (1-\labelProbability[\testTime]) (1-\regressionFunction)-\labelProbability[\testTime] \regressionFunction\right\rbrace\,d\classConditionalDistribution{1/2}\\
& = \labelProbability[\testTime]+2\int \dataDependentClassifier\, (1-\labelProbability[\testTime]-\regressionFunction)\,d\classConditionalDistribution{1/2}.
\end{align*}
Thus, with $\bayesClassifier[\testTime](x) = \one\left\lbrace \regressionFunction(x) +\labelProbability[\testTime]>1\right\rbrace$ we have
\begin{align*}
\testError_{\testTime}(\dataDependentClassifier)-\testError_{\testTime}(\bayesClassifier[\testTime])&=2 \int ( \dataDependentClassifier-\bayesClassifier[\testTime])(1-\labelProbability[\testTime]-\regressionFunction)d\classConditionalDistribution{1/2} =2 \int_{\{ \dataDependentClassifier\neq \bayesClassifier[\testTime]\}}|1-\labelProbability[\testTime]-\regressionFunction|\,d\classConditionalDistribution{1/2}.
\end{align*}

\end{proof}

We first give a sufficient condition for a lower bound on $(\positiveDistribution-\nullDistribution)(\hat{f}_\delta)$.

\begin{lemma}\label{lemma:keyTechnicalLemmaForProofOfMainClassificationSingleTimeStepFirstPartTotalVariation}
Suppose that Assumption \ref{assumption:labelShift} holds and take $\Delta_a,\, \Delta_b \in (0,\infty)$. Let $\eventForCentralTechnicalLemma$ denote the event
\begin{align*}
\eventForCentralTechnicalLemma^{\flat}:=&\left\lbrace\classConditionalDistribution{1/2}\left(\left\lbrace x \in \X : |\hat{\regressionFunction}(x)-\regressionFunction(x)|>\Delta_a \right\rbrace\right) \leq \Delta_b \right\rbrace.
\end{align*}
Then, on the event $\eventForCentralTechnicalLemma^\flat$ we have 
\begin{align*}
(\positiveDistribution-\nullDistribution)(\hat{f})\geq \totalVariation(\nullDistribution,\positiveDistribution)-4\Delta_a-2\Delta_b.
\end{align*}
\end{lemma}

\begin{proof}[Proof of Lemma \ref{lemma:keyTechnicalLemmaForProofOfMainClassificationSingleTimeStepFirstPartTotalVariation}] First define $f_\regressionFunction:\metricSpace\rightarrow \{0,1\}$ by $f_\regressionFunction(x):=\one\{\regressionFunction(x)\geq 1/2\}$ for $x \in \metricSpace$. Observe that 
\begin{align*}
\totalVariation(\nullDistribution,\positiveDistribution)&= \sup_{A \subseteq \metricSpace} \left| \positiveDistribution(A)-\nullDistribution(A)\right|\\
& = \sup_{A \subseteq \metricSpace} \left| 2\int_A \regressionFunction d\classConditionalDistribution{1/2}-2\int_A(1-\regressionFunction)d\classConditionalDistribution{1/2}\right|\\
& = 2\sup_{A \subseteq \metricSpace} \left| \int_A (2\regressionFunction -1)d\classConditionalDistribution{1/2}\right|\\
&= 2  \int_{\regressionFunction\geq \frac{1}{2}} (2\regressionFunction -1)d\classConditionalDistribution{1/2} = 2\int f_{\regressionFunction} (2\regressionFunction -1)d\classConditionalDistribution{1/2}.
\end{align*}
 Now let $\mathcal{G}_{1/2}:=\left\lbrace x \in \X : |\hat{\regressionFunction}(x)-\regressionFunction(x)|\leq \Delta_a \right\rbrace$. On the event $\eventForCentralTechnicalLemma^{\flat}$ we have
\begin{align*}
\frac{1}{2}\left\lbrace\totalVariation(\nullDistribution,\positiveDistribution)-(\positiveDistribution-\nullDistribution)(\hat{f})\right\rbrace
& = \frac{(\positiveDistribution-\nullDistribution)(f_{\regressionFunction}-\hat{f})}{2} \\
& =  \int (f_{\regressionFunction}-\hat{f}) (2\regressionFunction-1)d\classConditionalDistribution{1/2}\\
& =  \int_{\left\lbrace \left(\hat{\regressionFunction} - \frac{1}{2}\right) \left(\regressionFunction - \frac{1}{2}\right)<0\right\rbrace}  |2\regressionFunction-1|d\classConditionalDistribution{1/2}\\
& \leq  \int_{\mathcal{G}_{1/2} \cap \left\lbrace \left(\hat{\regressionFunction} - \frac{1}{2}\right) \left(\regressionFunction - \frac{1}{2}\right)<0\right\rbrace}  |2\regressionFunction-1|d\classConditionalDistribution{1/2}+\classConditionalDistribution{1/2}\left(\metricSpace\setminus \mathcal{G}_{1/2}\right)\\
& \leq  \int_{\left\lbrace \left|\regressionFunction - \frac{1}{2}\right|<\Delta_a \right\rbrace}  |2\regressionFunction-1|d\classConditionalDistribution{1/2}+\Delta_b\\
& \leq  2\Delta_a \classConditionalDistribution{1/2}\left(\left\lbrace x \in \metricSpace \,:\, \left|\regressionFunction(x) - \frac{1}{2}\right|<\Delta_a \right\rbrace\right)+\Delta_b\leq 2\Delta_a+\Delta_b.
\end{align*}
Doubling both sides completes the proof of the lemma.    
\end{proof}

Now we return to the proof of Proposition \ref{prop:keyTechnicalLemmaForProofOfMainClassificationSingleTimeStep}.

\begin{proof}[Proof of Proposition \ref{prop:keyTechnicalLemmaForProofOfMainClassificationSingleTimeStep}] Note that $\eventForCentralTechnicalLemma$ is measurable with respect to $\nullSample$, $\positiveSample$. Hence, by Corollary \ref{corr:labelProbabilityEstimatorBound} it suffices to show that whenever both
\begin{align}\label{eq:condtionFromCorrLabelProbabilityEstimatorBound}
\frac{|\estimatorLabelProbabilityLepskiUnspecified{\testTime}{\tilde{\delta}}{\hat{f}}-\labelProbability[\testTime]| |(\positiveDistribution-\nullDistribution)(\hat{f})|}{\estimatorMarginalDistributionErrorBound[\hat{f},\tilde{\delta}]+3\underset{y \in \{0,1\}}{\max} | (\estimatorClassConditionalSecondSample{y}-\classConditionalDistribution{y})(\hat{f}) | }\leq 1.
\end{align}
and $\eventForCentralTechnicalLemma$ hold we have
\begin{align}\label{eq:conclusionForLemmaKeyTechnicalLemmaForProofOfMainClassificationSingleTimeStep}
\testError_{\testTime}(\tilde{\dataDependentClassifier}_{\testTime,\tilde{\delta}}^{\hat{\regressionFunction}})-\testError_\testTime(\bayesClassifier[\testTime]) \leq 2\left\lbrace \tsybakovMarginFunction\left(\constantForMainClassificationSingleTimeStep \Delta_{\testTime}(\tilde{\delta})\right) +\Delta_b \right\rbrace.
\end{align}
Let's now assume both $\eventForCentralTechnicalLemma$ and \eqref{eq:condtionFromCorrLabelProbabilityEstimatorBound}. Note that since $\eventForCentralTechnicalLemma\subseteq \eventForCentralTechnicalLemma^{\flat}$ it follows from Lemma \ref{lemma:keyTechnicalLemmaForProofOfMainClassificationSingleTimeStepFirstPartTotalVariation} that $(\positiveDistribution-\nullDistribution)(\hat{f})\geq \totalVariation(\nullDistribution,\positiveDistribution)/2$. Note also that $\estimatorMarginalDistributionErrorBound[\hat{f},\tilde{\delta}] \leq 114 \maximalHolderExponentWeightFunction^2(\maximalHolderExponentWeightFunction+1)^2 \highProbBoundSingleTimeStepDeltaUnlabelledData[\tilde{\delta}] \totalVariation(\nullDistribution,\positiveDistribution)$. Hence, it follows from \eqref{eq:condtionFromCorrLabelProbabilityEstimatorBound} and $\eventForCentralTechnicalLemma$ that
\begin{align*}
|\estimatorLabelProbabilityLepskiUnspecified{\testTime}{\tilde{\delta}}{\hat{f}}-\labelProbability[\testTime]|& \leq \frac{\estimatorMarginalDistributionErrorBound[\hat{f},\tilde{\delta}]+3\underset{y \in \{0,1\}}{\max} | (\estimatorClassConditionalSecondSample{y}-\classConditionalDistribution{y})(\hat{f}) | }{ |(\positiveDistribution-\nullDistribution)(\hat{f})|}\\
& \leq \frac{114 \maximalHolderExponentWeightFunction^2(\maximalHolderExponentWeightFunction+1)^2 \highProbBoundSingleTimeStepDeltaUnlabelledData[\tilde{\delta}] \totalVariation(\nullDistribution,\positiveDistribution)+3\Delta_c
}{ |(\positiveDistribution-\nullDistribution)(\hat{f})|}\\
& \leq \frac{228 \maximalHolderExponentWeightFunction^2(\maximalHolderExponentWeightFunction+1)^2 \highProbBoundSingleTimeStepDeltaUnlabelledData[\tilde{\delta}] \totalVariation(\nullDistribution,\positiveDistribution)+6\Delta_c
}{\totalVariation(\nullDistribution,\positiveDistribution)}\\
& \leq 4 \left(57\maximalHolderExponentWeightFunction^2(\maximalHolderExponentWeightFunction+1)^2+3\right) \Delta_{\testTime}(\tilde{\delta}) =  \left( \constantForMainClassificationSingleTimeStep -16\right) \Delta_{\testTime}(\tilde{\delta}).
\end{align*}
Moreover, since $\Delta_a \leq 16 \Delta_{\testTime}(\tilde{\delta})$, this implies that for all $x \in \mathcal{G}:=\left\lbrace x \in \X : |\hat{\regressionFunction}(x)-\regressionFunction(x)|>\Delta_a \right\rbrace$ we have
\begin{align*}
\left|\left(\hat{\regressionFunction}(x)+\estimatorLabelProbabilityLepskiUnspecified{\testTime}{\tilde{\delta}}{\hat{f}}\right) - \left( \regressionFunction(x)+\labelProbability[\testTime]\right) \right| \leq \constantForMainClassificationSingleTimeStep \Delta_{\testTime}(\tilde{\delta}).
\end{align*}
In particular, for $x \in \mathcal{G}$ with $|1-\labelProbability[\testTime]-\regressionFunction| > \constantForMainClassificationSingleTimeStep \Delta_{\testTime}(\tilde{\delta})$ we have $\tilde{\dataDependentClassifier}_{\testTime,\tilde{\delta}}^{\hat{\regressionFunction}}(x)=\bayesClassifier[\testTime](x)$. The event $\eventForCentralTechnicalLemma$ also entails that $\classConditionalDistribution{1/2}(\metricSpace\setminus \mathcal{G}) \leq \Delta_b$. Thus, by Lemma \ref{lemma:classificationErrorSimpleFunction} and Assumption \ref{assumption:tsybakovMarginAssumption}  we have
\begin{align*}
\testError_{\testTime}(\tilde{\dataDependentClassifier}_{\testTime,\tilde{\delta}}^{\hat{\regressionFunction}})-\testError_\testTime(\bayesClassifier[\testTime])& =2 \int_{\{ \tilde{\dataDependentClassifier}_{\testTime,\tilde{\delta}}^{\hat{\regressionFunction}}\neq \bayesClassifier[\testTime]\}}|1-\labelProbability[\testTime]-\regressionFunction|\,d\classConditionalDistribution{1/2}\\
&\leq 2\classConditionalDistribution{1/2}\left(\metricSpace\setminus \mathcal{G}\right)+2\tsybakovMarginFunction\left( \constantForMainClassificationSingleTimeStep \Delta_{\testTime}(\tilde{\delta})\right)\leq 2 \left\lbrace \Delta_b+\tsybakovMarginFunction\left( \constantForMainClassificationSingleTimeStep \Delta_{\testTime}(\tilde{\delta})\right)\right\rbrace,
\end{align*}
provided both \eqref{eq:condtionFromCorrLabelProbabilityEstimatorBound} and $\eventForCentralTechnicalLemma$ hold. This establishes \eqref{eq:conclusionForLemmaKeyTechnicalLemmaForProofOfMainClassificationSingleTimeStep} and completes the proof of the lemma.
\end{proof}

Given $\delta \in (0,1)$, we let $\estimatedRegressionFunction$ denote the estimate for $\regressionFunction$ presented in Section \ref{sec:estRegressionFunction} and let $\hat{f}_\delta \in \setOfDataDependentClassifiers$ be the data-dependent classifier defined by $\hat{f}_\delta(x):=\one\{2\estimatedRegressionFunction(x)\geq 1\}$ for $x \in \metricSpace$.

\begin{lemma}\label{lemma:LabelledDataLemmaForApplyingKeyTechnicalLemmaForProofOfMainClassificationSingleTimeStep} Suppose that Assumption \ref{assumption:labelShift} holds and  Assumption \ref{assumption:smoothRegressionFunctions} holds with $\smoothnessExponentRegressionFunction \in (0,1]$, $\smoothnessConstantRegressionFunction \in [0,\infty)$. Suppose also that $\delta \in (0,1)$ and $\nMin$ satisfy $ \smoothnessConstantRegressionFunction\epsilonByNDeltaIteratedLogarithm{\nMin}{\delta}^\smoothnessExponentRegressionFunction \leq \left\lbrace  \totalVariation(\nullDistribution,\positiveDistribution)/168\right\rbrace^{2\smoothnessExponentRegressionFunction+1}$ and  $\sqrt{\epsilonByNDeltaIteratedLogarithm{\nMin}{\delta}} \leq  \totalVariation(\nullDistribution,\positiveDistribution)/168$ and $\delta \in (0,\totalVariation(\nullDistribution,\positiveDistribution)/(3\cdot 2^7)]$. Let $\Delta_{\testTime}(\tilde{\delta})$ and  $\eventForCentralTechnicalLemma$ be as in Proposition \ref{prop:keyTechnicalLemmaForProofOfMainClassificationSingleTimeStep}, with
\begin{align*}
\Delta_a& =14\left(\left\lbrace 
\smoothnessConstantRegressionFunction^{\frac{1}{\smoothnessExponentRegressionFunction}}\epsilonByNDeltaIteratedLogarithm{\nMin}{\delta} \right\rbrace^{\frac{\smoothnessExponentRegressionFunction}{2\smoothnessExponentRegressionFunction+1}}\vee\sqrt{\epsilonByNDeltaIteratedLogarithm{\nMin}{\delta}}\right)\\
\Delta_b&=2^5\delta\\
\Delta_c &= \sqrt{\epsilonByNDeltaIteratedLogarithm{\nMin}{\delta}}.
\end{align*}
Then, we have
\begin{align*}
&2{\Delta}_a+{\Delta}_b \leq \totalVariation(\nullDistribution,\positiveDistribution)/4\\
& \left(\frac{\Delta_a}{16}\right) \vee \left(\frac{\Delta_c}{2\totalVariation(\nullDistribution,\positiveDistribution)}\right) \leq \highProbBoundDeltaLabelledData\\
&\Prob\left\lbrace \eventForCentralTechnicalLemma({\Delta}_a,{\Delta}_b,{\Delta}_c)\right\rbrace \geq 1- \frac{2\delta}{3}.
\end{align*}
\end{lemma}
\begin{proof} Observe that the two upper bounds on $\epsilonByNDeltaIteratedLogarithm{\nMin}{\delta}$ entail
\begin{align*}
\Delta_a& =14\left(\left\lbrace 
\smoothnessConstantRegressionFunction^{\frac{1}{\smoothnessExponentRegressionFunction}}\epsilonByNDeltaIteratedLogarithm{\nMin}{\delta} \right\rbrace^{\frac{\smoothnessExponentRegressionFunction}{2\smoothnessExponentRegressionFunction+1}}\vee\sqrt{\epsilonByNDeltaIteratedLogarithm{\nMin}{\delta}}\right)\leq  \frac{\totalVariation(\nullDistribution,\positiveDistribution)}{12}.
\end{align*}
Since $\Delta_b =2^5 \delta \leq  \totalVariation(\nullDistribution,\positiveDistribution)/12$ we deduce that $2{\Delta}_a+{\Delta}_b \leq \totalVariation(\nullDistribution,\positiveDistribution)/4$. Our definitions also entail the bound
\begin{align*}
\left(\frac{\Delta_a}{16}\right) \vee \left(\frac{\Delta_c}{2\totalVariation(\nullDistribution,\positiveDistribution)}\right)< \highProbBoundDeltaLabelledData.
\end{align*}
By Proposition \ref{prop:regressionFunctionHighProbBound} we have
\begin{align*}
\Prob\left\lbrace\classConditionalDistribution{1/2}\left(\left\lbrace x \in \X : |\hat{\regressionFunction}(x)-\regressionFunction(x)|>\Delta_a \right\rbrace\right) \leq \Delta_b \right\rbrace \geq 1- \frac{\delta}{3}.
\end{align*}
By Hoeffding's inequality applied four times we also have
 \begin{align*}
 \Prob\left\lbrace \max_{y \in \{0,1\}} | (\estimatorClassConditionalSecondSample{y}-\classConditionalDistribution{y})(\hat{f}) | \leq \Delta_c \right\rbrace\geq 1- \frac{\delta}{3}.
 \end{align*}
Hence the conclusion of Lemma \ref{lemma:LabelledDataLemmaForApplyingKeyTechnicalLemmaForProofOfMainClassificationSingleTimeStep}  follows by a union bound.
\end{proof}


Next, we complete the proof of Theorem \ref{thm:mainClassificationSingleTimeStep}.

\begin{proof}[Proof of Theorem \ref{thm:mainClassificationSingleTimeStep}] The conclusion of Theorem \ref{thm:mainClassificationSingleTimeStep} now follows from Proposition \ref{prop:keyTechnicalLemmaForProofOfMainClassificationSingleTimeStep} combined with Lemma \ref{lemma:LabelledDataLemmaForApplyingKeyTechnicalLemmaForProofOfMainClassificationSingleTimeStep} since $\Delta_{\testTime}(\tilde{\delta})\leq  \highProbBoundSingleTimeStepDeltaUnlabelledData[\tilde{\delta}] \vee \highProbBoundDeltaLabelledData$.

\end{proof}

\section{Proof of Theorems \ref{thm:highProbBoundForTemporalSmoothnessWithJumps} and Corollary \ref{corr:totalVariationLabelProbBound}}\label{sec:proofOfAverageRegretBounds}

The key result in establishing Theorem \ref{thm:highProbBoundForTemporalSmoothnessWithJumps} is Proposition \ref{prop:technicalResultForProofOfAvgRegretBounds} which requires some further notation. For each $r \in [1,\infty)$ we define 
\begin{align*}
\convexFunctionForAverageRegretTechnicalProp\left(r\right):=\left\lbrace\frac{\averagePowerTransform{r}{\frac{\tsybakovMarginExponent}{2}}\,\epsilonByNDeltaBase{r}{\frac{\delta}{\maxTimeInterval}}}{ \totalVariation(\nullDistribution,\positiveDistribution)^2}\right\rbrace^{\frac{\tsybakovMarginExponent}{2}}+\left\lbrace\frac{\holderConstantWeightFunction^{\frac{1}{\holderExponentWeightFunction}}\epsilonByNDeltaBase{r}{\frac{\delta}{\maxTimeInterval}}}{ \totalVariation(\nullDistribution,\positiveDistribution)^2}\right\rbrace^{\frac{\holderExponentWeightFunction\tsybakovMarginExponent}{2\holderExponentWeightFunction+1}}\hspace{-1mm},
\end{align*}
for $\holderExponentWeightFunction>0$ and extending continuously to $\holderExponentWeightFunction=0$ by \begin{align*}
\convexFunctionForAverageRegretTechnicalProp(r):=\biggl\{\sqrt{{\averagePowerTransform{r}{{\tsybakovMarginExponent}/{2}}\epsilonByNDeltaBase{r}{{\delta}/{\maxTimeInterval}}}}/ \totalVariation(\nullDistribution,\positiveDistribution)\biggr\}^{\tsybakovMarginExponent}+ \holderConstantWeightFunction^{\tsybakovMarginExponent}.
\end{align*}

\begin{prop}\label{prop:technicalResultForProofOfAvgRegretBounds} Suppose we are in the setting of Theorem \ref{thm:highProbBoundForTemporalSmoothnessWithJumps}. Let  $(r_j)_{j=1}^\numJumps \in \N^{\numJumps}$ be a sequence and suppose Assumption \ref{assumption:smoothlyVaryingWithJumpsLabelProbabilities} holds with  respect to a strictly increasing sequence $(s_j)_{j=0}^{\numJumps} \in \N^{\numJumps+1}$ such that $s_j-s_{j-1}=r_j$ for all $j \in [\numJumps]$. Then, with probability at least $1-\delta$ we have
\begin{align*}
\regretOverTimeInterval_{\timeInterval}(\policy) \leq\, & 2^6\delta+2\tsybakovMarginConstant\{\constantForMainClassificationSingleTimeStep\highProbBoundDeltaLabelledData \}^{\tsybakovMarginExponent}+ \frac{\numJumps}{\timeHorizon}+\frac{2  \tsybakovMarginConstant(\sqrt{3}\constantForMainClassificationSingleTimeStep)^{\tsybakovMarginExponent}}{\timeHorizon}\sum_{j \in [\numJumps]} r_j\,   \convexFunctionForAverageRegretTechnicalProp(r_j).
\end{align*}
\end{prop}

\begin{proof}[Proof of Proposition \ref{prop:technicalResultForProofOfAvgRegretBounds}] Let $\constantForMainClassificationSingleTimeStep$ be and $\eventForCentralTechnicalLemma$ be as in Proposition \ref{prop:keyTechnicalLemmaForProofOfMainClassificationSingleTimeStep} with $\Delta_a$, $\Delta_b$, $\Delta_c$ as in Lemma \ref{lemma:LabelledDataLemmaForApplyingKeyTechnicalLemmaForProofOfMainClassificationSingleTimeStep}. By applying Lemma \ref{lemma:LabelledDataLemmaForApplyingKeyTechnicalLemmaForProofOfMainClassificationSingleTimeStep} we have 
\begin{align*}
&\Prob\{ \eventForCentralTechnicalLemma\}\geq 1- \frac{2\delta}{3},\hspace{1cm}
2{\Delta}_a+{\Delta}_b \leq \frac{\totalVariation(\nullDistribution,\positiveDistribution)}{4},\hspace{1cm}\left(\frac{\Delta_a}{16}\right) \vee \left(\frac{\Delta_c}{2\totalVariation(\nullDistribution,\positiveDistribution)}\right)< \highProbBoundDeltaLabelledData.
\end{align*}
Let  $(s_j)_{j=0}^{\numJumps} \in \N^{\numJumps+1}$ and $(\labelProbabilityFunctionWithJumps[j])_{ j =1}^\numJumps \in \holderFunctionClass(\holderExponentWeightFunction,\holderConstantWeightFunction)^{\numJumps}$ denote the sequences introduced in Assumption \ref{assumption:smoothlyVaryingWithJumpsLabelProbabilities}. For each $\testTimeWithinSequence \in \timeInterval$, let's choose $j(\testTimeWithinSequence):= \min\{ j \in [\numJumps]\,:\,s_j \geq \testTimeWithinSequence\}$, which is well-defined since $\testTimeWithinSequence \leq \maxTimeInterval = s_{\numJumps}$, let $\windowSize(\testTimeWithinSequence):=\testTimeWithinSequence-s_{j(\testTimeWithinSequence)-1}-1$ and $r_{j(\testTimeWithinSequence)}:=s_{j(\testTimeWithinSequence)}-s_{j(\testTimeWithinSequence)-1}-1$. Let's also define a function $g_{\testTimeWithinSequence}:[0,1] \rightarrow \R$ by 
\[g_{\testTimeWithinSequence}(z):=\labelProbabilityFunctionWithJumps[j(\testTimeWithinSequence)]\left\lbrace \left(\frac{\windowSize(\testTimeWithinSequence)}{r_{j(\testTimeWithinSequence)}}\right)(1-z)\right\rbrace,\] for all $z \in [0,1]$. Note that since $\labelProbabilityFunctionWithJumps[j(\testTimeWithinSequence)] \in \holderFunctionClass(\holderExponentWeightFunction,\holderConstantWeightFunction)$ it follows that $g_{\testTimeWithinSequence} \in \holderFunctionClass\left\lbrace \holderExponentWeightFunction,\holderConstantWeightFunction(\testTimeWithinSequence)\right\rbrace$ with $\holderConstantWeightFunction(\testTimeWithinSequence):=( \windowSize(\testTimeWithinSequence)/r_{j(\testTimeWithinSequence)})^{\holderExponentWeightFunction}$. Moreover, for each $\ell \in \{\testTimeWithinSequence-\windowSize(\testTimeWithinSequence),\ldots,\testTimeWithinSequence\}\subseteq \{s_{j(\testTimeWithinSequence)-1}+1,\ldots,s_{j(\testTimeWithinSequence)}\} $ we have
\begin{align*}
\labelProbability&=\labelProbabilityFunctionWithJumps[j(\testTimeWithinSequence)]\bigg(\frac{\ell-s_{j(\testTimeWithinSequence)-1}-1}{s_{j(\testTimeWithinSequence)}-s_{j(\testTimeWithinSequence)-1}-1}\bigg) =g_{\testTimeWithinSequence}\left(\frac{\testTimeWithinSequence-\ell}{\windowSize(\testTimeWithinSequence)}\right).
\end{align*}
Consequently, Assumption \ref{assumption:smoothlyVaryingLabelProbabilities} holds with $\testTimeWithinSequence$ in place of $\testTime$ and $\holderConstantWeightFunction(\testTimeWithinSequence)$ in place of $\holderConstantWeightFunction$. Assumption \ref{assumption:tsybakovMarginAssumptionSequential} ensures that Assumption \ref{assumption:tsybakovMarginAssumption}  holds with $\tsybakovMarginFunction=\tsybakovPolynomialMarginFunction$. Next, let
\begin{align*}
\highProbBoundSingleTimeStepDeltaWithinSequenceUnlabelledData&:= \frac{\epsilonByNDeltaLogarithm{\windowSize(\testTimeWithinSequence) }{\delta_t}^{\frac{1}{2}} }{\totalVariation(\nullDistribution,\positiveDistribution)}\vee \left\lbrace \frac{\holderConstantWeightFunction(\testTimeWithinSequence)^{\frac{1}{\holderExponentWeightFunction}}\epsilonByNDeltaLogarithm{\windowSize(\testTimeWithinSequence)}{\delta_t}}{\totalVariation(\nullDistribution,\positiveDistribution)^2}\right\rbrace^{\frac{\holderExponentWeightFunction}{2\holderExponentWeightFunction+1}},\\
\event^{\mathrm{seq}}_\testTimeWithinSequence&:=\left\lbrace\frac{ \testError_{\testTime}(\empiricalPolicySingleTimeStep)-\testError_{\testTimeWithinSequence}(\bayesClassifier[\testTimeWithinSequence])}{2\tsybakovPolynomialMarginFunction\left(\constantForMainClassificationSingleTimeStep\left\lbrace \highProbBoundDeltaLabelledData+\highProbBoundSingleTimeStepDeltaWithinSequenceUnlabelledData\right\rbrace\right) +2^6\delta }\leq 1 \right\rbrace.
\end{align*}
By Proposition \ref{prop:keyTechnicalLemmaForProofOfMainClassificationSingleTimeStep} we have $\Prob\{\event_{\testTimeWithinSequence}^{\mathrm{seq}}|\eventForCentralTechnicalLemma\}\geq 1-\delta_t/3$ for each $t \in \N$. Since  $\Prob\{ \eventForCentralTechnicalLemma\}\geq 1- {2\delta}/{3}$ we deduce that $\Prob\{ \bigcap_{\testTimeWithinSequence \in \timeInterval} \event_{\testTimeWithinSequence}^{\mathrm{seq}}\} \geq 1-\delta$. Note also that since $\delta < \sqrt{6/\pi^2}$ we have
\begin{align*}
\mathrm{L}(t):&=\logBar\{\windowSize(t)/\delta_t\}/\totalVariation(\nullDistribution,\positiveDistribution)^2 \leq \frac{\logBar\{\pi^2\maxTimeInterval^3/(6\delta)\}}{\totalVariation(\nullDistribution,\positiveDistribution)^2}\leq \frac{3\, \logBar(\maxTimeInterval/\delta)}{\totalVariation(\nullDistribution,\positiveDistribution)^2}=:\mathrm{L}_{\timeInterval}.
\end{align*}
Hence, with for each $j \in [\numJumps]$, if $\holderExponentWeightFunction>0$ we have
\begin{align*}
\sum_{t=s_{j-1}+2}^{s_j}\highProbBoundSingleTimeStepDeltaWithinSequenceUnlabelledData^{\tsybakovMarginExponent} &=\sum_{t=s_{j-1}+2}^{s_j}\Biggl(\left\lbrace\frac{ \mathrm{L}_t}{\windowSize(t)}\right\rbrace^{\frac{1}{2}}\vee \left\lbrace \frac{\holderConstantWeightFunction^{\frac{1}{\holderExponentWeightFunction}}\mathrm{L}_t}{ r_{j(\testTimeWithinSequence)}}\right\rbrace^{\frac{\holderExponentWeightFunction}{2\holderExponentWeightFunction+1}} \Biggr)^{\tsybakovMarginExponent}\\
 &\leq  \sum_{\windowSize=1}^{r_j}\Biggl(\left\lbrace\frac{\mathrm{L}_{\timeInterval}}{\windowSize}\right\rbrace^{\frac{1}{2}}\vee \left\lbrace \frac{\holderConstantWeightFunction^{\frac{1}{\holderExponentWeightFunction}}\mathrm{L}_{\timeInterval}}{ r_{j}}\right\rbrace^{\frac{\holderExponentWeightFunction}{2\holderExponentWeightFunction+1}} \Biggr)^{\tsybakovMarginExponent}\\
 &\leq r_j\left( \left\lbrace\frac{\mathrm{L}_{\timeInterval}\, \averagePowerTransform{r_j}{\frac{\tsybakovMarginExponent}{2}}}{r_j}\right\rbrace^{\frac{\tsybakovMarginExponent}{2}}+\left\lbrace\frac{\holderConstantWeightFunction^{\frac{1}{\holderExponentWeightFunction}}\mathrm{L}_{\timeInterval}}{ r_{j}}\right\rbrace^{\frac{\holderExponentWeightFunction\tsybakovMarginExponent}{2\holderExponentWeightFunction+1}} \right) \leq \sqrt{3}^{\tsybakovMarginExponent} r_j\, \convexFunctionForAverageRegretTechnicalProp(r_j),
\end{align*}
and similarly for $\holderExponentWeightFunction=0$ by replacing each occurance of $\bigl\{\holderConstantWeightFunction^{\frac{1}{\holderExponentWeightFunction}}\mathrm{L}_{\timeInterval}/ r_{j}\bigr\}^{\frac{\holderExponentWeightFunction\tsybakovMarginExponent}{2\holderExponentWeightFunction+1}}$ with $\holderConstantWeightFunction$. 

Now, suppose the event $ \bigcap_{\testTimeWithinSequence \in \timeInterval} \event_{\testTimeWithinSequence}^{\mathrm{seq}}$ holds, so that for each $j \in [\numJumps]$ we have
\begin{align*}
\sum_{t=s_{j-1}+1}^{s_j} \bigl\{  \testError_{\testTimeWithinSequence}(\policySingleTimeStep_{\testTimeWithinSequence})-  \testError_{\testTimeWithinSequence}(\bayesClassifier[\testTimeWithinSequence])\bigr\}&\leq \sum_{t=s_{j-1}+2}^{s_j} \bigl\{  \testError_{\testTimeWithinSequence}(\policySingleTimeStep_{\testTimeWithinSequence})-  \testError_{\testTimeWithinSequence}(\bayesClassifier[\testTimeWithinSequence])\bigr\}+1\\
& \leq \sum_{t=s_{j-1}+2}^{s_j} \left\lbrace 2\,\tsybakovMarginConstant\left(\constantForMainClassificationSingleTimeStep\left\lbrace \highProbBoundDeltaLabelledData \vee\highProbBoundSingleTimeStepDeltaWithinSequenceUnlabelledData\right\rbrace\right)^{\tsybakovMarginExponent} +2^6\delta  \right\rbrace+1\\
& \leq  2r_j\tsybakovMarginConstant\constantForMainClassificationSingleTimeStep^{\tsybakovMarginExponent}\left\lbrace\highProbBoundDeltaLabelledData +3^{\frac{\tsybakovMarginExponent}{2}} \convexFunctionForAverageRegretTechnicalProp(r_j) \right\rbrace+2^6r_j\delta  +1.
\end{align*}
Consequently, on the event $ \bigcap_{\testTimeWithinSequence \in \timeInterval} \event_{\testTimeWithinSequence}^{\mathrm{seq}}$ we have
\begin{align*}
\regretOverTimeInterval_{\timeInterval}(\policy)&=\frac{1}{\timeHorizon}\sum_{\testTimeWithinSequence \in \timeInterval} \bigl\{  \testError_{\testTimeWithinSequence}(\policySingleTimeStep_{\testTimeWithinSequence})-  \testError_{\testTimeWithinSequence}(\bayesClassifier[\testTimeWithinSequence])\bigr\}\\
&=\frac{1}{\timeHorizon} \sum_{j \in [\numJumps]}\sum_{t=s_{j-1}+1}^{s_j} \bigl\{  \testError_{\testTimeWithinSequence}(\policySingleTimeStep_{\testTimeWithinSequence})-  \testError_{\testTimeWithinSequence}(\bayesClassifier[\testTimeWithinSequence])\bigr\}\\
 & \leq  2^6\delta+2\tsybakovMarginConstant\{\constantForMainClassificationSingleTimeStep\highProbBoundDeltaLabelledData \}^{\tsybakovMarginExponent}+ \frac{\numJumps}{\timeHorizon}+\frac{2  \tsybakovMarginConstant(\sqrt{3}\constantForMainClassificationSingleTimeStep)^{\tsybakovMarginExponent}}{\timeHorizon}\sum_{j \in [\numJumps]} r_j\,   \convexFunctionForAverageRegretTechnicalProp(r_j).
\end{align*}
Since $\Prob\{ \bigcap_{\testTimeWithinSequence \in \timeInterval} \event_{\testTimeWithinSequence}^{\mathrm{seq}}\} \geq 1-\delta$ this completes the proof of the proposition.
\end{proof}

\begin{proof}[Proof of Theorem \ref{thm:highProbBoundForTemporalSmoothnessWithJumps}] Let  $(s_j)_{j=0}^{\numJumps} \in \N^{\numJumps+1}$  denote the sequence introduced in Assumption \ref{assumption:smoothlyVaryingWithJumpsLabelProbabilities} and define a sequence $(r_j)_{j=1}^\numJumps \in \N^{\numJumps}$ by $r_j:=s_j-s_{j-1}$ for $j \in [\numJumps]$. Hence, by Proposition \ref{prop:technicalResultForProofOfAvgRegretBounds}, 
with probability at least $1-\delta$, we have
\begin{align}\label{eq:defForProofThmHighProbBoundForTemporalSmoothnessWithJumps}
\regretOverTimeInterval_{\timeInterval}(\policy) &\leq  2^6\delta+2\tsybakovMarginConstant\{\constantForMainClassificationSingleTimeStep\highProbBoundDeltaLabelledData \}^{\tsybakovMarginExponent}+ \frac{\numJumps}{\timeHorizon}+\frac{2  \tsybakovMarginConstant(\sqrt{3}\constantForMainClassificationSingleTimeStep)^{\tsybakovMarginExponent}}{\timeHorizon}\sum_{j \in [\numJumps]} r_j\,   \convexFunctionForAverageRegretTechnicalProp(r_j). 
\end{align}
Moreover, the function $r \mapsto r\,\convexFunctionForAverageRegretTechnicalProp(r)$ may be expressed as the integral of a non-increasing function, so is concave. Hence, by Jensen's inequality,
\begin{align*}
\sum_{j \in [\numJumps]} r_j\,   \convexFunctionForAverageRegretTechnicalProp(r_j)&\leq \left( \sum_{j \in [\numJumps]}{r_j} \right) \,   \convexFunctionForAverageRegretTechnicalProp\left(\frac{1}{\numJumps}\sum_{j=1}^{\numJumps}r_j\right) = \timeHorizon \,\convexFunctionForAverageRegretTechnicalProp\left(\frac{\timeHorizon}{\numJumps}\right) \leq \timeHorizon\,\highProbabilityUpperBoundUnlabelledForAverageRegretBound\left(\frac{\timeHorizon}{\numJumps},\holderConstantWeightFunction\right)^{\tsybakovMarginExponent},
\end{align*}
since $\timeHorizon =|\timeInterval|=s_{\numJumps}-s_0=\sum_{j \in [\numJumps]} r_j$. Note also that $\left(\averagePowerTransform{r}{q}/r\right)^q\geq 1/r$ for all $q \in (0,\infty)$ and $r \in [1,\infty)$ so that $\numJumps/\timeHorizon \leq \highProbabilityUpperBoundUnlabelledForAverageRegretBound\left({\timeHorizon}/{\numJumps},\holderConstantWeightFunction\right)^{\tsybakovMarginExponent}$. Thus, with $\constantForSmoothnessWithJumpsAverageRegretBound = 2  (\sqrt{3}\constantForMainClassificationSingleTimeStep)^{\tsybakovMarginExponent}+1$ the bound in \eqref{eq:defForProofThmHighProbBoundForTemporalSmoothnessWithJumps} implies
\begin{align*}
\regretOverTimeInterval_{\timeInterval}(\policy)\leq \tsybakovMarginConstant\,\constantForSmoothnessWithJumpsAverageRegretBound \bigl\{\highProbBoundDeltaLabelledData\vee \highProbabilityUpperBoundUnlabelledForAverageRegretBound(\timeHorizon/\numJumps,\holderConstantWeightFunction)\bigr\}^{\tsybakovMarginExponent} +\constantForSmoothnessWithJumpsAverageRegretBound\delta.
\end{align*}
\end{proof}

Next we shall show that Corollary \ref{corr:totalVariationLabelProbBound} follows from Theorem \ref{thm:highProbBoundForTemporalSmoothnessWithJumps}.

\begin{proof}[Proof of Corollary \ref{corr:totalVariationLabelProbBound}] Let $\constantForTotalVariationLabelProbabilitiesBound:=2^{\frac{\tsybakovMarginExponent}{2}}\,\constantForSmoothnessWithJumpsAverageRegretBound$, where $\constantForSmoothnessWithJumpsAverageRegretBound$ is chosen as in Theorem \ref{thm:highProbBoundForTemporalSmoothnessWithJumps} with $\maximalHolderExponentWeightFunction=0$. In addition, we define $\maxTimeInterval=\max\timeInterval$,
\begin{align*}
\ratioProofOfcorrTotalVariationLabelProbBound&:=\left\lbrace \left(\frac{\averagePowerTransform{\timeHorizon}{{\tsybakovMarginExponent}/{2}} \logBar(\maxTimeInterval/\delta)}{\{ 2^{\holderExponentWeightFunction}\totalVariationLabelProbabilities\totalVariation(\nullDistribution,\positiveDistribution)\}^2 }\right)^{\frac{1}{2\holderExponentWeightFunction+1}}\timeHorizon^{\frac{2\holderExponentWeightFunction}{2\holderExponentWeightFunction+1}}\right\rbrace \wedge \left(\frac{\timeHorizon}{2}\right),
\end{align*}
and $\holderConstantProofOfcorrTotalVariationLabelProbBound:= \totalVariationLabelProbabilities \left(2 \ratioProofOfcorrTotalVariationLabelProbBound/\timeHorizon\right)^{\holderExponentWeightFunction}$. Let's also construct a sequence $(s_j)_{j=0}^{\infty} \in \N_0^{\N_0}$ recursively as follows. We let $s_0=\min \timeInterval-1$. Then, for $j\in \N$, if $s_{j-1}<\maxTimeInterval-1$ and $\sum_{\ell=s_{j-1}+1}^{\maxTimeInterval-1}\left|\labelProbability[\ell]-\labelProbability[\ell+1]\right|^{\frac{1}{\holderExponentWeightFunction}} >\holderConstantProofOfcorrTotalVariationLabelProbBound^{\frac{1}{\holderExponentWeightFunction}}$ we choose $s_j$ to be the minimal value of $t \in \{s_{j-1}+1,\ldots,\maxTimeInterval-1\}$ for which 
\begin{align*}
\sum_{\ell=s_{j-1}+1}^{t}\left|\labelProbability[\ell]-\labelProbability[\ell+1]\right|^{\frac{1}{\holderExponentWeightFunction}} > \holderConstantProofOfcorrTotalVariationLabelProbBound^{\frac{1}{\holderExponentWeightFunction}}.
\end{align*}
Otherwise we choose $s_j:=\maxTimeInterval$. We then let $\numJumps:=\min\left\lbrace j \in \N_0 : s_j=\timeHorizon\right\rbrace$. As such, the sequence $(s_j)_{j=0}^{\numJumps}$ is strictly increasing with $s_0=\min\timeInterval-1$, $s_{\numJumps}=\max \timeInterval$.  Given $j \in [\numJumps]$, we choose $\labelProbabilityFunctionWithJumps:[0,1]\rightarrow [0,1]$ to be the continuous piece-wise linear function such that 
\begin{align*}
\labelProbabilityFunctionWithJumps\bigg(\frac{\ell-s_{j-1}-1}{s_{j}-s_{j-1}-1}\bigg)=\labelProbability,
\end{align*}
for all $\ell \in \{s_{j-1}+1,\ldots,s_{j}\}$, and $\labelProbabilityFunctionWithJumps$ is linear on each interval of the form
\begin{align*}
\left[ \frac{\ell-s_{j-1}-1}{s_{j}-s_{j-1}-1},\frac{\ell-s_{j-1}}{s_{j}-s_{j-1}-1}\right],
\end{align*}
for $\ell \in \{s_{j-1}+1,\ldots,s_j-1\}$.  The construction of  $(s_j)_{j=0}^{\numJumps}$ also ensures that for any $j \in [\numJumps]$ and $\ell_0,\ell_1 \in \{s_{j-1}+1,\ldots,s_{j}\}$ we have
\begin{align*}
\left|\labelProbability[\ell_0]-\labelProbability[\ell_1]\right|^{\frac{1}{\holderExponentWeightFunction}}& \leq \left(\sum_{\ell=s_{j-1}+1}^{s_j-1}\left|\labelProbability[\ell]-\labelProbability[\ell+1]\right|\right)^{\frac{1}{\holderExponentWeightFunction}}\leq \sum_{\ell=s_{j-1}+1}^{s_j-1}\left|\labelProbability[\ell]-\labelProbability[\ell+1]\right|^{\frac{1}{\holderExponentWeightFunction}}\leq \holderConstantProofOfcorrTotalVariationLabelProbBound^{\frac{1}{\holderExponentWeightFunction}},
\end{align*}
where we applied Minkowski's inequality in the second line. Hence, we have $\labelProbabilityFunctionWithJumps\in \holderFunctionClass\left(0,\holderConstantProofOfcorrTotalVariationLabelProbBound\right)$. Thus, Assumption \ref{assumption:smoothlyVaryingWithJumpsLabelProbabilities} holds with H\"{o}lder exponent $0$ and H\"{o}lder constant $\holderConstantProofOfcorrTotalVariationLabelProbBound$. Thus, by Theorem \ref{thm:highProbBoundForTemporalSmoothnessWithJumps}, the following bound holds with probability at least $1-\delta$,
\begin{align*}
\regretOverTimeInterval_{\timeInterval}(\policy)&\leq \tsybakovMarginConstant\,\constantForSmoothnessWithJumpsAverageRegretBound \bigl\{\highProbBoundDeltaLabelledData\vee \highProbabilityUpperBoundUnlabelledForAverageRegretBoundZeroHolder(\timeHorizon/\numJumps,\holderConstantProofOfcorrTotalVariationLabelProbBound)\bigr\}^{\tsybakovMarginExponent} +\constantForSmoothnessWithJumpsAverageRegretBound\delta\\
&\leq \tsybakovMarginConstant\,\constantForTotalVariationLabelProbabilitiesBound \bigl\{\highProbBoundDeltaLabelledData\vee  \highProbabilityUpperBoundUnlabelledForAverageRegretBoundZeroHolder(\timeHorizon/\numJumps,\holderConstantProofOfcorrTotalVariationLabelProbBound)/\sqrt{2}\bigr\}^{\tsybakovMarginExponent} +\constantForTotalVariationLabelProbabilitiesBound\delta.
\end{align*}
Hence, to complete the proof of Corollary \ref{corr:totalVariationLabelProbBound} it suffices to show that 
\begin{align}\label{eq:toProveCorrTotalVarLabelProb}
\highProbabilityUpperBoundUnlabelledForAverageRegretBoundZeroHolder(\timeHorizon/\numJumps,\holderConstantProofOfcorrTotalVariationLabelProbBound) \leq \sqrt{2} \highProbabilityUpperBoundUnlabelledForAverageRegretBound(\timeHorizon,\totalVariationLabelProbabilities).
\end{align}

Next, we note that the construction of  $(s_j)_{j=0}^{\numJumps}$ ensures
\begin{align*}
\sum_{\ell=s_{j-1}+1}^{s_j}\left|\labelProbability[\ell]-\labelProbability[\ell+1]\right|^{\frac{1}{\holderExponentWeightFunction}}> \holderConstantProofOfcorrTotalVariationLabelProbBound^{\frac{1}{\holderExponentWeightFunction}}.
\end{align*}
for each $j \in [\numJumps-1]$. As such,
\begin{align*}
(\numJumps-1)\holderConstantProofOfcorrTotalVariationLabelProbBound^{\frac{1}{\holderExponentWeightFunction}} &\leq \sum_{j\in[\numJumps-1]}
\sum_{\ell=s_{j-1}+1}^{s_j}\left|\labelProbability[\ell]-\labelProbability[\ell+1]\right|^{\frac{1}{\holderExponentWeightFunction}}\leq \totalVariationLabelProbabilities^{\frac{1}{\holderExponentWeightFunction}},
\end{align*}
so that $\numJumps\leq \{\totalVariationLabelProbabilities/\holderConstantProofOfcorrTotalVariationLabelProbBound\}^{\frac{1}{\holderExponentWeightFunction}}+1 \leq 2\{\totalVariationLabelProbabilities/\holderConstantProofOfcorrTotalVariationLabelProbBound\}^{\frac{1}{\holderExponentWeightFunction}}$, since $\ratioProofOfcorrTotalVariationLabelProbBound \leq \timeHorizon/2$. Moreover, the function $r\mapsto \highProbabilityUpperBoundUnlabelledForAverageRegretBound(r)$ is non-increasing, so this implies
\begin{align*}
\highProbabilityUpperBoundUnlabelledForAverageRegretBoundZeroHolder\left(\frac{\timeHorizon}{\numJumps}\right) & \leq \highProbabilityUpperBoundUnlabelledForAverageRegretBoundZeroHolder\left( \frac{\holderConstantProofOfcorrTotalVariationLabelProbBound^{\frac{1}{\holderExponentWeightFunction}}\timeHorizon}{2\totalVariationLabelProbabilities^{\frac{1}{\holderExponentWeightFunction}}}\right)=\highProbabilityUpperBoundUnlabelledForAverageRegretBoundZeroHolder\left( \ratioProofOfcorrTotalVariationLabelProbBound\right)\\
& = \holderConstantProofOfcorrTotalVariationLabelProbBound \vee \frac{\sqrt{\averagePowerTransform{\ratioProofOfcorrTotalVariationLabelProbBound}{{\tsybakovMarginExponent}/{2}} \epsilonByNDeltaBase{\ratioProofOfcorrTotalVariationLabelProbBound}{\delta/\maxTimeInterval} }}{ \totalVariation(\nullDistribution,\positiveDistribution)}\\
& \leq \totalVariationLabelProbabilities\,\left(\frac{2 \ratioProofOfcorrTotalVariationLabelProbBound}{\timeHorizon}\right)^{\holderExponentWeightFunction} \vee \sqrt{\frac{ \timeHorizon \averagePowerTransform{\timeHorizon}{{\tsybakovMarginExponent}/{2}} \epsilonByNDeltaBase{\timeHorizon}{\delta/\maxTimeInterval}}{\ratioProofOfcorrTotalVariationLabelProbBound \totalVariation(\nullDistribution,\positiveDistribution)^2 }} \leq \sqrt{2} \highProbabilityUpperBoundUnlabelledForAverageRegretBound\left(\timeHorizon,\totalVariationLabelProbabilities\right),
\end{align*}
which establishes the claim \eqref{eq:toProveCorrTotalVarLabelProb} and completes the proof of Corollary \ref{corr:totalVariationLabelProbBound}. 
\end{proof}
\newpage

\bibliography{bibliography}       
\bibliographystyle{apalike}

\appendix

\section{Technical results}\label{sec:appendixLegendrePolys}

This section contains a proof of  Lemma \ref{lemma:UCovarianceMatrixMinEigenValue}, which was used in the proof of Lemma \ref{lemma:boundingEstimatorMarginalDistributionWeights}, and a proof of Lemma \ref{lemma:ratioDifference}, which was applied in the proof of Corollary \ref{corr:labelProbabilityEstimatorBound}.

\begin{lemma}\label{lemma:shiftedLegendrePolysOrthonormal} The polynomials $(\legendrePolynomialShiftedOrthonormal[k])_{k \in \N_0}$ are orthonormal with respect to the inner product $\left\langle f,g \right\rangle := \int_0^1 f(z)g(z)dz$. Moreover, for all $k \in \N_0$ and $z \in [0,1]$,
\begin{align*}
|\legendrePolynomialShiftedOrthonormal[k](z)| \leq \sqrt{2k+1} \hspace{5mm}\text{  and  }\hspace{5mm}
|\legendrePolynomialShiftedOrthonormal[k]'(z)| \leq 2k^2 \sqrt{2k+1}. 
\end{align*}
\end{lemma}
\begin{proof}[Proof of Lemma \ref{lemma:shiftedLegendrePolysOrthonormal}] For each $k \in \N\cup\{0\}$, lets write $\legendrePolynomialKaplan$ for the standard Legendre polynomial of degree $k$, defined by 
\begin{align*}
\legendrePolynomialKaplan(z):=\frac{1}{2^kk!}\frac{d^{k}}{dz^{k}}\{z^2-1\}^k,
\end{align*}
for $k \in \N$ and $\legendrePolynomialKaplan[0]\equiv 1$. Note that for all $k \in \N_0$ we have
\begin{align}\label{eq:kaplanNormalisedLegendreDef}
\legendrePolynomialShiftedOrthonormal(z) = \sqrt{2k+1} \legendrePolynomialKaplan(2z-1), 
\end{align}
for all $z \in [0,1]$. Hence, by \cite[Theorem 17(i)]{kaplan2003AdvancedCalculus} we have $\sup_{z \in [-1,1]}|\legendrePolynomialKaplan[k](z)|\leq 1$.  Thus, by the Markov brother's inequality \cite{shadrin2004twelve} we have $\sup_{z \in [-1,1]}|\legendrePolynomialKaplan[k]'(z)|\leq k^2$.  Consequently, we have $\sup_{z \in [0,1]}|\legendrePolynomialShiftedOrthonormal[k](z)| \leq \sqrt{2k+1}$ and $\sup_{z \in [0,1]}|\legendrePolynomialShiftedOrthonormal[k]'(z)| \leq 2k^2\sqrt{2k+1}$. Moreover, by \cite[Theorem 18]{kaplan2003AdvancedCalculus} we have
\begin{align*}
\int_{-1}^1 \legendrePolynomialKaplan[k_0](z)\legendrePolynomialKaplan[k_1](z)dz = \begin{cases} \frac{2}{2k_0+1} & \text{ if }k_0=k_1\\
0 & \text{ if }k_0 \neq k_1.\end{cases}
\end{align*}
for any $k_0,~k_1 \in \N_0$. Hence, by \eqref{eq:kaplanNormalisedLegendreDef} the polynomials $(\legendrePolynomialShiftedOrthonormal[k])_{k \in \N_0}$ are orthonormal with respect to the inner product $\left\langle f,g \right\rangle := \int_0^1 f(z)g(z)dz$.
\end{proof}

Next, we apply Lemma \ref{lemma:shiftedLegendrePolysOrthonormal} to prove Lemma \ref{lemma:UCovarianceMatrixMinEigenValue}.

\begin{proof}[Proof of Lemma \ref{lemma:UCovarianceMatrixMinEigenValue}] First note that given $k_0$, $k_1 \in \{0\}\cup [p]$ and $i \in [q]$ we have
\begin{align*}
&\left| \frac{1}{q} \cdot \legendrePolynomialShiftedOrthonormal[k_0]\left( \frac{i}{q}\right)\legendrePolynomialShiftedOrthonormal[k_1]\left( \frac{i}{q}\right)-\int_{\frac{i-1}{q}}^{\frac{i}{q}}\legendrePolynomialShiftedOrthonormal[k_0](z)\legendrePolynomialShiftedOrthonormal[k_1](z)dz\right|\\
& =  \left| \int_{\frac{i-1}{q}}^{\frac{i}{q}} \int_z^{\frac{i}{q}} \frac{d}{dt}\left\lbrace \legendrePolynomialShiftedOrthonormal[k_0](t)\legendrePolynomialShiftedOrthonormal[k_1](t)\right\rbrace dt dz\right|\\
& \leq  \int_{\frac{i-1}{q}}^{\frac{i}{q}} \int_z^{\frac{i}{q}} \left|\frac{d\legendrePolynomialShiftedOrthonormal[k_0](t)}{dt}\legendrePolynomialShiftedOrthonormal[k_1](t)+ \legendrePolynomialShiftedOrthonormal[k_0](t)\frac{d\legendrePolynomialShiftedOrthonormal[k_1](t)}{dt} \right| dt dz\\
& \leq  4p^2(2p+1)\int_{\frac{i-1}{q}}^{\frac{i}{q}} \int_z^{\frac{i}{q}} dtdz=\frac{2p^2(2p+1)}{q^2},
\end{align*}
where we applied the bounds on $\legendrePolynomialShiftedOrthonormal$ and $\legendrePolynomialShiftedOrthonormal'$ from Lemma \ref{lemma:shiftedLegendrePolysOrthonormal} to obtain the final inequality. Hence,  $k_0$, $k_1 \in \{0\}\cup [p]$ and $i \in [q]$ the entry in the $(k_0+1)$-th row and $(k_1+1)$-th column of the matrix matrix $({\tilde{U}^{(p,q)}})^\top ({\tilde{U}^{(p,q)}})$ satisfies
\begin{align*}
\left|(\weightMatrixByPAndQNoArg_{:,k_0+1})^\top (\weightMatrixByPAndQNoArg_{:,k_1+1})-q\one\{k_0=k_1\}\right|
&=\left|\sum_{i=1}^q\legendrePolynomialShiftedOrthonormal[k_0]\left( \frac{i}{q}\right)\legendrePolynomialShiftedOrthonormal[k_1]\left( \frac{i}{q}\right) - q \int_0^1 \legendrePolynomialShiftedOrthonormal[k_0](z) \legendrePolynomialShiftedOrthonormal[k_1](z)dz\right| \\ &\leq 2p^2(2p+1).
\end{align*}
Hence, letting $\mathrm{I}_{p+1}$ denote the $(p+1)\times (p+1)$ identity matrix, we have
\begin{align*}
\left\|\weightMatrixByPAndQNoArg^\top \weightMatrixByPAndQNoArg -qI_{p+1}\right\|_2
& \leq (p+1)\left\|\mathrm{vec}\left\lbrace \weightMatrixByPAndQNoArg^\top \weightMatrixByPAndQNoArg -q\mathrm{I}_{p+1}\right\rbrace \right\|_{\infty}\leq 2p^2(p+1)(2p+1).
\end{align*}
The conclusion of the lemma follows.
\end{proof}

We close this section with a proof of Lemma \ref{lemma:ratioDifference}.

\begin{proof}[Proof of Lemma \ref{lemma:ratioDifference}] To prove the first bound we note that
\begin{align*}
 \frac{\hat{a}}{\hat{b}}-\frac{a}{b}=\frac{\hat{a}-a+(\hat{a}/ \hat{b})(b-\hat{b})}{b},
\end{align*}
and apply the triangle inequality.

To prove \eqref{eq:secondBoundFromRatioDifference} we consider four cases. First, if $\hat{a}/\hat{b} \in [0,1]$ then $[\hat{a}/\hat{b}]_0^1=\hat{a}/\hat{b}$, so  \eqref{eq:firstBoundFromRatioDifference} implies \eqref{eq:secondBoundFromRatioDifference}. Second, if $\hat{a}/\hat{b}> 1$ and $\hat{b}/b \geq 0$ we have
\begin{align*}
\left| \left[ \frac{\hat{a}}{\hat{b}}\right]_0^1-\frac{a}{b}\right|&=1-\frac{a}{b}\leq \frac{\hat{a}-a+b-\hat{b}}{b},
\end{align*}
so \eqref{eq:firstBoundFromRatioDifference} holds. Thirdly, if $\hat{b}/b<0$ then
\begin{align*}
\left| \left[ \frac{\hat{a}}{\hat{b}}\right]_0^1-\frac{a}{b}\right|&\leq 1 <\frac{b-\hat{b}}{b}   \leq \frac{|\hat{a}-a|+|\hat{b}-b|}{|b|}. 
\end{align*}
Finally, if $\hat{a}/\hat{b}<0$ and $\hat{b}/b\geq 0$ then $\hat{a}/b<0$ so
\begin{align*}
\left| \left[ \frac{\hat{a}}{\hat{b}}\right]_0^1-\frac{a}{b}\right|&=\frac{a}{b}\leq \frac{a-\hat{a}}{b}  \leq \frac{|\hat{a}-a|+|\hat{b}-b|}{|b|}.
\end{align*}

\end{proof}

\end{document}